\documentclass[11pt,           
english                       
]{article}

\usepackage[english,strings]{babel}     
\usepackage{amsmath}           
\usepackage[utf8]{inputenc}    
\usepackage[T1]{fontenc}       
\usepackage{longtable}         
\usepackage{exscale}           
\usepackage[final]{graphicx}   
\usepackage[sort]{cite}        
\usepackage{array}             
\usepackage{fancyhdr}          
\usepackage[a4paper]{geometry} 
\usepackage[multiuser]{fixme}  
\usepackage{xspace}            
\usepackage{nchairx} 
\usepackage[expansion=false    
           ]{microtype}        
\usepackage[nottoc]{tocbibind} 
\usepackage[
backref=page,                  
final=true,                    
plainpages=false,              
pdfpagelabels=true,            
pdfencoding=auto,              
unicode=true,                  
hypertexnames=true,            
naturalnames=true              
]{hyperref}                    

\newcommand{\AuthorEmailOne}{\texttt{michael.heins@mathematik.uni-wuerzburg.de}}
\newcommand{\AuthorEmailTwo}{\texttt{roth@mathematik.uni-wuerzburg.de}}
\newcommand{\AuthorEmailThree}{\texttt{stefan.waldmann@mathematik.uni-wuerzburg.de}}

\newcommand{\AuthorOne}{\textbf{Michael Heins}}
\newcommand{\AuthorTwo}{\textbf{Oliver Roth}}
\newcommand{\AuthorThree}{\textbf{Stefan Waldmann}}

\author{\AuthorOne\thanks{\AuthorEmailOne},
  \addtocounter{footnote}{2}%
  \;
  \AuthorTwo\thanks{\AuthorEmailTwo},
  \;
  \textbf{and}
  \;
  \addtocounter{footnote}{2}%
  \AuthorThree\thanks{\AuthorEmailThree}\\[0.5cm]
  Julius Maximilian University of Würzburg \\
  Institute of Mathematics \\
  Würzburg, Germany
}

\newcommand{\liegroup}[1]{\operatorname{#1}}

\newcommand{\stdrep}{\varrho_{\std}}

\newcommand{\krep}{\varrho_\kappa}

\newcommand{\weylrep}{\varrho_{\Weyl}}

\newcommand{\starstd}{\star_\std}

\newcommand{\stark}{\star_\kappa}

\newcommand{\starweyl}{\star_\Weyl}

\newcommand{\staralg}{\star_\mathfrak{g}}

\newcommand{\SymR}[1][{R'}]{\Sym^{\bullet}_{\! #1}}

\newcommand{\SymRC}[1][{R'}]{\hat{\Sym}_{\! #1}^{\bullet}}

\newcommand{\Entire}[1][{\! R}]{\mathscr{E}_{#1}}

\newcommand{\EntireC}[1][{\! R}]{\hat{\mathscr{E}}_{#1}}

\newcommand{\Orbit}[1]{\langle #1 \rangle}

\newcommand{\PolR}[1][{R, R'}]{\Pol^{\bullet}_{#1}}

\newcommand{\PolRC}[1][{R, R'}]{\widehat{\Pol}^{\bullet}_{#1}}

\newcommand{\Sh}{\mathrm{Sh}}

\newcommand{\eps}{\varepsilon}

\newcommand{\Taylor}{\operatorname{T}}

\newcommand{\Maj}{\operatorname{F}}

\renewcommand{\Holomorphic}{\mathcal{H}}

\title{Convergent Star Products on Cotangent Bundles of Lie Groups}

\date{July 2021}

\begin{document}

\selectlanguage{english}

%
%

\maketitle

%
%

\begin{abstract}
    For a connected real Lie group $G$ we consider the canonical
    standard-ordered star product arising from the canonical global
    symbol calculus based on the half-commutator connection of
    $G$. This star product trivially converges on polynomial functions
    on $T^*G$ thanks to its homogeneity. We define a nuclear Fréchet
    algebra of certain analytic functions on $T^*G$, for which the
    standard-ordered star product is shown to be a well-defined
    continuous multiplication, depending holomorphically on the
    deformation parameter $\hbar$. This nuclear Fréchet algebra is
    realized as the completed (projective) tensor product of a nuclear
    Fréchet algebra of entire functions on $G$ with an appropriate
    nuclear Fréchet algebra of functions on $\mathfrak{g}^*$. The
    passage to the Weyl-ordered star product, i.e. the Gutt star
    product on $T^*G$, is shown to preserve this function space,
    yielding the continuity of the Gutt star product with holomorphic
    dependence on $\hbar$.
\end{abstract}

%
%

\tableofcontents
\newpage

%
%

\section{Introduction}
\label{sec:Introduction}%

Formal deformation quantization as introduced in
\cite{bayen.et.al:1978a} is one of the very successful quantization
schemes for Hamiltonian mechanical systems. The basic idea is to
deform the commutative algebra of smooth functions $\Cinfty(M)$ on a
Poisson manifold $M$ into a noncommutative algebra
$\Cinfty(M)\formal{\hbar}$ by introducing a formal star product
$\star$: this is an associative product bilinear over the formal power
series in $\hbar$ such that the zeroth order in $\hbar$ is the
undeformed pointwise product of functions and the first order
commutator equals the Poisson bracket. Additional requirements are
that each order of $\star$ consists of bidifferential operators on
$M$.

The existence of such formal star products was first shown on
symplectic manifolds \cite{dewilde.lecomte:1983b, fedosov:1994a} and
then by Kontsevich for the general case of Poisson manifolds
\cite{kontsevich:2003a}. While these results provide spectacular
successes with many further developments and applications, for honest
physical applications one has to overcome the formal power series
formulation: the deformation parameter $\hbar$ is to be interpreted as
Planck's constant. Thus one is interested in \emph{strict} versions of
deformation quantization.

One scenario to obtain reasonable definitions and results for
strictness is to use $C^*$-algebraic deformations instead of formal
deformations. This has been introduced by Rieffel, see in particular
\cite{rieffel:1989a, rieffel:1993a}, and used by many others in the
sequel, see e.g. \cite{landsman:1998a, bieliavsky.gayral:2015a,
  bieliavsky.detournay.spindel:2009a, bieliavsky.massar:2001b,
  bieliavsky:2002a}. The basic ingredient is to use (oscillatory)
integral formulas for the star product, which then admit good enough
estimates to arrive at constructions of $C^*$-norms. While giving
strong results, the main difficulty with these approaches is that
unfortunately there is no general construction of star products via
oscillatory integrals available.

Thus a different approach was proposed, namely to use the formal star
products and investigate their convergence directly. It turns out that
in various classes of examples the following strategy is successful:
first one needs to understand the example well enough to find a small
subalgebra of functions for which the star product converges for
some trivial reasons. In the examples considered so far, the star
products simply terminate after finitely many terms on e.g. polynomial
functions on a vector space. Here no general results are available and
one is restricted to classes of examples. In a second step one then
tries to establish a locally convex topology for the small subalgebra
in such a way that the star product becomes continuous. Again, also in
this step no general results are available, but examples show promising
cases. Having succeeded, a completion of the small subalgebra then
gives a hopefully large and interesting locally convex algebra,
typically a Fréchet algebra, which then can be investigated further.

In finite dimensions this program might not seem more promising than
the previous ones, as it also lacks general existence
theorems. However, different types of examples can be covered,
yielding e.g. analogs of unbounded operator algebras. Moreover,
infinite-dimensional examples are very well possible, where oscillatory
integrals definitely are no longer available. Thus this approach can
be seen as complementing the previous strict deformation quantizations
by new and different examples. A detailed overview on these ideas can
be found in the review \cite{waldmann:2019a}, the original results are
in \cite{esposito.schmitt.waldmann:2019a, schmitt.schoetz:2019a:pre,
  schmitt2019a:pre, kraus.roth.schoetz.waldmann:2019a,
  schoetz.waldmann:2018a, esposito.stapor.waldmann:2017a,
  beiser.waldmann:2014a, waldmann:2014a}.

Of indisputable interest for geometric mechanics are the cotangent
bundles with their canonical Poisson structure. Their quantization is
known to be strongly related to various symbol calculi for
pseudo-differential operators. In fact, the asymptotic expansion of
the corresponding integrals yield star products when interpreted
correctly. Important for us is the other direction: one can directly
construct (global) star products on cotangent bundles $T^*Q$ out of a
covariant derivative on the configuration space $Q$, see e.g.
\cite{bordemann.neumaier.pflaum.waldmann:2003a, fedosov:2001a,
  pflaum:2000a, pflaum:1998b, pflaum:1998c,
  bordemann.neumaier.waldmann:1998a,
  bordemann.neumaier.waldmann:1999a}. One of their crucial features
is the homogeneity with respect to the Euler vector field, which causes
the functions $\Pol(T^*Q)$ polynomial in the fibers to form a
subalgebra, on which the star product trivially converges. Thus we
have found a good starting point for the above program.

A particular case of cotangent bundles is obtained for a Lie group $G$
as configuration space. This highly symmetric situation admits a
distinguished covariant derivative, the half-commutator connection,
which is entirely Lie-theoretic. The corresponding (standard-ordered)
star product $\starstd$ has been introduced already in
\cite{gutt:1983a} and was further investigated in
\cite{bordemann.neumaier.waldmann:1998a}. Using a left-invariant
volume form on $G$ one then can pass to a Weyl ordered star product,
as well. The left-invariant polynomial functions
$\Pol^\bullet(T^*G)^G \cong \Sym^\bullet(\liealg{g})$ are in linear
bijection to the symmetric algebra over the Lie algebra. The star
product $\starstd$ restricts and yields the Gutt star product on
$\Sym^\bullet(\liealg{g})$, thereby quantizing the linear Poisson
structure on $\liealg{g}^*$. For this star product, the above
convergence program has been carried through in
\cite{esposito.stapor.waldmann:2017a} by establishing a nuclear
locally convex topology on $\Sym^\bullet_{R'}(\liealg{g})$ such that
$\starstd$ becomes continuous. Here $R' \ge 1$ is a parameter.  The
completion is explicitly given by certain real-analytic functions on
the vector space $\liealg{g}^*$ with controlled growth at infinity and
becomes largest for the limiting case $R' = 1$.

Using the trivialization $T^*G \cong G \times \liealg{g}^*$ we arrive
at the first main result of this paper: We define a subspace
$\Entire(G)$ of real-analytic functions on a connected Lie group $G$
together with a suitable nuclear Fréchet topology, depending on a
parameter $R \ge 0$ in such a way that
$\Entire(G) \tensor \Sym^\bullet_{R'}(\liealg{g})$ becomes a
subalgebra of $\Pol(T^*G)$ for which the star product $\starstd$ is
continuous. Here the tensor product is equipped with the projective
topology. The completion $\widehat{\Pol}_{R,R'}(T^*G)$ is a nuclear Fréchet
algebra with largest completion for $R = 0$ and $R' = 1$. The
assumption to have a connected Lie group is convenient as then
real-analytic functions are determined by their Taylor expansion
at the unit element.

While the precise size of $\Entire(G)$ and $\Pol_{R, R'}(T^*G)$ is not
easy to grasp, the representative functions on $G$ always belong to
$\Entire(G)$ as soon as $0 \le R < 1$, thus guaranteeing a nontrivial
algebra of functions. Moreover, the continuity properties of
$\starstd$ immediately imply the continuity of the standard-ordered
quantization. This results in a symbol calculus for differential
operators on $G$ with coefficient functions in $\Entire(G)$ acting
continuously on $\Entire(G)$.

The second result is that the star product of two functions in the
completion $\Pol_{R, R'}(T^*G)$ depends holomorphically on
$\hbar \in \field{C}$. This way, the star product becomes a convergent
series in $\hbar$ as wanted.  As a consequence, also the
standard-ordered quantization is holomorphic in $\hbar$ and yields not
only differential operators but certain pseudo-differential operators
for which the composition is holomorphic in $\hbar$.

Finally, the passage from the standard-ordered star product to the
physically more appealing Weyl-ordered star product is compatible with
the above topologies: the equivalence transformation preserves
$\Pol_{R, R'}(T^*G)$, is continuous, and depends holomorphically on
$\hbar$ itself. Thus the Weyl-ordered star product inherits all the
nice properties of $\starstd$ with the additional feature that the
complex conjugation is now a $^*$-involution and the corresponding
Weyl-ordered quantization is a $^*$-representation.

While these results yield another large class of examples for the
aforementioned program to construct convergent star products, we also
mention the following list of further questions and possible
continuations:
\begin{itemize}
\item Having a $^*$-algebra we can ask for its normalized positive
    functionals, i.e. its states. Here one first question is whether
    each classical state, i.e. a positive Borel measure on $T^*G$, can
    be deformed into a state of the Weyl-ordered star product algebra?
    Ideally, this can be accomplished in a way with good dependence on
    $\hbar$. Note that, unlike in
    \cite{kraus.roth.schoetz.waldmann:2019a, beiser.waldmann:2014a},
    such a deformation is expected to be necessary. In the case of
    formal star products this is known to be possible in general
    \cite{bursztyn.waldmann:2005a}.
\item The standard-ordered or Weyl-ordered representation gives now
    certain pseudo-differential operators which can be studied by
    means of the symbols in $\Pol_{R, R'}(T^*G)$. The strong analytic
    framework should help to establish functional-analytic properties
    like self-adjointness in the same spirit as this was done in
    \cite{kraus.roth.schoetz.waldmann:2019a, schoetz.waldmann:2018a}.
\item Since the star product $\starstd$ has all needed symmetry
    properties this raises the question whether we can construct
    further classes of examples of convergent star products by means
    of phase space reduction starting with $T^*G$. In view of the
    examples \cite{schmitt2019a:pre,
      kraus.roth.schoetz.waldmann:2019a} one expects a more
    complicated dependence on $\hbar$ after reduction with
    singularities reflecting the geometry of the reduced phase spaces.
\end{itemize}

The paper is organized as follows: In Section~\ref{sec:StarProducts}
we recall the basic construction of $\starstd$ and $\starweyl$ on
$T^*G$ and establish formulas which allow for efficient estimations.
Section~\ref{sec:TopologiesTensor} contains the construction of the
topology on $\Sym^\bullet(\liealg{g})$. We recall some of the basic
properties of the resulting
algebra. Section~\ref{sec:TopologiesEntire} is at the heart of the
paper. We define the entire functions on $G$ by means of their
Lie-Taylor coefficients and study first properties of the resulting
space $\Entire(G)$. In particular, we show that representative
functions belong to $\Entire(G)$ for $0 \le R < 1$. In
Section~\ref{sec:TopologiesObservable} we combine the entire functions
$\Entire(G)$ on $G$ with the polynomials $\SymR(\liealg{g})$ to
the observable algebra $\Pol_{R, R'}(T^*G)$, whose completion will
then be studied in the final Section~\ref{sec:Continuity}. Here the
continuity of the star products is established. In two appendices we
recall the general construction of star products on cotangent bundles
and explain some combinatorial aspects of the Leibniz rule.

%
%

\textbf{Acknowledgements:} We would like to thank Pierre Bieliavsky
for a valuable remark leading to
Proposition~\ref{proposition:TstarGActsAsWell}.

%
%

\section{Star Products on $T^*G$}
\label{sec:StarProducts}%

%
%

In this section, we specialize the constructions of a global symbol
calculus and the corresponding star products on cotangent bundles
\cite{bordemann.neumaier.waldmann:1999a,
  bordemann.neumaier.waldmann:1998a,
  bordemann.neumaier.pflaum.waldmann:2003a} to the cotangent bundle of
a Lie group $G$, see also
Appendix~\ref{sec:StarProductsCotangentBundles} for a brief
introduction to the general situation. The main idea is that having a
\emph{global} frame of left invariant vector fields simplifies many of
the formulas and allows us to use previously local formulas now
globally. Essentially, all formulas we present are known from
\cite{gutt:1983a} as well as \cite{bordemann.neumaier.waldmann:1998a},
but given in slightly different form, making it necessary to adapt
them to our needs.

%
%

\subsection{The Standard-Ordered Star Product on $T^*G$}
\label{subsec:Std}%

Let $G$ be an $n$-dimensional Lie group with Lie algebra
$\liealg{g} = T_{\mathrm{e}} G$. We write $X_\xi \in \Secinfty(TG)$
for the left invariant vector field with $X_\xi(\E) = \xi$ and
$\theta^\alpha \in \Secinfty(T^*G)$ for the left invariant one-form
with $\theta^\alpha(\E) = \alpha$, where $\xi \in \liealg{g}$ and
$\alpha \in \liealg{g}^*$. Then the natural pairing
$\theta^\alpha(X_\xi) = \alpha(\xi) \in \Cinfty(G)$ yields a constant
function on $G$.

After once and for all choosing a basis
$(\basis{e}_1, \ldots, \basis{e}_n)$ of the Lie algebra $\liealg{g}$
with corresponding dual basis $(\basis{e}^1, \ldots, \basis{e}^n)$ of
$\liealg{g}^*$, we write shorthand
\begin{equation}
    \label{eq:LeftInvariantFrames}
    X_i
    =
    X_{\basis{e}_i}
    \quad \textrm{and} \quad
    \theta^i
    =
    \theta^{\basis{e}_i}
\end{equation}
for $i = 1, \ldots, n$ in the sequel. Following
\cite{bordemann.neumaier.waldmann:1998a}, see also
Appendix~\ref{sec:StarProductsCotangentBundles}, to construct a
standard-ordered star product on the cotangent bundle $T^*G$, we have
to specify a torsion-free covariant derivative on $G$ first. The
perhaps first surprising observation is that the most natural
covariant derivative, the half-commutator connection $\nabla$ on $G$,
is \emph{not} the Levi-Civita connection for a Riemannian metric in
general. It would be the Levi-Civita connection of a biinvariant
pseudo Riemannian metric. However, a positive definite one might not
exist at all. Since we have a trivial tangent bundle, it suffices to
specify $\nabla$ on left invariant vector fields. One sets
\begin{equation}
    \label{eq:HalfCommutator}
    \nabla_{X_\xi}
    X_\eta
    =
    \frac{1}{2}
    X_{[\xi, \eta]},
\end{equation}
which is torsion-free, as taking left invariant vector fields is a Lie
algebra morphism by the very definition of the Lie algebra. This then
induces covariant derivatives on the various tensor bundles and their
complexifications, as usual. This is the only covariant derivative we
shall use in the sequel, wherefore we stick to the simple notation
$\nabla$.

The covariant derivatives of the global frames $X_1, \ldots, X_n$ and
$\theta^1, \ldots, \theta^n$ are thus given by the structure constants
$c^k_{ij} = \basis{e}^k([\basis{e}_i, \basis{e}_j])$ of the Lie
algebra $\liealg{g}$, i.e. we have
\begin{equation}
    \label{eq:Christoffel}
    \nabla_{X_i}
    X_j
    =
    \frac{1}{2} c_{ij}^k X_k
    \quad
    \textrm{and}
    \quad
    \nabla_{X_i}
    \theta^k
    =
    - \frac{1}{2} c_{ij}^k \theta^j.
\end{equation}
Here and in the following we shall use Einstein's summation
convention.  The antisymmetry of the structure constants now gives the
following result for the powers of the symmetrized covariant
derivative from \eqref{eq:SymDDef}:
\begin{lemma}
    \label{lemma:SymD}%
    Let $G$ be a Lie group.
    \begin{lemmalist}
    \item \label{item:SymDGlobalFrame}%
        Let $\alpha \in \Secinfty(\Sym^k T^*_\field{C} G)$. Its
        symmetrized covariant derivative is given by the global
        formula
        \begin{equation}
            \label{eq:SymDGlobalFrame}
            \SymD \alpha
            =
            \theta^i
            \vee
            \big(
            \nabla_{X_i}
            \alpha
            \big),
        \end{equation}
        where $\vee$ denotes the symmetric tensor product as usual.
    \item \label{item:SymDPowers}%
        For the $k$-th power of $\SymD$ acting on a function
        $\psi \in \Cinfty(G)$ we have the global formula
        \begin{equation}
            \label{eq:SymDPowers}
            \SymD^k \psi
            =
            \Big(\Lie_{X_ {i_1}} \cdots \Lie_{X_{i_k}} \psi \Big)
            \cdot
            \theta^{i_1} \vee \cdots \vee \theta^{i_k}.
        \end{equation}
    \end{lemmalist}
\end{lemma}
\begin{proof}
    We have already noted \ref{item:SymDGlobalFrame} in
    \eqref{eq:SymDLocalFrame} with the crucial feature that now we
    have a global frame. The second statement is a straightforward
    induction based on \eqref{eq:SymDGlobalFrame} and
    \eqref{eq:Christoffel}.
\end{proof}

Since we have a global frame for $TG$, we can use it to identify the
invariant polynomial functions on $T^*G$ with the complexified
symmetric algebra over the Lie algebra $\liealg{g}$:
\begin{lemma}
    \label{lemma:PolynomialFactorization}%
    Let $G$ be a Lie group.
    \begin{lemmalist}
    \item \label{item:InvariantPolynomials} We have the canonical
        isomorphism
        \begin{equation}
            \label{eq:PolInvariantSg}
            \Sym^\bullet_{\field{C}}(\liealg{g})
            \cong
            \Secinfty\big(\Sym^\bullet_{\field{C}} TG\big)^{G}
            \stackrel{\mathcal{J}}{\longrightarrow}
            \Pol^\bullet(T^*G)^G
        \end{equation}
        between the symmetric algebra of the Lie algebra and the
        invariant polynomials on $T^*G$.
    \item We have the isomorphisms
        \begin{equation}
            \label{eq:PolynomialFactorization}
            \Cinfty(G) \tensor \Sym^\bullet(\liealg{g})
            \cong
            \Secinfty
            \big(
            \Sym_{\field{C}}^\bullet(TG)
            \big)
            \cong
            \Pol(T^*G)
        \end{equation}
        of graded algebras induced by the pullback $\pi^*$
        with the cotangent bundle projection $\pi$.
    \end{lemmalist}
\end{lemma}
Here $\mathcal{J}$ is the canonical algebra isomorphism
\eqref{eq:CanonicalIsoJ}. This factorization will be used extensively
in the sequel. Note that we do not have to complexify the symmetric
algebra in \eqref{eq:PolynomialFactorization}, but doing so would not
change the resulting algebra. We will switch between these points of
view, whenever it is convenient to do so in the sequel.

Using this observation and Lemma~\ref{lemma:SymD}, we get the
following surprisingly simple formula for the standard-ordered
quantization map:
\begin{proposition}[Standard-ordered quantization map]
    \label{proposition:StdRep}%
    Let $G$ be a Lie group.
    \begin{propositionlist}
    \item \label{item:StdRepExplicitInvariantStuff}%
        The standard-ordered quantization map on invariant polynomial
        functions is globally given by
        \begin{equation}
            \label{eq:StdRepExplicit}
            \stdrep
            \big(
            \mathcal{J}
            (X_{\xi_1} \vee \cdots \vee X_{\xi_k})
            \big)
            =
            \bigg(\frac{\hbar}{\I}\bigg)^{k}
            \frac{1}{k!}
            \sum_{\sigma \in S_k}
            \Lie_{X_{\xi_{\sigma(1)}}} \cdots \Lie_{X_{\xi_{\sigma(k)}}}
        \end{equation}
        for $\xi_1, \ldots, \xi_k \in \liealg{g}$.
    \item \label{item:StdRepIsoOnInvariantThings}%
        It provides an isomorphism
        \begin{equation}
            \label{eq:StdRepIso}
            \stdrep
            \colon
            \Sym^{(\bullet)}_{\field{C}}(\liealg{g})
            \longrightarrow
            \Diffop^{(\bullet)}(G)^G
        \end{equation}
        between the complexification of the symmetric algebra over the
        Lie algebra and the invariant differential operators on $G$,
        both viewed as filtered vector spaces.
    \item \label{item:StdRepGeneralFunction} For $\phi \in \Cinfty(G)$
        and $\xi_1, \ldots, \xi_k \in \liealg{g}$ one has
        \begin{equation}
            \label{eq:StdRepGeneral}
            \stdrep
            \big(
            \pi^*(\phi)
            \mathcal{J}
            (X_{\xi_1} \vee \cdots \vee X_{\xi_k})
            \big)
            =
            \phi
            \cdot
            \stdrep
            \big(
                \mathcal{J}
                (X_{\xi_1} \vee \cdots \vee X_{\xi_k})
            \big),
        \end{equation}
        i.e. the smooth function $\phi$ acts as a multiplication
        operator.
    \end{propositionlist}
\end{proposition}
\begin{proof}
    The first part was obtained in
    \cite[Prop.~11]{bordemann.neumaier.waldmann:1998a}. For
    \ref{item:StdRepIsoOnInvariantThings} we note that
    $\stdrep(\mathcal{J} (X_{i_1} \vee \cdots \vee X_{i_k}))$ is
    clearly an invariant differential operator. Conversely, if
    $D \in \Diffop^k(G)^G$ is invariant, it has an invariant leading
    symbol
    $\sigma_k(D) \in \Secinfty(\Sym^k_{\field{C}} TG)^G \cong
    \Sym^k_{\field{C}}(\liealg{g})$. Quantizing this symbol via
    $\stdrep$ gives an invariant differential operator with the same
    leading symbol, thus $D - \stdrep(\sigma_k(D))$ is of strictly
    lower order. A simple induction on $k$ then proves the isomorphism
    \eqref{eq:StdRepIso}, since we already know that $\stdrep$ is
    injective. The last statement is clear, as $\stdrep$ is left
    $\Cinfty(G)$-linear in general.
\end{proof}

As the standard-ordered quantization map is the quantization map for
the standard-ordered star product $\starstd$ in the sense of
\eqref{eq:starstdDef}, the strategy is now to use
\eqref{eq:StdRepExplicit} to derive a formula for the standard-ordered
star product suitable for estimation.

Thanks to Lemma~\ref{lemma:PolynomialFactorization}, we can compute
the star products for $\Cinfty(G) \tensor \Sym^\bullet(\liealg{g})$
directly. The star product of two functions from $\Cinfty(G)$ is the
commutative pointwise product, a feature which holds for all cotangent
bundles and not only for $T^*G$. The next combination we are
interested in are two elements of
$\Sym^\bullet_{\field{C}}(\liealg{g})$. Since the covariant derivative
$\nabla$ we use to construct $\starstd$ is left invariant, their star
product is an invariant polynomial, i.e. an element of
$\Sym^\bullet_{\field{C}}(\liealg{g})$ again. From
\cite[Lem.~10]{bordemann.neumaier.waldmann:1998a} we infer that
$\starstd$ coincides with the Gutt star product \cite{gutt:1983a} on
$\Sym^\bullet_{\field{C}}(\liealg{g})$, which is obtained from the
linear Poincaré-Birkhoff-Witt isomorphism
\begin{equation}
    \label{eq:PBWIso}
    \Sym^\bullet_{\field{C}}(\liealg{g})
    \cong
    \mathrm{U}_{\field{C}}(\liealg{g})
\end{equation}
to the universal enveloping algebra via symmetrization. Incorporating
the correct powers of the formal parameter into the definition then
yields the star product $\staralg$ for
$\Sym^\bullet_{\field{C}}(\liealg{g})$, where we follow the sign
conventions from \cite{esposito.stapor.waldmann:2017a}. Finally, we
have to take care of the mixed products: the property of a standard
ordered star product immediately gives
$(\phi \tensor 1) \starstd (\psi \tensor \xi) = (\phi \psi) \tensor
\xi$ for all $\phi, \psi \in \Cinfty(G)$ and
$\xi \in \Sym^\bullet(\liealg{g})$. Thus it is the opposite order
which needs to be computed. We summarize the result from
\cite[Prop.~11]{bordemann.neumaier.waldmann:1998a} in the following
proposition, adapting it to our present notation:
\begin{proposition}
    \label{proposition:StdFactorizations}%
    Let $G$ be a Lie group.
    \begin{propositionlist}
    \item \label{item:TrivialFromLeft}%
        Functions act trivially from the left, i.e. we have
        \begin{equation}
            \label{eq:StdFactorizationsLeft}
            \big(\phi \tensor 1 \big)
            \starstd
            \big( \psi \tensor \xi \big)
            =
            (\phi \cdot \psi) \tensor \xi
        \end{equation}
        for all $\phi, \psi \in \Cinfty(G)$ and
        $\xi \in \Sym^\bullet_{\field{C}}(\liealg{g})$.
    \item \label{item:xietaLieAlg} %
        Products of invariant polynomials are given by the Lie algebra
        star product $\staralg$, i.e. we have
        \begin{equation}
            \label{eq:StdFactorizationsRight}
            \big(1 \tensor \xi \big)
            \starstd
            \big(1 \tensor \eta \big)
            =
            1 \tensor \big( \xi \staralg \eta \big)
        \end{equation}
        for $\xi, \eta \in \Sym^\bullet_{\field{C}}(\liealg{g})$.
    \item \label{item:xiStarphi}%
        The remaining combination of interest is
        \begin{align}
            \nonumber
            \big(
            1 &\tensor \xi_{1} \vee \cdots \vee \xi_{k}
            \big)
            \starstd
            \big(
            \phi \tensor 1
            \big) \\
            &=
            \label{eq:StdExplicitMixed}
            \sum_{p=0}^k
            \bigg(\frac{\hbar}{\I}\bigg)^{p}
            \frac{1}{p! \; (k-p)!}
            \sum_{\sigma \in S_k}
            \Big(
            \Lie_{X_{\xi_{\sigma(1)}}}
            \cdots
            \Lie_{X_{\xi_{\sigma(p)}}} \phi
            \Big)
            \tensor
            \xi_{\sigma(p+1)} \vee \cdots \vee \xi_{\sigma(k)},
        \end{align}
        where $\xi_1, \ldots, \xi_k \in \liealg{g}$ and
        $\phi \in \Cinfty(G)$.
    \item \label{item:GeneralStarstdGeneral}%
        In general, one has
        \begin{align}
            \nonumber
            &(\phi \tensor \xi_1 \vee \cdots \vee \xi_k)
            \starstd
            (\psi \tensor \eta) \\
            &\quad=
            \label{eq:StdFactorization}
            (\phi \tensor 1)
            \starstd
            (\mathbb{1} \tensor \xi_1 \vee \cdots \vee \xi_k)
            \starstd
            (\psi \tensor 1)
            \starstd
            (\mathbb{1} \tensor \eta) \\
            &\quad=
            \label{eq:StdExplicit}
            \sum_{p=0}^k
            \bigg(\frac{\hbar}{\I}\bigg)^{p}
            \frac{\phi}{p! \; (k-p)!}
            \sum_{\sigma \in S_k}
            \Lie_{X_{\xi_{\sigma(1)}}} \cdots \Lie_{X_{\xi_{\sigma(p)}}} \psi
            \tensor
            \big(\xi_{\sigma(p+1)} \vee \cdots \vee \xi_{\sigma(k)}\big)
            \staralg
            \eta
        \end{align}
        for $\phi, \psi \in \Cinfty(G)$,
        $\xi_1, \ldots, \xi_k \in \liealg{g}$ and
        $\eta \in \Sym^\bullet(\liealg{g})$.
    \end{propositionlist}
\end{proposition}
\begin{proof}
    The presented formulae are obtained from
    \cite[Prop.~11]{bordemann.neumaier.waldmann:1998a} after the
    suitable identification of polynomial functions with elements in
    $\Cinfty(G) \tensor \Sym^\bullet_{\field{C}}(\liealg{g})$. We list
    them here, since directly working with the symmetric algebra will
    be easier for continuity estimates down the line.
\end{proof}

%
%

\subsection{Weyl Ordering and The Neumaier Operator}
\label{subsec:Neumaier}

This completes our algebraic considerations for the standard-ordered
star product. In a next step, we turn towards other ordering
prescriptions. More precisely, we are going to simplify the general
formulas for the Neumaier operator \eqref{eq:NeumaierOperatorTstarQ}
by utilizing the Lie-theoretic situation. As for the standard-ordered
star product, having a global frame allows us to obtain considerably
more explicit formulas, see again
\cite[Sect.~8]{bordemann.neumaier.waldmann:1998a}.

The main idea is that we use the half-commutator connection to lift
the left-invariant global frame $X_1, \ldots, X_n$ to vector fields
\begin{equation}
    \label{eq:YiDef}
    Y_i = X_i^\hor
    \in
    \Secinfty\big(T(T^*G)\big)
\end{equation}
on the cotangent bundle. Together with the vertical lifts of the
global frame $\theta^1, \ldots, \theta^n$, denoted by
\begin{equation}
    \label{eq:ZiDef}
    Z^i
    =
    (\theta^i)^\ver
    \in
    \Secinfty\big(T(T^*G)\big),
\end{equation}
one thus obtains a global frame $Y_1, \ldots, Y_n, Z^1, \ldots Z^n$
for the tangent bundle of $T^*G$. Having a global frame it is of
course advantageous to express differential operators like the
Laplacian $\Laplace_0$ from \eqref{eq:LaplaceTstarQ} of the pseudo
Riemannian metric $g_0$ by iterated Lie derivatives with respect to
the frame vector fields instead of covariant derivatives.

Since we also need the choice of a volume density $\mu$ on $G$ to
construct the Weyl ordering, we use the left-invariant volume form
$\mu = \theta^1 \wedge \cdots \wedge \theta^n$. The required one-form
$\alpha$ with $\nabla_X \mu = \alpha(X) \mu$ is then given by
$\alpha(X_\xi) = - \tfrac{1}{2} \tr(\ad_\xi)$ for $\xi \in \liealg{g}$
and hence
\begin{equation}
    \label{eq:alphaGlobally}
    \alpha
    =
    \frac{1}{2} c_{ik}^i \theta^k \in \Secinfty(T^*G).
\end{equation}
Note that in general $\alpha \ne 0$ unless the Lie algebra is
unimodular. The vertical lift of $\alpha$ gives
$\alpha^\ver = \frac{1}{2} c_{ik}^i Z^k$. For the operator $N$ we need
the combination $\Laplace_0 + \Lie_{\alpha^\ver}$, wherefore we
compute the action of this operator on factorizing tensors explicitly:
\begin{proposition}
    \label{proposition:TheTrueLaplacian}%
    For $\phi \in \Cinfty(G)$ and $\xi_1, \ldots, \xi_k \in
    \liealg{g}$ one has
    \begin{equation}
        \label{eq:NeumaierLaplacian}
        \big(
            \Laplace_0 + \Lie_{\alpha^\ver}
        \big)
        (\phi \tensor \xi_1 \vee \cdots \vee \xi_k)
        =
        \sum_{\ell=1}^k
        \Lie_{X_{\xi_\ell}} \phi
        \tensor
        \xi_1 \vee \overset{\xi_\ell}{\cdots} \vee \xi_k
        +
        2 \Lie_{\alpha^\ver},
    \end{equation}
    where we identify elements of
    $\Cinfty(G) \tensor \Sym(\liealg{g})$ with polynomial functions
    $\Pol(T^*G)$ as before.
\end{proposition}
\begin{proof}
    In general, the covariant divergence
    $\divergence_\nabla \colon \Secinfty(\Sym^k TQ) \longrightarrow
    \Secinfty(\Sym^{k-1} TQ)$ on an arbitrary manifold $Q$ with
    covariant derivative $\nabla$ is given by the local formula
    \begin{equation}
        \label{eq:DivergenceFrame}
        \divergence_\nabla\at[\Big]{U}
        =
        \inss(\frames{e}^i) \nabla_{\frames{e}_i},
        \tag{$*$}
    \end{equation}
    where $\frames{e}_1, \ldots, \frames{e}_n \in \Secinfty(TU)$ is a
    local frame on an open subset $U \subseteq M$ with corresponding
    dual frame
    $\frames{e}^1, \ldots, \frames{e}^n \in \Secinfty(T^*U)$ and
    $\inss(\argument)$ denotes the symmetric insertion derivation. As
    one easily verifies, this provides a global definition independent
    of the local frame. Directly from the definition one infers the
    Leibniz rule
    \begin{equation}
        \label{eq:DivergenceLeibniz}
        \divergence_\nabla(\phi X)
        =
        \inss(\D \phi)X + \phi \divergence_\nabla(X)
        \tag{$**$}
    \end{equation}
    for all $\phi \in \Cinfty(Q)$ and $X \in \Secinfty(\Sym^k TQ)$.
    On polynomial functions $\mathcal{J}(X) \in \Pol^k(T^*Q)$ with
    $X \in \Secinfty(\Sym^k TQ)$ the Lie derivatives with respect to
    horizontal and vertical lifts act like
    \begin{equation*}
        \Lie_{Y^\hor} \mathcal{J}(X)
        =
        \mathcal{J}(\nabla_Y X)
        \quad
        \textrm{and}
        \quad
        \Lie_{\beta^\ver} \mathcal{J}(X)
        =
        \mathcal{J}(\inss(\beta) X),
    \end{equation*}
    where $Y \in \Secinfty(TQ)$ and $\beta \in \Secinfty(T^*Q)$. Since
    $\mathcal{J}$ is an algebra homomorphism this can be easily
    checked on generators. Using the local expression
    \eqref{eq:LaplaceTstarQ} for the Laplacian with respect to the
    $g_0$ as well as the local formulas for the horizontal and
    vertical lifts, one verifies
    \begin{equation*}
        \Laplace_0 \circ \mathcal{J}
        =
        \mathcal{J} \circ \divergence_\nabla,
    \end{equation*}
    see also \cite[Eq.~(111)]{bordemann.neumaier.waldmann:1998a}.  Now
    we focus on the Lie group case. Here we first notice that
    $\divergence_\nabla(X_\xi) = - \tfrac{1}{2} \tr \ad_\xi$ for all
    $\xi \in \liealg{g}$. Note that \eqref{eq:DivergenceFrame} becomes
    a global formula once we use the global frame $X_1, \ldots, X_n$.
    From the antisymmetry of the structure constants we get the
    divergence of higher polynomials as
    \begin{equation*}
        \divergence_\nabla
        (X_{\xi_1} \vee \cdots \vee X_{\xi_k})
        =
        - \tfrac{1}{2}
        \inss(\tr\ad) (X_{\xi_1} \vee \cdots \vee X_{\xi_k}).
    \end{equation*}
    Together with the Leibniz rule \eqref{eq:DivergenceLeibniz} we
    arrive at the explicit formula
    \begin{equation*}
        \divergence_\nabla
        (\phi X_{\xi_1} \vee \cdots \vee X_{\xi_k})
        =
        \sum_{\ell = 1}^k
        \Lie_{X_{\xi_\ell}} \phi
        \cdot
        X_{\xi_1} \vee \stackrel{\ell}{\cdots} \vee X_{\xi_k}
        -
        \tfrac{1}{2}
        \phi
        \inss(\tr\ad)(X_{\xi_1} \vee \cdots \vee X_{\xi_k}).
    \end{equation*}
    Applying the algebra isomorphism $\mathcal{J}$ turns the
    divergence into the Laplacian and the insertion of the modular
    one-form into the Lie derivative in direction of the vertical lift
    of $\alpha$, finally proving \eqref{eq:NeumaierLaplacian}.
\end{proof}

From this explicit description of the Laplacian $\Laplace_0$ we see
that it might be advantageous to focus on the combination
\begin{equation}
    \label{eq:TheTrueLaplacian}
    \Laplace
    =
    \Laplace_0 - \Lie_{\alpha^\ver}
\end{equation}
acting on polynomial functions as
\begin{equation}
    \label{eq:TrueLaplaceActsNicely}
    \Laplace
    (\phi \tensor \xi_1 \vee \cdots \vee \xi_k)
    =
    \sum_{\ell=1}^k
    \Lie_{X_{\xi_\ell}} \phi
    \tensor
    \xi_1 \vee \overset{\xi_\ell}{\cdots} \vee \xi_k.
\end{equation}

The vertical Lie derivative $\Lie_{\alpha^\ver}$ is now easily shown
to commute with both operators $\Laplace_0$ and $\Laplace$. Thus the
Neumaier operator $\mathcal{N}_\kappa$ factorizes
\begin{equation}
    \label{eq:FactorizeNeumaier}
    \mathcal{N}_\kappa
    =
    \exp\left(
        - \I\kappa \hbar (\Laplace_0 + \Lie_{\alpha^\ver})
    \right)
    =
    \exp\left(
        - \I\kappa \hbar \Laplace
    \right)
    \circ
    \exp\left(
        2\I\kappa\hbar \Lie_{\alpha^\ver}
    \right).
\end{equation}
As observed in \cite[Lem.~11]{bordemann.neumaier.waldmann:1998a}, the
second factor $\exp(2\I\kappa\hbar \Lie_{\alpha^\ver})$ is an
automorphism of $\starstd$ for all $\kappa$. Thus the simpler operator
\begin{equation}
    \label{eq:TrueNeumaier}
    N_\kappa
    =
    \exp\left(
        - \I\kappa \hbar \Laplace
    \right)
\end{equation}
with $\Laplace$ as in \eqref{eq:TrueLaplaceActsNicely} is still an
equivalence transformation from $\starstd$ to $\stark$ for all
$\kappa \in \field{R}$. The reason to use $N_\kappa$ instead of
$\mathcal{N}_\kappa$ is that the simpler formulas are easier to
estimate later on.
\begin{corollary}
    \label{corollary:LaplaceViaPoissonBracket}%
    The Laplacian is given by the mixed Poisson bracket, i.e. we have
    \begin{equation}
        \label{eq:NeumaierLaplacianIsPoisson}
        \Laplace
        (\phi \tensor \xi_1 \vee \cdots \vee \xi_k)
        =
        \big\{
        \mathbb{1} \tensor \xi_1 \vee \cdots \vee \xi_k,
        \phi \tensor 1
        \big\}.
    \end{equation}
\end{corollary}
\begin{proof}
    Indeed, the Poisson bracket of a function $\phi \in \Cinfty(G)$
    with left-invariant vector fields can be directly obtained from
    Proposition~\ref{proposition:TheTrueLaplacian} and the first order
    commutator of $\starstd$ as in
    Proposition~\ref{proposition:StdFactorizations}. Note that such a
    formula is only possible since we can factorize elements in
    $\Pol(T^*G)$ into $\Cinfty(G)$ and $\Sym^\bullet(\liealg{g})$.
\end{proof}
\begin{corollary}
    \label{corollary:LaplacianPowers}%
    Let $\ell \le k$. The powers of $\Laplace$ act as
    \begin{equation}
        \label{eq:NeumaierLaplacianPowers}
        \big(\Laplace\big)^{\ell}
        (\phi \tensor \xi_1 \vee \cdots \vee \xi_k)
        =
        \frac{1}{(k-\ell)!}
        \sum_{\sigma \in S_k}
        \Lie_{X_{\xi_{\sigma(1)}}}
        \cdots
        \Lie_{X_{\xi_{\sigma(\ell)}}}
        \phi
        \tensor
        \xi_{\sigma(\ell+1)}
        \vee
        \cdots
        \vee
        \xi_{\sigma(k)}.
    \end{equation}
    For $\ell \ge k$ the result is zero for degree reasons.
\end{corollary}

To show the continuity of the $\kappa$-Neumaier operators later on, we
require a more explicit formula. Remarkably,
\eqref{eq:NeumaierLaplacianIsPoisson} exponentiates very nicely: it
turns out that the square $N^2$ of Neumaier operator is given by the
mixed star product from Theorem~\ref{proposition:StdFactorizations}:
\begin{proposition}
    \label{proposition:NeumaierIsMixedStd}%
    Let $G$ be a Lie group. For $\kappa = \frac{1}{2}$, the square of
    the Neumaier operator $N = N_{\frac{1}{2}}$ is given by
    \begin{equation}
        \label{eq:NeumaierIsMixedStd}
        N^2
        (\phi \tensor \xi_1 \vee \cdots \vee \xi_k)
        =
        (\mathbb{1} \tensor \xi_1 \vee \cdots \vee \xi_k)
        \starstd
        (\phi \tensor 1)
    \end{equation}
    for $\phi \in \Cinfty(G)$ and $\xi_1, \ldots, \xi_k \in \liealg{g}$.
\end{proposition}
\begin{proof}
    Taking another look at \eqref{eq:StdExplicitMixed} and
    \eqref{eq:NeumaierLaplacianPowers} confirms our claim.
\end{proof}

Notably, incorporating $\kappa \neq 1$ is not that easy here. Down the
line, the trick will thus be to absorb it into the $\hbar$ dependence,
at which point we can employ
Proposition~\ref{proposition:NeumaierIsMixedStd} again.

%
%

\section{The $R'$-Topologies on the Symmetric Algebra}
\label{sec:TopologiesTensor}%

%
%

In view of the factorization
$\Pol(T^*G) \cong \Cinfty(G) \tensor \Sym^\bullet(\liealg{g})$, we
want to define a suitable locally convex topology on the symmetric
algebra $\Sym^\bullet(\liealg{g})$ in such a way that the star product
$\star_{\liealg{g}}$ is continuous. This has been accomplished and
studied in detail in \cite{esposito.stapor.waldmann:2017a}. We briefly
recall the construction based on the earlier work
\cite{waldmann:2014a} and recollect some of the crucial features.

Let $V$ be a locally convex space over $\field{K} = \field{R}$ or
$\field{K} = \field{C}$. We fix a parameter $R' \in \field{R}$. Then
for a seminorm $\seminorm{p}$ on $V$ and a weight $c \ge 0$ one
defines the seminorm $\seminorm{p}_{R', c}$
\begin{equation}
    \label{eq:RSeminorm}
    \seminorm{p}_{R', c}
    \colon
    \Sym^\bullet(V)
    \longrightarrow
    \field{R}_0^+,
    \quad
    \seminorm{p}_{R', c}
    =
    \sum_{k=0}^\infty
    k!^{R'} \; c^k \;
    \seminorm{p}^{k},
\end{equation}
where $\seminorm{p}^{k}$ denotes the $k$-fold projective tensor power
of the seminorm $\seminorm{p}$ acting on
$\Sym^k(V) \subseteq \Tensor^\bullet(V)$. By convention,
$\Sym^0(V) = \field{K} = \Tensor^0(V)$ and $\seminorm{p}^0$ is the
absolute value on $\field{K}$.

Let $\mathcal{P}$ be a defining system of seminorms for $V$. The
locally convex topology on $\Sym^\bullet(V)$ induced by the set of
seminorms
$\big( \seminorm{p}_{R', c} \big)_{\seminorm{p} \in \mathcal{P}, \; c
  \ge 0}$ is called the $R'$-topology. We write $\SymR(V)$ for
$\Sym^\bullet(V)$ equipped with the $R'$-topology. It is independent
of the chosen defining system of seminorms and thus intrinsic to
$V$. However, it depends on $R'$ in a very sensitive way. We note the
obvious inequalities
\begin{equation}
    \label{eq:RSeminormInequalities}
    \seminorm{p}_{R', c}
    \le
    \seminorm{p}_{S', c}
    \quad \textrm{and} \quad
    \seminorm{p}_{R', c}
    \le
    \seminorm{p}_{R', d}
    \quad \textrm{and} \quad
    \seminorm{p}_{R', c}
    \le
    \seminorm{q}_{R', c},
\end{equation}
whenever $R' \le S'$, $0 \le c \le d$ and
$\seminorm{p} \le \seminorm{q}$. This implies that the inclusion (in
fact equality)
\begin{equation}
    \label{eq:ContinuousInclusionStoR}
    \SymR[S'](V)
    \subseteq
    \SymR(V)
\end{equation}
is continuous. Thus it extends to a continuous inclusion for the
completions.

In \cite{waldmann:2014a}, all continuous seminorms were chosen. In
this case, there is no need for the parameter $c \ge 0$ as with
$\seminorm{p}$ also $c \seminorm{p}$ is continuous. Nevertheless, a
smaller collection is sometimes convenient: Let $V = \liealg{g}$ be
the Lie algebra of a Lie group $G$. In the sequel we will always
choose a basis of $\liealg{g}$ and equip it with the corresponding
$\ell^1$-topology. This then induces an $R'$-topology on
$\Sym^\bullet(\liealg{g})$ via $\mathcal{P} = \{\norm{\argument}_1\}$,
i.e. the system consists of a single norm. Note, however, that the
topology for $\SymR(\liealg{g})$ is not at all normable. In general,
we also note
\begin{equation}
    \label{eq:RSeminormNotation}
    \big(
        \lambda \cdot \seminorm{p}
    \big)_{R', c}
    =
    \seminorm{p}_{R', \lambda c}
\end{equation}
for all $c \ge 0$, $\lambda \ge 0$ and $R' \in \field{R}$. Another
consequence of having all polynomial weights at our disposal is that
instead of the $\ell^1$-like seminorms \eqref{eq:RSeminorm} we could
have used the $\ell^\infty$-like seminorms
\begin{equation}
    \label{eq:RSeminormLinfty}
    \seminorm{p}_{R', c, \infty}
    \colon
    \Sym^\bullet(V)
    \longrightarrow \field{R}_0^+, \quad
    \seminorm{p}_{R', c, \infty}
    =
    \sup_{k \in \field{N}_0}
    k!^{R'} \;
    c^k \;
    \seminorm{p}^{k}(v_k)
\end{equation}
with $c \ge 0$. The mutual estimates between them show that the
resulting locally convex topology stays the same, see
\cite[Lem.~3.4]{waldmann:2014a}. We collect a few less obvious
properties of the $R'$-topology from \cite{waldmann:2014a}:
\begin{proposition}
    \label{proposition:RTopology}%
    Let $V$ be a locally convex vector space with defining system of
    seminorms $\mathcal{P}$ and $R' \in \field{R}$.
    \begin{propositionlist}
    \item \label{item:RTopologyProjective}%
        Let $k \ge 0$. The subspace topology induced by the inclusion
        $\Sym^k(V) \subset \Sym^\bullet(V)$ is the projective tensor
        power topology and the inclusion is continuous.
    \item \label{item:RTopologyDirectSum}%
        The $R'$-topology is coarser than the locally convex direct
        sum topology.
    \item \label{item:RTopologyCartesian}%
        The $R'$-topology is finer than the subspace topology induced
        by the Cartesian product topology.
    \item \label{item:RTopologyHausdorff}%
        The locally convex space $\SymR(V)$ is Hausdorff iff $V$ is
        Hausdorff.
    \item \label{item:RTopologyFirstCountable}%
        The locally convex space $\SymR(V)$ is first countable iff $V$
        is first countable.
    \item \label{item:RTopologyNuclear}%
        Let $R' \ge 0$. The locally convex space $\SymR(V)$ is nuclear
        iff $V$ is nuclear.
    \item \label{item:RTopologyCompletion}%
        The completion $\hat{\Sym}_{R'}^\bullet(V)$ of $\SymR(V)$
        is explicitly given by
        \begin{equation}
            \label{eq:RTopologyCompletion}
            \SymRC(V)
            =
            \bigg\{
            v \in \prod_{k=0}^\infty \hat{\Sym}^k(V)
            \; \bigg\vert \;
            \seminorm{p}_{R', c}(v) < \infty
            \; \text{for all} \;
            c \ge 0
            \; \textrm{and} \;
            \seminorm{p} \in \mathcal{P}
            \bigg\},
        \end{equation}
        where $\hat{\Sym}^k(V)$ denotes the completion of $\Sym^k(V)$
        with respect to the projective tensor product topology.
    \item \label{item:RTopologyMontel}%
        Let $R' \ge 0$ and $V$ be nuclear, Hausdorff as well as first countable.
        Then the completion $\SymRC(V)$ is nuclear Fréchet, Montel,
        separable, and reflexive.
    \end{propositionlist}
\end{proposition}
\begin{proof}
    All statements except for the last one have been obtained in
    \cite{waldmann:2014a}. The earlier parts of our theorem guarantee
    that nuclearity, the Hausdorff property and first countability get
    inherited by $\SymR(V)$. Moreover, everything passes to the
    completion $\SymRC(V)$, as well (see \cite[(50.3)]{treves:1967a}
    for the nuclearity). By
    \cite[Prop.~50.2]{treves:1967a} every nuclear Fréchet space is
    Montel. Nuclear Montel spaces are separable by \cite[Section~11.6,
    Thm.~2]{jarchow:1981a} and by \cite[Cor.~36.9]{treves:1967a} every
    Montel space is reflexive.
\end{proof}

From \cite{waldmann:2014a} we also get the following continuity
statements for the symmetric tensor product:
\begin{proposition}
    \label{proposition:PropertiesSR}%
    Let $V$ be a locally convex vector space, $c \ge 0$ and
    $R' \in \field{R}$.
    \begin{propositionlist}
    \item \label{item:RTopologyProductsSymmetric}%
        The symmetric tensor product
        \begin{equation}
            \label{eq:SymmetricTensorProduct}
            \vee
            \colon
            \SymR(V)
            \times
            \SymR(V)
            \longrightarrow
            \SymR(V)
        \end{equation}
        is continuous. More precisely, for $R' \ge 0$ we have
        \begin{equation}
            \label{eq:RTopologyProduct}
            \seminorm{p}_{R',c}(v \vee w)
            \le
            \seminorm{p}_{R', 2^{R'} c}(v)
            \cdot
            \seminorm{p}_{R', 2^{R'} c}(w)
        \end{equation}
        for all $v, w \in \Sym^\bullet(V)$ and the seminorms
        $\seminorm{p}_{R', c}$ are submultiplicative for $R' \le 0$.
    \item \label{item:RTopologyEvaluationsContinuity}%
        Let $R' \ge 0$. For $\varphi \in V'$ the evaluation
        functionals
        \begin{equation}
            \label{eq:RTopologyEvaluations}
            \delta_\varphi
            \colon
            \SymR(V)
            \longrightarrow
            \field{C}, \quad
            \delta_\varphi(v)
            =
            \sum_{k=0}^\infty
            \varphi^{k}(v_k)
        \end{equation}
        are continuous algebra characters.
    \item \label{item:RTopologyEvaluationsDiscontinuity}%
        Let $R' < 0$ and $\varphi \in V^*$ with continuous
        $\delta_\varphi$ as in \eqref{eq:RTopologyEvaluations}. Then
        we have $\varphi = 0$.
    \end{propositionlist}
\end{proposition}
\begin{proof}
    The only new statement is
    \ref{item:RTopologyEvaluationsDiscontinuity}: let $R' < 0$ and
    $\varphi \in V^*$ with continuous $\delta_\varphi$, i.e. we find a
    continuous seminorm $\seminorm{p}$ on $V$ and $c \ge 0$ such that
    we have
    \begin{equation*}
        \label{eq:RTopologyEvaluationsProof}
        \abs[\bigg]
        {\sum_{k=0}^\infty
        \varphi^{k}(v_k)}
        =
        \abs[\big]
        {\delta_\varphi (v)}
        \le
        \seminorm{p}_{R', c}(v)
        =
        \sum_{k=0}^\infty
        k!^{R'} \; c^k \;
        \seminorm{p}^k(v_k)
        \tag{$*$}
    \end{equation*}
    for all $v \in V$. Assume that we had $\varphi \neq 0$.  Then
    there exists a $v \in V$ with $\varphi(v) \neq 0$.  Consider now
    $v^{\tensor n} \in \Tensor^n(V)$ for all $n \in \field{N}$. Then
    \eqref{eq:RTopologyEvaluationsProof} implies
    \begin{equation*}
        \abs[\big]{\varphi(v)}^n
        =
        \abs[\bigg]
        {\sum_{k=0}^\infty
            \varphi^{k}(v^{\tensor n})}
        \le
        \seminorm{p}_{R', c}
        \big( v^{\tensor n} \big)
        =
        n!^{R'} \cdot c^n
        \cdot
        \seminorm{p}^n
        \big(v^{\tensor n}\big)
        =
        n!^{R'} \cdot c^n
        \cdot
        \seminorm{p}(v)^n
    \end{equation*}
    for all $n \in \field{N}$. However, this inequality is absurd:
    factorial growth in $n$ can not be estimated by the \emph{fixed}
    base
    $\Big( \tfrac{c \cdot \seminorm{p}(v)}{\abs{\varphi(v)}}
    \Big)^{-(R'^{-1})}$ to the power of $n$.
\end{proof}

%
%

\section{The $R$-Entire Functions}
\label{sec:TopologiesEntire}%

The purpose of this section is to introduce and study Fréchet
subalgebras $\Entire(G) \subseteq \Cinfty(G)$ depending on another
parameter $R \in \field{R}$. These Fr\'echet algebras will ultimately
serve as the other tensor factors in the observable algebra for our
strict deformation. While in the critical borderline case $R=0$ the
algebra $\Entire[0](G)$ can be seen as a Lie--theoretic descendent of
the algebra of all holomorphic entire functions, the algebras
$\Entire(G)$ for $R>0$ share many properties with the classical and
well studied Fréchet algebras of entire holomorphic functions of
\textit{finite order and minimal type}.

In this section $G$ denotes always a real Lie group with corresponding
Lie algebra $\liealg{g}$ of dimension $n \in \field{N}$. We
furthermore assume that the Lie group $G$ is \textit{connected}. As it
is standard, we denote for an open set $U \subseteq \field{C}^n$ the
set of all holomorphic functions
$F \colon U \longrightarrow \field{C}$ by $\Holomorphic(U)$.

%
%

\subsection{Lie-Taylor Series of Smooth Functions on a Lie Group}

Taking another look at \eqref{eq:RTopologyCompletion} and
\eqref{eq:StdRepExplicit}, we anticipate certain power series of Lie
derivatives $\Lie_{X_\xi}$ for $\xi \in \liealg{g}$ to make an
appearance, since the completion $\SymRC(\liealg{g})$ of
$\SymR(\liealg{g})$ contains already non-trivial (though not all)
entire functions, when interpreting the elements of the completion as
maps on $\liealg{g}^*$. Thus we need to find a space of functions on
$G$ on which all elements of $\SymRC(\liealg{g})$ act and which is
preserved by this action.

Formalizing this idea, it turns out that the functions we are looking
for are exactly the ``entire vectors'' for suitably chosen seminorms
on $\Cinfty(G)$ and the lifted Lie algebra representation
\begin{equation}
    \label{eq:LieDerivativeRep}
    \Lie
    \colon
    \Universal_\field{C}(\liealg{g})
    \longrightarrow
    \End\big(\Cinfty(G)\big),
    \quad
    \Lie(\xi_1 \cdots \xi_k)
    =
    \Lie_{X_{\xi_1}} \cdots \Lie_{X_{\xi_k}}.
\end{equation}
Here $\Universal(\liealg{g})$ denotes the universal enveloping algebra
of the Lie algebra $\liealg{g}$, as before.

To make the notion of an ``entire vector'' precise, we introduce some
more notation. Let $\field{N}_n = \{1, \ldots, n\}$ and
$\alpha = (\alpha_1,\ldots, \alpha_k) \in \mathbb{N}_n^k$ be an
ordered $k$-tuple. For a fixed basis
$\mathcal{B} = (\basis{e}_1, \ldots, \basis{e}_n)$ of the Lie algebra
$\liealg{g}$, we then write
\begin{equation}
    \label{eq:LieDerivativePower}
    \Lie_{X_{\alpha}}
    =
    \Lie_{X_{\alpha_1}}
    \cdots
    \Lie_{X_{\alpha_k}},
\end{equation}
where we once again use the convention
\eqref{eq:LeftInvariantFrames}. Finally, we also use the shorthand
notation
\begin{equation}
    \label{eq:VectorPower}
    \underline{z}_{\alpha}
    =
    z_{\alpha_1} \cdots z_{\alpha_k}
    \quad \textrm{for} \quad
    \underline{z}
    =
    (z_1, \ldots, z_n)
    \in
    \mathbb{C}^n.
\end{equation}
\begin{definition}[Lie-Taylor series and majorants]
    \label{definition:LieTaylor}%
    Let $\phi \in \Cinfty(G)$ be a smooth function and let
    $\mathcal{B} = (\basis{e}_1, \ldots, \basis{e}_n)$ be a basis of
    the Lie algebra $\liealg{g}$.
    \begin{definitionlist}
    \item \label{item:LieTaylorSeries}%
        We call the formal series
        \begin{equation}
            \label{eq:LieTaylorFormal}
            \Taylor_\phi
            \colon
            G
            \longrightarrow
            \field{C}\formal{\underline{z}}, \quad
            \Taylor_\phi
            (\underline{z}; g)
            =
            \sum \limits_{k=0}^{\infty}
            \frac{1}{k!}
            \sum \limits_{\alpha \in \mathbb{N}_n^k}
            \big(
                \Lie_{X_{\alpha}} \phi
            \big)(g)
            \cdot
            \underline{z}_{\alpha}
        \end{equation}
        the Lie-Taylor series of $\phi$ at the point $g \in G$
        (w.r.t.~the basis $\mathcal{B}$).
    \item \label{item:LieTaylorMajorant}%
        Using the coefficients
        \begin{equation}
            \label{eq:LieTaylorMajorantCoefficients}
            c_k(\phi)
            =
            \frac{1}{k!}
            \sum \limits_{\alpha \in \mathbb{N}_n^k}
            \abs[\Big]{\big( \Lie_{X_{\alpha}} \phi \big)(\E)}
        \end{equation}
        we define the Lie-Taylor majorant of $\phi$ (w.r.t.~the basis
        $\mathcal{B}$) as
        \begin{equation}
            \label{eq:LieTaylorMajorant}
            \Maj_\phi(z)
            =
            \sum \limits_{k=0}^{\infty}
            c_k(\phi)
            \cdot
            z^k
            \in
            \field{C}\formal{z}.
        \end{equation}
    \end{definitionlist}
\end{definition}
\begin{remark}[Lie-Taylor majorants]
    \label{remark:LieTaylorMajorant}%
    Let $\phi \in \Cinfty(G)$ be a smooth function.
    \begin{remarklist}
    \item \label{item:LieTaylorMajorantIsMajorant}%
        The Lie-Taylor majorant $\Maj_{\phi}(\supnorm{\underline{z}})$
        is a majorant of $\Taylor_{\phi}(\underline{z};e)$, i.e.
        \begin{equation}
            \label{eq:LT-Maj}
            \abs[\big]
            {\Taylor_\phi(\underline{z}; \E)}
            \le
            \Maj_{\phi}
            \big(
                \supnorm{\underline{z}}
            \big)
        \end{equation}
        for all $\underline{z} \in \field{C}^n$. In particular, if
        $\Maj_{\phi} \in \Holomorphic(\Ball_r(0))$, i.e.~$\Maj_{\phi}$
        is holomorphic on the open disk
        $\Ball_r(0) = \{z \in \field{C} \; \big| \; \abs{z} < r\}$,
        then $\Taylor_\phi(\argument; \E)$ is holomorphic on the
        polydisk $\Ball_r(0)^n \subseteq \mathbb{C}^n$.
    \item \label{item:LieTaylorMajorantBasis}%
        The coefficients $c_k(\phi)$ and hence $\Maj_{\phi}$ depend on
        the choice of the basis $\mathcal{B}$, so one should tend to
        write $c_{k, \mathcal{B}}(\phi)$ instead. However, if
        $\mathcal{B}'$ is another basis of $\liealg{g}$, then it is
        easy to see that there is a constant
        $M = M(\mathcal{B}, \mathcal{B}') > 0$ such that
        \begin{equation}
            \label{eq:LieTaylorMajorantBasis}
            c_{k, \mathcal{B}'}
            (\phi)
            \le
            M^k
            \cdot
            c_{k, \mathcal{B}}(\phi)
        \end{equation}
        for all $k \in \field{N}_0$ and $\phi \in \Cinfty(G)$. In
        particular, if $\Maj_{\phi} \in \Holomorphic(\field{C})$
        w.r.t.~some basis of $\liealg{g}$, then
        $\Maj_{\phi} \in \Holomorphic(\field{C})$ w.r.t.~any
        basis. The upshot is that whenever we are dealing with
        \emph{entire} Lie-Taylor majorants or only care about
        analyticity with no specific radius of convergence, we can
        safely ignore the specific choice of the basis of
        $\liealg{g}$.
    \end{remarklist}
\end{remark}
The following simple observations will prove very useful:
\begin{proposition}[Leibniz and chain rule]
    \label{proposition:LeibnizAndChain}%
    Let $\phi \in \Cinfty(G)$ and $z \in \field{C}$.
    \begin{propositionlist}
    \item \label{item:Leibniz}%
        Let $\psi \in \Cinfty(G)$ be another smooth function. We have
        the Leibniz inequality
        \begin{equation}
            \label{eq:Leibniz}
            \abs[\big]
            {\Maj_{\phi \cdot \psi}(z)}
            \le
            \Maj_{\phi}\big(\abs{z}\big)
            \cdot
            \Maj_{\psi}\big(\abs{z}\big).
        \end{equation}
    \item \label{item:ChainRule}%
        Let $\Phi \colon G \longrightarrow H$ be a morphism of Lie
        groups. Then
        \begin{align}
            \label{eq:ChainRuleCoefficients}
            c_k
            \big(
                \Phi^* \phi
            \big)
            &\le
            (D n)^k
            \cdot
            c_k(\phi) \\
            \label{eq:ChainRule}
            \textrm{and} \quad
            \abs[\big]
            {\Maj_{\Phi^* \phi}(z)}
            &\le
            \Maj_\phi
            \big(
                Dn \cdot \abs{z}
            \big),
        \end{align}
        where $D$ is the matrix supnorm of the matrix representation
        of the tangent map
        $T_{\E} \Phi \colon \liealg{g} \longrightarrow \liealg{h}$ in
        the bases used for the construction of the Taylor majorants.
    \end{propositionlist}
\end{proposition}
\begin{proof}
    The Leibniz rule~\ref{item:Leibniz} is an easy consequence of the
    noncommutative higher Leibniz rule \eqref{eq:HigherLeibniz}, the
    Cauchy product formula and the triangle inequality. For
    \ref{item:ChainRule} recall that for $\xi \in \liealg{g}$, the
    left invariant vector fields $X^G_\xi$ and
    $X^H_{T_{\E_G} \Phi \xi}$ are $\Phi$-related. This implies for the
    corresponding Lie derivatives
    \begin{equation*}
        \Lie_{X^G_\xi}
        \big(
            \Phi^* \phi
        \big)
        =
        \Phi^*
        \big(
            \Lie_{X^H_{T_\E \Phi \xi}}
            \phi
        \big).
    \end{equation*}
    As in the formulation, we set
    \begin{equation*}
        D
        =
        \max_{i, j=1, \ldots, n}
        \abs[\big]
        {(T_{\E} \Phi)_{i}^j},
    \end{equation*}
    where we take matrix representation $(T_{\E} \Phi)_{i}^j = d_i^j$
    of $T_{\E} \Phi \colon \liealg{g} \longrightarrow \liealg{h}$ with
    respect to the chosen bases. Thus we obtain for
    $k \in \field{N}_0$
    \begin{align*}
        c_k
        \big(
            \Phi^* \phi
        \big)
        &=
        \frac{1}{k!}
        \sum_{\alpha \in \field{N}_n^k}
        \abs[\Big]
        {\big( \Lie_{X_{\alpha}}^G \Phi^* \phi \big)(\E)} \\
        &=
        \frac{1}{k!}
        \sum_{\alpha \in \field{N}_n^k}
        \abs[\Big]
        {
          \Phi^*
          \big(
          \Lie_{X^H_{T_\E \Phi(\basis{e}_{\alpha_1})}}
          \cdots
          \Lie_{X^H_{T_\E \Phi(\basis{e}_{\alpha_k})}}
          \phi
          \big)
          \at{\E}
        } \\
        &=
        \frac{1}{k!}
        \sum_{\alpha \in \field{N}_n^k}
        \abs[\Big]
        {
            d_{\alpha_1}^{j_1} \cdots d_{\alpha_k}^{j_k}
            \big(
                \Lie_{X^H_{j_1}} \cdots \Lie_{X^H_{j_k}}
                \phi
            \big)
            \at{\Phi(\E)}
        } \\
        &\le
        \frac{(Dn)^k}{k!}
        \sum_{\beta \in \field{N}_n^k}
        \abs[\Big]
        {
            \big(
                \Lie_{X_{\beta}}^H
                \phi
            \big)
            (\E)
        } \\
        &=
        (Dn)^k c_k(\phi),
    \end{align*}
    where we wrote $\beta = (j_1, \ldots, j_k)$. This implies
    \eqref{eq:ChainRule} at once.
\end{proof}

To understand the representation \eqref{eq:LieDerivativeRep} it is
essential to estimate the Lie-Taylor majorant of Lie-derivatives
$\Lie_{X_{\xi}} \phi$ in terms of the formal ``complex'' derivative
\begin{equation}
    \label{eq:TaylorMajorantDerivative}
    \Maj_{\phi}'(z)
    =
    \sum \limits_{k=0}^{\infty}
    (k+1)
    c_{k+1}(\phi)
    \cdot
    z^k
\end{equation}
of the Lie-Taylor majorant $\Maj_{\phi}$ of $\phi$. Such an estimate
is provided by the following result. Here and in what follows we
slightly abuse notation and denote by $\supnorm{\xi}$ the supnorm of
the coordinate vector $(\xi_1, \ldots, \xi_n) \in \mathbb{C}^n$ of
$\xi \in \liealg{g}$ w.r.t.~the basis $\mathcal{B}$.
\begin{proposition}
    \label{proposition:TaylorMajorantDerivative}%
    Let $\xi \in \liealg{g}$, $k \in \field{N}_0$,
    $\phi \in \Cinfty(G)$ and $z \in \field{C}$. We have the estimates
    \begin{align}
        \label{eq:TaylorMajorantDerivativeCoefficients}
        c_k
        \big(
            \Lie_{X_{\xi}} \phi
        \big)
        &\le
        \supnorm{\xi}
        \cdot
        (k+1)
        c_{k+1}(\phi) \\
        \shortintertext{and}
        \label{eq:TaylorMajorantDerivativeEstimate}
        \abs[\big]
        {\Maj_{\Lie_{X_{\xi}} \phi}(z)}
        &\le
        \supnorm{\xi}
        \cdot
        \Maj'_{\phi}
        \big(
            \abs{z}
        \big).
  \end{align}
\end{proposition}
\begin{proof}
    The estimate \eqref{eq:TaylorMajorantDerivativeCoefficients} is
    just
    \begin{align*}
        c_k
        \big(
            \Lie_{X_{\xi}}\phi
        \big)
        =
        \frac{1}{k!}
        \sum \limits_{\alpha \in \mathbb{N}_n^k}
        \abs[\Big]
        {\big(
            \Lie_{X_{\alpha}}
            \Lie_{X_{\xi}}
            \phi
        \big)(\E)}
        \le
        \frac{\supnorm{\xi}}{k!}
        \sum \limits_{\alpha \in \mathbb{N}_n^{k+1}}
        {\big(
            \Lie_{X_{\alpha}}
            \phi
        \big)(\E)}
        \le
        \supnorm{\xi}
        \cdot
        (k+1)
        c_{k+1}(\phi),
    \end{align*}
    from which \eqref{eq:TaylorMajorantDerivativeEstimate} follows at
    once via \eqref{eq:TaylorMajorantDerivative}.
\end{proof}

We now shift our focus to the set $\Comega(G)$ of \textit{(real)
  analytic} functions on $G$. Here we can say more and warrant our
terminology from Definition~\ref{definition:LieTaylor}. For a more
conceptual understanding, we recall the classical concept of
Lie-Taylor series, which can be found e.g. in the beginning of
\cite[Sec.~2.1.4]{helgason:2001a} for Lie groups and
\cite[(1.48)]{forstneric:2011a} for arbitrary analytic manifolds:
\begin{proposition}[Lie-Taylor]
    \label{proposition:LieTaylorGeneral}%
    Let $M$ be an analytic manifold, $\phi \in \Comega(M)$ and
    $X \in \Gamma^\omega(M)$ an analytic vector field with
    corresponding flow $\Phi$.
    \begin{propositionlist}
    \item \label{item:LieTaylor}%
        Given $p \in M$, there exists a parameter $r > 0$ such that
        the Lie-Taylor formula
        \begin{equation}
            \label{eq:LieTaylorGeneral}
            \phi
            \big(
                \Phi(t, p)
            \big)
            =
            \big(
                \exp( t \Lie_X )
                \phi
            \big)
            (p)
        \end{equation}
        holds, whenever $\abs{t} < r$.
    \item \label{item:LieTaylorCinftyConvergence}%
        The series \eqref{eq:LieTaylorGeneral} is
        $\Cinfty\big(\Ball_r(0)\big)$-convergent in the parameter $t$.
    \item \label{item:LieTaylorExponential}%
        Given a Lie algebra element $\xi \in \liealg{g}$, we get a
        well-defined exponential operator
        \begin{equation}
            \label{eq:LieTaylorExponential}
            \exp \big( \Lie_{X_\xi} \big)
            \colon
            \Comega(G)
            \longrightarrow
            \Comega(G).
        \end{equation}
    \end{propositionlist}
\end{proposition}
\begin{proof}
    We first compute how powers of $\Lie_X$ act on $\phi$. To this
    end, we rewrite
    \begin{equation*}
        \Lie_X
        \phi
        \at[\Big]{\Phi(t, p)}
        =
        \frac{\D}{\D s}
        \Phi
        \big(
        s, \Phi(t, p)
        \big)^{*}
        \phi
        \at[\Big]{s = 0}
        =
        \frac{\D}{\D s}
        \Phi(s+t, p)^*
        \phi
        \at[\Big]{s = 0}
        =
        \frac{\D}{\D t}
        \phi\big(\Phi(t, p)\big),
    \end{equation*}
    which we can now iterate easily. This yields
    \begin{equation}
        \label{eq:LiePowers}
        \Lie_X^k
        \phi
        \at[\Big]{\Phi(t, p)}
        =
        \frac{\D^k}{\D t^k}
        \phi\big(\Phi(t, p)\big)
    \end{equation}
    for all $k \in \field{N}_0$. Note that these considerations also
    work for smooth functions. The case $X(p) = 0$ is special, so we
    deal with it first: then $\Lie_X \phi = X(p) \phi = 0$ and thus
    the right hand side of \eqref{eq:LieTaylor} reduces to the
    constant term, which is just $\phi(p)$. From the differential
    equation it is moreover clear that $\Phi(t, p) \equiv p$ in this
    case, so the left hand side matches. Thus the interesting case
    $X(p) \neq 0$ remains. Here we finally use the analyticity to
    obtain a chart $(U, x)$ of $M$ centered at $p$ such that the
    function
    \begin{equation*}
        \psi
        \colon
        x(U)
        \longrightarrow
        \field{K}, \quad
        \psi
        =
        \phi
        \circ
        x
    \end{equation*}
    is given by its power series around $0$ on all of
    $x(U) \subseteq \field{K}^n$ as well as
    $x^{-1}(t \basis{e}_1) = \Phi(t, p)$ for $t$ with
    $t \basis{e}_1 \in x(U)$. The latter condition is achievable by
    \cite[Lem.~1.9.2]{forstneric:2011a}, which yields an analytic
    chart if we go through its construction starting with an analytic
    chart as well as using an analytic vector field: having a chart
    with $X \at{U} = \partial_1$ then gives the differential equation
    \begin{equation*}
        \frac{\D}{\D t}
        \Phi(t, p)
        =
        \frac{\partial}{\partial x^1}
        \big(
        \Phi(t, p)
        \big),
    \end{equation*}
    which applied to a function on $U$ gets solved by
    $\Phi(t, p) = x^{-1}(t\basis{e}_1)$. Note that the left hand side
    acts by pullback here. Thus by uniqueness of solutions our
    condition indeed holds. Let now
    $\big( \basis{e}_1, \ldots, \basis{e}_n \big)$ be a basis of
    $\field{K}^n$. As $x(U)$ is open, we find an $r > 0$ with
    $\Ball_r(0) \subseteq x(U)$ for some auxiliary norm on
    $\field{K}^n$. For $\abs{t} < r$ we have the simple Taylor
    expansion
    \begin{align*}
        \psi( t \basis{e}_1)
        &=
        \sum_{k = 0}^\infty
        \frac{t^k}{k!}
        \big( \partial_1^k \psi \big)
        (0)
        \\
        &=
        \sum_{k = 0}^\infty
        \frac{t^k}{k!}
        \frac{\D^k}{\D s^k}
        \phi\big(x(s \basis{e}_i)\big)
        \at[\Big]{s = 0}
        \\
        &=
        \sum_{k = 0}^\infty
        \frac{t^k}{k!}
        \frac{\D^k}{\D s^k}
        \phi\big(\Phi(s, p)\big)
        \at[\Big]{s = 0}
        \\
        &=
        \sum_{k = 0}^\infty
        \frac{t^k}{k!}
        \Lie_X^k
        \phi
        \at[\Big]{\Phi(0, p)}
        \\
        &=
        \sum_{k = 0}^\infty
        \frac{t^k}{k!}
        \Lie_X^k
        \phi
        \at[\Big]{p},
    \end{align*}
    which is exactly \eqref{eq:LieTaylorGeneral}. This calculation gives also
    the uniform convergence statement: the $k$-th partial sum on the
    right hand side of \eqref{eq:LieTaylorGeneral} exactly corresponds to the
    $k$-th partial sum of the classical Taylor series. As the Taylor
    series is $\Cinfty$-convergent in the interior of the polydisk of
    convergence, so is \eqref{eq:LieTaylorGeneral} in $t$. The statement
    \ref{item:LieTaylorExponential} follows immediately from reading
    \eqref{eq:LieTaylorGeneral} backwards.
\end{proof}

Taking $M = G$ as a Lie group and $X = X_\xi$ as a left invariant
vector field for some $\xi \in \liealg{g}$ gives the flow
$\Phi(t, g) = g \exp(t \xi)$, which notably also yields a suitable
chart centered at $g$, as we just used in the proof. The Lie-Taylor
series then has the form
\begin{equation}
    \label{eq:LieTaylorSingleVariable}
    \big(
        \phi \circ r_{\exp(t\xi)}
    \big)(g)
    =
    \phi
    \big(
        g \exp(t \xi)
    \big)
    =
    \big(
        \exp( t \Lie_{X_\xi} )
        \phi
    \big)
    (g)
\end{equation}
for all $g \in G$ and sufficiently small $t \in \field{R}$. It
moreover coincides with the Taylor series of
$\phi \circ \ell_g \circ \exp$ on the ray through $0$ in direction of
$\xi$. Taking now $\xi = x^j \basis{e}_j$ as a basis decomposition
yields the following:
\begin{corollary}[Lie-Taylor series on $G$]
    \label{corollary:LieTaylor}%
    Let $\phi \in \Comega(G)$ and $g \in G$. Then there is a radius
    $r>0$ such that
    \begin{equation}
        \label{eq:LieTaylor}
        \phi
        \big(
            g \exp( x^j \basis{e}_j)
        \big)
        =
        \Taylor_{\phi}(\underline{x}; g)
    \end{equation}
    for all $\underline{x}=(x^1, \ldots, x^n) \in \field{R}^n$ with
    $\supnorm{\underline{x}} < r$. The right hand side of
    \eqref{eq:LieTaylor} is $\Cinfty$-convergent, whenever it
    converges at all. In particular, if the Lie-Taylor majorant
    $\Maj_{\phi}$ of $\phi$ is an entire holomorphic function, then
    \begin{equation}
        \label{eq:LieTaylorHolomorphicExtension}
        \field{R}^n
        \longrightarrow
        \mathbb{C}, \quad
        \underline{x}
        \mapsto
        \phi
        \big(
            g \exp (x^j \basis{e}_j)
        \big)
    \end{equation}
    has a holomorphic extension to $\field{C}^n$, which is provided by
    the Lie-Taylor series $\Taylor_{\phi}(\argument; g)$ of $\phi$ at
    $g$.
\end{corollary}

The next result has a similar flavor as
Proposition~\ref{proposition:TaylorMajorantDerivative} and roughly
ensures that the concept of Lie-Taylor majorants is also compatible
with the exponentiated action \eqref{eq:LieTaylorExponential} of $G$
on $\Cinfty(G)$ by pullbacks with right multiplications $r_g(h) = hg$.
\begin{proposition}
    \label{prop:Translation}%
    Let $\xi \in \liealg{g}$, $\phi \in \Comega(G)$ and
    $z \in \field{C}$. Then
    \begin{equation}
        \label{eq:Translation}
        \abs[\big]
        {\Maj_{\phi \circ r_{\exp(\xi)}}(z)}
        \le
        \Maj_{\phi}
        \big(
            \abs{z}
            +
            \supnorm{\xi}
        \big).
    \end{equation}
\end{proposition}
\begin{proof}
    Using \eqref{eq:LieTaylorSingleVariable} and
    applying \eqref{eq:TaylorMajorantDerivativeCoefficients}
    repeatedly yields
    \begin{equation*}
        c_k
        \big(
            \phi \circ r_{\exp{\xi}}
        \big)
        =
        c_k
        \big(
            \exp(\Lie_{X_{\xi}})
            \phi
        \big)
        \le
        \sum \limits_{\ell=0}^{\infty}
        \frac{1}{\ell!}
        c_k
        \big(
            \Lie_{X_{\xi}}^\ell \phi
        \big)
        \le \sum \limits_{\ell=0}^{\infty}
        \frac{(k+\ell)!}{k! \; \ell!}
        c_{k+\ell}(\phi)
        \supnorm{\xi}^\ell.
    \end{equation*}
    Consequently,
    \begin{align*}
        \abs[\big]
        {\Maj_{\phi \circ r_{\exp{\xi}}}(z)}
        &=
        \sum \limits_{k=0}^{\infty}
        c_k
        \big(
            \phi \circ r_{\exp{\xi}}
        \big)
        \cdot
        \abs{z}^k \\
        &\le
        \sum \limits_{\ell=0}^{\infty}
        \frac{\supnorm{\xi}^\ell }{\ell!}
        \sum \limits_{k=0}^{\infty}
        \frac{(k+\ell)!}{k!}
        c_{k+\ell}(\phi)
        \cdot
        \abs{z}^k \\
        &=
        \sum \limits_{\ell=0}^{\infty}
        \frac{\supnorm{\xi}^\ell}{\ell!}
        \Maj_{\phi}^{(\ell)}\big(\abs{z}\big) \\
        &\overset{(\star)}{=}
        \Maj_{\phi}
        \big(
            \abs{z} + \supnorm{\xi}
        \big),
    \end{align*}
    provided that $\Maj_{\phi} (\abs{z} + \supnorm{\xi}) < \infty$. In
    fact, in this case $\Maj_{\phi}$ is holomorphic on the disk
    $\Ball_{\abs{z}+\supnorm{\xi}}(0)$, so we can expand $\Maj_{\phi}$
    as a Taylor series around $\abs{z}$ in the disk
    $\Ball_{\supnorm{\xi}}(\abs{z})$, at least. This proves ($\star$)
    with $\supnorm{\xi}$ replaced by $r \supnorm{\xi}$ for any
    $0<r<1$. Letting $r \to 1$ gives ($\star$).  If
    $\Maj_{\phi}\big(\abs{z} + \supnorm{\xi}\big) = \infty$, then the
    estimate is certainly true.
\end{proof}

Intuitively, \eqref{eq:Translation} lets us estimate Taylor expansions
of analytic functions with perturbed expansion point on the group by
perturbing the expansion point on the Lie algebra.
\begin{lemma}[Inversion Invariance of Lie-Taylor majorants]
    \label{lemma:Inversion}%
    Let $\inv \colon G \longrightarrow G$ denote group inversion on
    $G$. Then $\Maj_{\phi \circ \inv} = \Maj_{\phi}$ for all
    $\phi \in \Cinfty(G)$.
\end{lemma}
\begin{proof}
    This is immediate from
    \begin{equation*}
        \Lie_{X_\xi}
        (\phi \circ \inv)
        \at[\Big]{\E}
        =
        \frac{\D}{\D t}
        \big(
        \phi \circ \inv
        \big)
        (\E \cdot \exp(t\xi))
        \at[\Big]{t=0}
        =
        \frac{\D}{\D t}
        \phi
        \big(
            \exp(-t \xi)
        \big)
        \at[\Big]{t=0}
        =
        -
        \Lie_{X_\xi} \phi
        \at[\Big]{\E}
    \end{equation*}
    for $\xi \in \liealg{g}$ by  definition of the Lie
    derivative.
\end{proof}
\begin{corollary}
    \label{corollary:LieTaylorTranslation}%
    There is a locally constant function
    \begin{equation*}
        \gamma
        \colon
        G
        \longrightarrow
        [0, \infty)
    \end{equation*}
    such that for any $\phi \in \Comega(G)$ and any
    $\underline{z} \in \field{C}^n$ the Lie-Taylor majorant
    $\Maj_{\phi}(\supnorm{\underline{z}} + \gamma(g))$ is a majorant
    for the Lie-Taylor series $\Taylor_{\phi}(\underline{z}; g)$ of
    $\phi$ at $g$ evaluated at $\underline{z} \in \mathbb{C}^n$ and
    hence
    \begin{equation}
        \label{eq:LieTaylorTranslation}
        \abs[\big]
        {\Taylor_{\phi}
        (\underline{z};g)}
        \le
        \Maj_{\phi}
        \big(
            \supnorm{\underline{z}}
            +
            \gamma(g)
        \big).
    \end{equation}
\end{corollary}
\begin{proof}
    We denote left multiplication with $g$ by $\ell_g$, as usual.  The
    left invariance of $X_\xi$ gives
    $\Lie_{X_\xi} \circ \ell_g^* = \ell_g^* \circ
    \Lie_{X_\xi}$. Applying this to a function $\phi$ and evaluating
    at $e \in G$ gives
    \begin{equation}
        \label{eq:LieTaylorTranslationProof}
        \Taylor_{\phi}
        (\underline{z}; g)
        =
        \Taylor_{\phi \circ \ell_g}
        (\underline{z}; \E).
        \tag{$*$}
    \end{equation}
    Now
    $\phi \circ \ell_g = \phi \circ \mathord{\inv} \circ r_{g^{-1}}
    \circ \mathord{\inv}$, so Proposition~\ref{lemma:Inversion}
    implies
    \begin{equation*}
        \Maj_{\phi \circ \ell_g}(z)
        =
        \Maj_{\phi \circ \mathord{\inv} \circ r_{g^{-1}}}(z).
    \end{equation*}
    We now choose $\xi_1, \ldots, \xi_m \in \liealg{g}$ with
    $\supnorm{\xi_j} \le 1$ and
    $g^{-1} = \exp(\xi_1) \cdots \exp(\xi_m)$. By openness of the Lie
    exponential the integer $m$ can be chosen in a locally constant
    manner, i.e. there is an open neighbourhood $U$ of $g^{-1}$
    s.t. each $h \in U$ can be written as a product of $m$
    exponentials. Then applying Proposition~\ref{prop:Translation}
    $m$-times and once again Proposition~\ref{lemma:Inversion} yields
    \begin{equation*}
        \abs[\big]{\Maj_{\phi \circ \ell_g}(z)}
        =
        \abs[\big]{\Maj_{\phi \circ \inv \circ r_{g^{-1}}}(z)}
        \le
        \Maj_{\phi \circ \inv}
        \big(
            \abs{z} + m
        \big)
        =
        \Maj_{\phi}
        \big(
            \abs{z} + m
        \big).
    \end{equation*}
    Combining this with \eqref{eq:LieTaylorTranslationProof} and
    \eqref{eq:LT-Maj} completes the proof.
\end{proof}

%
%

\subsection{Entire Functions on $G$}

In this subsection we first focus on the case $R=0$ and introduce the
pendant $\Entire[0](G)$ of the optimal case from the strict
deformation \cite[Prop.~3.2,
\textit{ii.)}]{esposito.stapor.waldmann:2017a},
i.e. $\SymR(\liealg{g})$ with $R' = 1$. In a second step, we then
introduce the algebras $\Entire(G)$ for $R>0$.  There are several
reasons for this two--step approach. Firstly, our methods in the case
$R=0$ seem completely natural and do not call for any specific
motivation. Secondly, it puts us in a position to reintroduce the
classical notion of an entire vector for the lifted Lie algebra
representation \eqref{eq:LieDerivativeRep}. While our approach to this
notion is novel, the locally convex space we are about to consider is
not. We will make this and its history precise in
Remark~\ref{remark:RepresentationTheory}. Thirdly, the construction
for the case $R=0$ provides a solid motivation for the cases $R>0$,
since it makes clear that we need to identify locally convex algebras
of entire functions with controlled growth, whose topology is finer
than that of locally uniform convergence, but which still are
invariant with respect to differentiation and translation in the
argument.

\begin{definition}[Entire functions on $G$]
    \label{definition:Entire}%
    An analytic function $\phi \in \Comega(G)$ is called an entire
    function on $G$ if its Lie-Taylor majorant
    $\Maj_{\phi} \in \Holomorphic(\field{C})$ is entire. By
    \begin{equation}
        \label{eq:Entire}
        \Entire[0](G)
        =
        \big\{
            \phi \in \Comega(G)
            \; \big| \;
            \Maj_{\phi} \in \Holomorphic(\field{C})
        \big\}
    \end{equation}
    we denote the set of all entire functions on $G$.
\end{definition}
In particular, each element of $\Entire[0](G)$ is analytic by
definition, so it does have a local Lie-Taylor series representation
by Corollary~\ref{corollary:LieTaylor}. Hence by connectedness of $G$
it follows that
\begin{equation}
    \label{eq:hausdorff}
    \phi = 0
    \quad \iff \quad
    \Maj_{\phi} = 0.
\end{equation}

Recall that the $\field{C}$-vector space $\Holomorphic(\field{C})$
carries a canonical topology, namely the compact-open topology (or
topology of locally uniform convergence). This locally convex topology
can be induced by the family of norms
\begin{equation}
    \label{eq:CSeminorm}
    \norm{F}_{0, c}
    =
    \max \limits_{\abs{z} \le c}
    \abs[\big]{F(z)},
\end{equation}
and is metrizable in a translation-invariant manner by
\begin{equation*}
    \label{eq:CMetric}
    d(F, G)
    =
    d(F - G, 0)
    =
    \sum \limits_{j=1}^{\infty}
    \frac{1}{2^j}
    \frac{\norm{F-G}_{0, j}}
    {1 + \norm{F-G}_{0, j}}
\end{equation*}
for $F, G \in \Holomorphic(\field{C})$. It is well-known that
$(\Holomorphic(\field{C}), d)$ is then a multiplicatively convex
(commutative) nuclear Fr\'echet algebra w.r.t.~pointwise
multiplication. We are thus naturally led to define a metric $d_0$ on
the vector space $\Entire[0](G)$ by
\begin{equation}
    \label{eq:EntireMetric}
    d_0(\phi,\psi)
    =
    d(\Maj_{\phi-\psi}, 0)
\end{equation}
for $\phi, \psi \in \Entire[0](G)$. An associated family of seminorms
is given by
\begin{equation}
    \label{eq:EntireSeminorm}
    \seminorm{q}_{0, c}
    (\phi)
    =
    \norm{\Maj_\phi}_{0, c}
    =
    \max \limits_{\abs{z} \le c}
    \abs[\big]
    {\Maj_{\phi}(z)}
    =
    \Maj_{\phi}(c)
    =
    \sum \limits_{k=0}^{\infty} c_k(\phi) \, c^k
    =
    \sum \limits_{k=0}^{\infty}
    \frac{c^k}{k!}
    \sum \limits_{\alpha \in \mathbb{N}_n^k}
    \abs[\Big]{\big( \Lie_{X_{\alpha}} \phi \big)(\E)}
\end{equation}
with $c > 0$. Note that (\ref{eq:hausdorff}) ensures that each
$\seminorm{q}_{0, c} \colon \Entire[0](G) \longrightarrow \field{R}$
is in fact a \emph{norm} on $\Entire[0](G)$. By the Leibniz inequality
from Proposition~\ref{proposition:LeibnizAndChain},
\ref{item:Leibniz}, the norms \eqref{eq:EntireSeminorm} are
submultiplicative. In particular, $(\Entire[0](G), d_0)$ is a locally
multiplicatively convex algebra w.r.t.~to pointwise multiplication.

We now introduce the family of subspaces $\Entire(G)$, $R>0$, of the
algebra $\Entire[0](G)$ of all entire functions on $G$, which will
serve as the other tensor factors in the observable algebra. The idea
is to restrict the Lie-Taylor majorants
$\Maj_{\phi} \in \Holomorphic(\mathbb{C})$ to holomorphic entire
functions of fixed \textit{finite order and minimal type}. The
intimate reason for this choice is that these functions and thus their
Taylor coefficients do have controlled growth and they form a
particularly well-studied subalgebra of $\Entire[0](G)$ with the added
feature of being invariant w.r.t.~to \textit{differentiation} and
\textit{translation}, i.e. both the Lie algebra and the Lie group
representations. This puts us in a position to make use of Proposition
\ref{proposition:TaylorMajorantDerivative} and Proposition
\ref{prop:Translation} and will directly lead us to faithful analogues
of those algebras on the Lie group $G$.  Recall that for any $\rho>0$
a function $F \in \Holomorphic(\mathbb{C})$ is said to have
\textit{finite order $\le \rho$ and minimal type} if
\begin{equation}
    \label{eq:Entire2}
    \sup \limits_{z \in \field{C}}
    \abs[\big]{F(z)}
    \exp
    \big(
        - \eps \abs{z}^{\rho}
    \big)
    <
    \infty
\end{equation}
for every $\epsilon > 0$. We denote the set of all such functions by
$\Holomorphic_{\rho}(\field{C})$.  Note that for $\rho \le 1$ we are
speaking of entire functions of exponential type zero.
\begin{definition}[$R$-entire functions]
    Let $R>0$. A function $\phi \in \Comega(G)$ is called an
    $R$-entire function if $\Maj_{\phi} \in \Holomorphic(\field{C})$
    has finite order $\le \tfrac{1}{R}$ and minimal type. We denote
    the set of all $R$-entire functions by $ \Entire(G)$.
\end{definition}

Unwrapping the definition, we thus have
\begin{equation}
    \label{eq:seminorm2}
    \phi \in \Entire(G)
    \quad \iff \quad
    \forall_{\eps>0}
    \quad
    \norm{\phi}_{R, \eps}
    =
    \sup \limits_{z \in \field{C}}
    \abs[\big]
    {\Maj_{\phi}(z)}
    \exp
    \big(
        -\eps \abs{z}^{1/R}
    \big)
    <
    \infty.
\end{equation}
\begin{remark}
    \label{remark:Entire}%
    \begin{remarklist}
    \item \label{item:EntireInclusions}%
        Clearly, each of the sets $\Entire(G)$ is a unital subalgebra
        of $\Entire[0](G)$ and we have the inclusions
        \begin{equation}
            \label{eq:EntireInclusions}
            \Entire(G)
            \subseteq
            \Entire[S](G),
        \end{equation}
        whenever $S \le R$.
    \item \label{item:EntireMajorant}%
        Let $\phi \in \Entire(G)$ for some $R>0$. Then the Lie-Taylor
        series $\Taylor_{\phi}(\argument; \E)$ of $\phi$ is an entire
        holomorphic function on $\field{C}^n$ of order
        $\le \tfrac{1}{R}$ and minimal type. We denote the set of such
        functions by $\Holomorphic_{1/R}(G)$ and equip it with the
        family of norms \eqref{eq:seminorm2}. For a general treatment
        we refer e.g. to the textbook \cite{gruman.lelong:1986a}.
     \end{remarklist}
\end{remark}

It is well-known (see e.g. \cite[Prop.~4.2]{meise:1985a}) that
$\Holomorphic_{\rho}(\mathbb{C})$ equipped with the family of norms
\eqref{eq:Entire2} is a nuclear Fr\'echet space. It follows at once
from \eqref{eq:seminorm2} that for $\eps > 0$
\begin{equation}
    \label{eq:praesubmultiplikativ}
    \norm[\big]
    {\phi \cdot \psi}_{R, \eps}
    \le
    \norm{\phi}_{R, \eps/2}
    \cdot
    \norm{\psi}_{R, \eps/2},
\end{equation}
so $(\Entire(G), d_R)$ is a locally convex commutative algebra
w.r.t.~pointwise multiplication. Note, however, that $\Entire(G)$ is
\emph{not} multiplicatively convex as soon as $R > 0$. In fact, it can
be easily shown that $\Entire(G)$ has no entire holomorphic functional
calculus.

It will turn out convenient to introduce an equivalent family of
(semi)norms on $\Entire(G)$. As in the classical case of entire
holomorphic functions of finite order this is achieved by relating the
order of an analytic function $\phi \in \Comega(G)$ to the growth of
its Taylor coefficients:
\begin{definition}[$R$-Lie-Taylor majorant]
    Let $R \ge 0$ and $\phi \in \Cinfty(G)$. Then we call
    \begin{equation}
        \label{eq:LieTaylorMajorantR}
        \Maj_{R, \phi}(z)
        =
        \sum \limits_{k=0}^{\infty}
        k!^R
        c_k(\phi)
        \cdot
        z^k
        =
        \sum \limits_{k=0}^{\infty}
        k!^{R-1}
        \sum \limits_{\alpha \in \mathbb{N}_n^k}
        \abs[\Big]{\big( \Lie_{X_{\alpha}} \phi \big)(\E)}
        z^k
    \end{equation}
    the $R$-Lie-Taylor majorant of $\phi$ and define for any $c \ge 0$
    \begin{equation}
        \label{eq:EntireSeminormR}
        \seminorm{q}_{R, c}
        (\phi)
        =
        \Maj_{R, \phi}(c)
        =
        \sum \limits_{k=0}^{\infty}
        k!^R
        \cdot
        c_k(\phi)
        \cdot
        c^k
        =
        \sum \limits_{k=0}^{\infty}
        k!^{R-1}
        \sum \limits_{\alpha \in \mathbb{N}_n^k}
        \abs[\Big]{\big( \Lie_{X_{\alpha}} \phi \big)(\E)}
        c^k.
    \end{equation}
\end{definition}

The following elementary result tells us that for a function
$\phi \in \Comega(G)$ membership in $\Entire(G)$ can be checked using
the $R$-Lie-Taylor majorant $\Maj_{R, \phi}$ and the seminorms
$\seminorm{q}_{R,c}$:
\begin{equation}
    \label{eq:aequiSemiNorms}
    \phi \in \Entire(G)
    \qquad \iff \qquad
    \Maj_{R, \phi} \in \Holomorphic(\field{C})
    \qquad \iff \qquad
    \forall_{c>0} \,\, \seminorm{q}_{R,c}(\phi) < \infty,
\end{equation}
and also that
\begin{equation}
    \label{eq:aequiSemiNormsConvergence}
    d_R(\phi_n,\phi) \to 0
    \qquad \iff \qquad
    \Maj_{R,\phi_n-\phi}
    \to 0 \; \textrm{in} \; \Holomorphic(\field{C}).
\end{equation}
\begin{remark}
    It seems (and perhaps is) over the top to introduce two different
    notions, the seminorms $\seminorm{q}_{R, c}$ and the
    $R$-Lie-Taylor majorant $\Maj_{R, \phi}$, for essentially the same
    object, namely $\Maj_{R, \phi}(c) = \seminorm{q}_{R,
      c}(\phi)$. However, it emphasizes that
    $\seminorm{q}_{R,c}(\phi)$ simply \textit{is} a
    \textit{holomorphic} function of one \textit{complex} variable
    $c$, and this point of view brings along some useful tools such as
    the Cauchy integral formula.
\end{remark}
\begin{proposition}
    \label{proposition:REntire}%
    Let $\phi, \psi \in \Cinfty(G)$. Then
    \begin{propositionlist}
    \item \label{item:EntireVsTaylor}%
        for any $R > 0$ and $\eps>0$,
        \begin{equation}
            \label{eq:EntireVsTaylor}
            \norm{\phi}_{R, \eps}
            \le
            \seminorm{q}_{R, (R/\eps)^R}(\phi);
        \end{equation}
    \item \label{item:TaylorVsEntire}%
        for any $R > 0$, $c > 0$ and $\eps > 0$ such that
        $c \cdot \big( \frac{\E \eps}{R} \big)^{R} < 1$,
        \begin{equation}
            \label{eq:TaylorVsEntire}
            \seminorm{q}_{R,c}(\phi)
            \le
            \frac{\norm{\phi}_{R,\eps}}
            {1 - c \cdot ( \tfrac{\E \eps}{R} )^{R}};
        \end{equation}
    \item \label{item:EntireMultiplication}%
        for any $R \ge 0$ and $c>0$
        \begin{equation}
            \label{eq:EntireMultiplication}
            \seminorm{q}_{R, c}(\phi \cdot \psi)
            \le
            \seminorm{q}_{R, 2^R c}(\phi)
            \cdot
            \seminorm{q}_{R, 2^R c}(\psi).
        \end{equation}
        Moreover, naive extension of \eqref{eq:EntireSeminormR} to
        $R < 0$ yields submultiplicative seminorms
        $\seminorm{q}_{R, c}$.
    \item \label{item:EntireRRepresentations}%
        for any $R \ge 0$, $\xi \in \liealg{g}$ and $z \in \field{C}$
        \begin{align}
            \label{eq:EntireRRepresentationsGroup}
            \abs[\big]
            {\Maj_{R, \phi \circ r_{\exp{\xi}}}(z)}
            &\le
            \Maj_{R, \phi}
            \big(
            \abs{z} + \supnorm{\xi}
            \big) \\
            \shortintertext{and}
            \label{eq:EntireRRepresentationsAlgebra}
            \abs[\big]
            {\Maj_{R, \Lie_{X_{\xi}} \phi}(z)}
            &\le
            \supnorm{\xi}
            \cdot
            \Maj_{R, \phi}'
            \big(
                \abs{z}
            \big).
    \end{align}
\end{propositionlist}
\end{proposition}
\begin{proof}
    The proofs of \ref{item:EntireVsTaylor} and
    \ref{item:TaylorVsEntire} are standard and will not be given
    here. For \ref{item:EntireMultiplication} we recall
    $c_k(\phi \cdot \psi) \le c_k(\phi) \cdot c_k(\psi)$ by the
    Leibniz inequality in
    Proposition~\ref{proposition:LeibnizAndChain},
    \ref{item:Leibniz}. Hence
    \begin{align*}
        \seminorm{q}_{R, c}(\phi \cdot \psi)
        &\le
        \sum \limits_{k=0}^{\infty}
        k!^R
        \sum \limits_{j=0}^k
        c_j(\phi) c^j
        c_{k-j}(\psi) c^{k-j} \\
        &=
        \sum \limits_{k=0}^{\infty}
        \binom{k}{j}^{R}
        \sum \limits_{j=0}^k
        j!^R c_j(\phi) c^j
        \cdot
        (k-j)!^R c_{k-j}(\psi) c^{k-j} \\
        &\le
        \sum \limits_{k=0}^{\infty}
        2^{kR}
        \sum \limits_{j=0}^k
        j!^R c_j(\phi) c^j \, (k-j)!^R
        c_{k-j}(\psi) c^{k-j} \\
        &=
        \seminorm{q}_{R, 2^R c}(\phi)
        \cdot
        \seminorm{q}_{R, 2^R c}(\psi),
  \end{align*}
  using the rather crude estimate $\binom{k}{j} \le 2^k$ resp.~the
  Cauchy product formula in the last two steps. Going through our
  computation for $R < 0$, one can be even cruder and estimate the
  binomial to the power of $R$ by $1$. The proof of
  \ref{item:EntireRRepresentations} is identical to the ones of
  Proposition~\ref{proposition:TaylorMajorantDerivative} and
  Proposition~\ref{prop:Translation} and will not be repeated
  here. The reader will notice that it \emph{does} rely on $R \ge 0$,
  though.
\end{proof}

We are thus naturally led to equip the vector space $\Entire(G)$ with
the family of seminorms \eqref{eq:seminorm2} and the corresponding
metric
\begin{equation}
    \label{eq:EntireRMetric}
    d_R(\phi, \psi)
    =
    d_0
    \big(
    \Maj_{R, \phi-\psi}, 0
    \big)
    =
    \sum \limits_{j=1}^{\infty}
    \frac{1}{2^j}
    \frac{\seminorm{q}_{R, j}(\phi - \psi)}
    {1 + \seminorm{q}_{R, c}(\phi - \psi)}.
\end{equation}

Notably, the inclusions \eqref{eq:EntireInclusions} are then
continuous. Before we take a closer look at the somewhat deeper
properties of the locally convex algebras $\Entire(G)$, we take a
detour to relate our considerations to classical notions from
infinite-dimensional representation theory.
\begin{remark}[Representation theory]
    \label{remark:RepresentationTheory}%
    In the literature, \cite[Sec.~2]{goodman:1969a} was the first to
    consider the Fréchet space of entire vectors for Lie algebra
    representations on Banach spaces $B$ induced by strongly
    continuous Lie group representations
    $\pi \colon G \longrightarrow L(B)$. Recall that a vector
    $v \in B$ is called $\Fun$-vector for $\pi$ if the maps
    \begin{equation}
        \label{eq:FunVectors}
        \pi_v
        \colon
        G
        \longrightarrow
        B, \quad
        \pi_v(g)
        =
        \pi(g)v
    \end{equation}
    are $\Fun$-functions with values in $B$, where
    $k \in \field{N}_0 \cup \{\infty, \omega\}$. By sequential
    completeness of the Banach space $B$ one can indeed define
    differentiability by means of differential quotients. One writes
    $\Fun(\pi) \subseteq B$ for the set of all $\Fun$-vectors for the
    representation $\pi$. The assumed strong continuity of the
    representation $\pi$ then just means that
    $\Continuous(\pi) = \Fun[0](G) = B$. Note that doing things in
    this pointwise fashion corresponds to considering all limits in
    the weak topology of $L(B)$. In this context, the natural question
    is about density of the spaces $\Fun(G)$ as subspaces of
    $\Continuous(G)$. For a quite nice, albeit dated discussion, we
    refer to \cite{moore:1968a}. Some more modern discourses can be
    found in the textbooks \cite[Chap.~10]{schmuedgen:1990a} and
    \cite[Appendix~D]{taylor:1986a}.

    In fact, it is often necessary to go beyond the Banach space
    scenario and include more general locally convex spaces like
    e.g. Fréchet spaces as representation spaces as well.

    Passing to the infinitesimal situation, we differentiate and
    obtain a Lie algebra representation $T \pi$ of not necessarily
    continuous linear operators $T \pi \xi$, each defined on some
    subspace of the representation space. By the classical
    \cite{garding:1947a} they share a common dense invariant domain,
    the so-called Gårding space
    $\mathscr{G}(\pi) \subseteq \Cinfty(\pi)$. For Fréchet spaces, the
    seminal work \cite{dixmier.malliavin:1978a} proved the equality
    $\mathscr{G}(\pi) = \Cinfty(\pi)$.

    Analytic vectors are more complicated. By
    \cite[Sec.~3]{nelson:1959a} a smooth vector $v \in \Cinfty(\pi)$
    is analytic if and only if the formal exponential
    $\exp(T \pi \xi)v$ converges for all $\xi$ in some neighbourhood
    of $0 \in \liealg{g}$. \cite{goodman:1969a} turned this into a
    seminorm condition and obtained families of Fréchet spaces
    $\mathscr{H}_t(\pi)$ this way: demanding convergence for
    $\xi \in \Ball_t(0) \subseteq \liealg{g}$ prescribes a uniform
    radius of convergence for $\exp(T \pi \xi)v$. Of particular
    interest are then the union over $t = \tfrac{1}{n}$ and the
    intersection over $t = n$ with $n \in \field{Z}$: the former
    endows the space of analytic vectors with the structure of a
    locally convex inductive limit, the latter is Fréchet again as the
    countable intersection of Fréchet spaces. For obvious reasons,
    \cite{goodman:1969a} calls this intersection entire vectors and so
    shall we. Note that there is no need to start with a Lie group
    representation at all: all notions make sense for Lie algebra
    representations from the start. The natural question is then, when
    such a representation can be integrated. Some answers can be found
    in \cite{moore:1965a},
    \cite{flato.simon.snellman.sternheimer:1972a} and the much more
    recent developments \cite{cabral:2019a, schmuedgen:1990a,
      schmuedgen:2020a}.

    The action we are interested in is the translation action of the
    group $G$ by means of pullbacks
    \begin{align}
        \label{eq:TranslationAction}
        \ell_{\bullet}^*, r_{\bullet}^*
        \colon
        G
        \longrightarrow
        L\big(\Continuous(G)\big)
    \end{align}
    by left and right multiplications, which is surprisingly
    ill-behaved. Note that the space of continuous functions
    $\Continuous(G)$ is \emph{not} Banach with respect to its usual
    topology of locally uniform convergence in general, but a Fréchet
    space. It is straightforward to generalize the notions of
    $\Fun$-vectors of a group representation into this more general
    setting. These more general cases were studied e.g. in
    \cite{miyadera:1959a} and \cite{komura:1968a}. However, most of
    the useful techniques break down, unless the set $\pi(G)$ is an
    equicontinuous family of operators. Most importantly, the
    utilization of (Riemann or Bochner) integral methods require this
    assumption. Notably, the translation actions
    \eqref{eq:TranslationAction} are not equicontinuous at all: there
    is no compact subset of $G$ that contains all compact subsets,
    unless the group is compact itself. Remarkably, we will be able to
    reproduce most of the pleasant results for our particular
    situation regardless.

    Our Taylor formula \eqref{eq:LieTaylorSingleVariable} implies that
    each $\phi \in \Entire[0](G)$ is an entire vector with respect to
    the representation $r_\bullet^*$. Or differently put, the lifted
    Lie algebra representation \eqref{eq:LieDerivativeRep}
    exponentiates to the Lie group representation $r_\bullet^*$. Our
    Definition~\ref{definition:Entire} is thus the appropriate
    infinitesimal version of entire vectors in our locally convex
    situation. This matches with a straightforward reformulation of
    the textbook definitions \cite[Def.~10.3.1 and
    10.3.2]{schmuedgen:1990a} in the following sense: We consider the
    lifted Lie algebra representation \eqref{eq:LieDerivativeRep} and
    equip $\Continuous(G)$ with the continuous seminorm
    $\abs{\delta_\E}$ given by the absolute value of the Dirac
    functional at the group unit
    \begin{equation}
        \label{eq:Dirac}
        \delta_\E
        \colon
        \Continuous(G)
        \longrightarrow
        \field{C}, \quad
        \delta_\E(\phi)
        =
        \phi(\E).
    \end{equation}
    While this is not the natural topology of $\Continuous(G)$ at all,
    this is the useful choice for estimation and to generate examples
    later on. By the upcoming Theorem~\ref{theorem:EntireRep},
    \ref{item:EntireSeminormUniform}, this choice leads to the same
    locally convex space $\Entire[0](G)$ as \emph{the natural one} a
    posteriori. This also identifies $\Entire[0](G)$ as the space of
    entire vectors for the group representations
    \eqref{eq:TranslationAction} in the sense of \cite{goodman:1969a}.
\end{remark}

With this incomplete discussion in mind, we can show the following:
\begin{theorem}[Representation theory]
    \label{theorem:EntireRep}%
    Let $G$ be a connected Lie group and let $R \ge 0$.
    \begin{theoremlist}
    \item \label{item:EntireInversion}%
        Group inversion $\inv \colon G \longrightarrow G$ induces an
        isometry of $(\Entire(G),d_R)$, that is,
        $\phi \circ \inv \in \Entire(G)$ and
        \begin{equation}
            \label{eq:EntireInversion}
            d_R(\phi \circ \inv, \psi \circ \inv)
            =
            d_R(\phi, \psi)
        \end{equation}
        for all $\phi,\psi \in \Entire(G)$.
    \item \label{item:EntireTranslation}%
        Pullbacks with left and right translations yield
        representations
        \begin{equation}
            \label{eq:TranslationRep}
            \ell_{\bullet}^*, r_{\bullet}^*
            \colon
            G
            \longrightarrow
            L\big( \Entire(G) \big)
        \end{equation}
        by continuous linear maps.
    \item \label{item:EntireDifferentiation}%
        The space $\Entire(G)$ of entire functions is invariant under
        the lifted Lie algebra representation
        \eqref{eq:LieDerivativeRep} by continuous maps. More
        precisely, we have the estimate
        \begin{equation}
            \label{eq:EntireDifferentiationEstimate}
            \seminorm{q}_{R, c}
            \big(
                \Lie_{X_{\xi}}\phi
            \big)
            \le
            \seminorm{q}_{R, c+1}(\phi)
            \cdot
            \supnorm{\xi}
        \end{equation}
        for $c \ge 0$, $\xi \in \liealg{g}$ and $\phi \in \Entire(G)$.
    \item \label{item:EntireLieTaylor}%
        The Lie-Taylor series $\Taylor_\phi(\argument; \underline{z})$
        is absolutely convergent in $\Entire(G)$ for every
        $\underline{z} \in \field{C}^n$. Thus every entire function
        $\phi \in \Entire(G)$ is an entire vector for the translation
        representations \eqref{eq:TranslationRep}, which are in
        particular strongly continuous.
    \item \label{item:EntireVsCinfty}%
        The $\Entire(G)$-topology is finer than the
        $\Cinfty$-topology. In particular, the evaluation functionals
        \begin{equation}
            \label{eq:EntireEvaluations}
            \delta_{g, \alpha}
            \colon
            \Entire(G)
            \longrightarrow
            \field{C}, \quad
            \delta_g(\phi)
            =
            \big(
                \Lie_{X_\alpha}
                \phi
            \big)(g)
        \end{equation}
        are continuous linear maps for $g \in G$ and
        $\alpha \in \field{N}^k_n$.
    \item \label{item:EntireSeminormUniform}%
        The alternative seminorms
        \begin{equation}
            \label{eq:EntireSeminormUniform}
            \seminorm{r}_{R, c}(\phi)
            =
            \sum_{k=0}^\infty
            k!^R \frac{c^k}{k!}
            \sum_{\alpha \in \field{N}_n^k}
            \sup_{g \in K}
            \abs[\Big]
            {
                \big(
                    \Lie_{X_{\alpha}}
                    \phi
                \big)(g)
            }
        \end{equation}
        with $c \ge 0$ and compact $K \subseteq G$ are well-defined on
        $\Entire(G)$ and constitute a defining system for the
        $\Entire$-topology. Thus $\Entire[0](G)$ is the space of
        entire vectors for the representation
        \eqref{eq:TranslationAction} if we equip $\Continuous(G)$ with
        its canonical topology.
    \end{theoremlist}
\end{theorem}
\begin{proof}
    The statement \ref{item:EntireInversion} is just
    Lemma~\ref{lemma:Inversion} again. For
    \ref{item:EntireTranslation} we consider
    $g = \exp(\xi_1) \cdots \exp(\xi_m)$ with
    $\xi_1, \ldots, \xi_m \in \liealg{g}$
    s.t.~$\supnorm{\xi_j} \le 1$. Applying
    Proposition~\ref{proposition:REntire},
    \ref{item:EntireRRepresentations}, $m$-times we obtain
    \begin{equation*}
        \abs[\big]
        {\Maj_{R,\phi \circ r_g}(z)}
        \le
        \Maj_{R,\phi}
        \big(
            \abs{z} + m
        \big)
    \end{equation*}
    for $z \in \field{C}$. This proves
    \begin{equation*}
        \seminorm{q}_{R, c}
        (\phi \circ r_g)
        \le
        \seminorm{q}_{R, c+m}
        (\phi)
    \end{equation*}
    for $c \ge 0$. The corresponding property of left multiplication
    by $g$ follows now from \ref{item:EntireInversion} by once again
    noting that
    $\phi \circ \ell_g = \phi \circ \inv \circ r_{g^{-1}} \circ
    \inv$. Thus the translations $r_g$ and $\ell_g$ are continuous
    selfmaps of $\Entire(G)$. For \ref{item:EntireDifferentiation} the
    Cauchy integral formula yields
    \begin{equation*}
        \label{eq:EntireRepProof}
        \Maj'_{R,\phi}(c)
        =
        \frac{1}{2\pi}
        \abs[\bigg]
        {
            \int_{\boundary \Ball_r(c)}
            \frac{\Maj_{R,\phi}(w)}{(w - c)^2}
            \D w
        }
        \le
        r^{-1}
        \Maj_{R,\phi}(c+r)
        =
        r^{-1}
        \cdot
        \seminorm{q}_{R, c+r}(\phi)
        \tag{$*$}
    \end{equation*}
    for every $r > 0$. Taking $r=1$,
    Proposition~\ref{proposition:REntire},
    \ref{item:EntireRRepresentations}, shows that
    \begin{equation*}
        \seminorm{q}_{R,c}
        \big(
            \Lie_{X_{\xi}}\phi
        \big)
        =
        \Maj_{R,\Lie_{X_{\xi}} \phi}(c)
        \le
        \supnorm{\xi}
        \cdot
        \Maj'_{R,\phi} (c)
        \le
        \supnorm{\xi}
        \cdot
        \seminorm{q}_{R, c+1}(\phi).
    \end{equation*}
    This completes \ref{item:EntireDifferentiation}, wherefore asking
    for \ref{item:EntireLieTaylor} makes sense at all: each of the
    $\Lie_{X_\alpha} \phi$ is an $R$-entire function itself. The idea
    is now to choose $r_k = \tfrac{1}{k}$ in
    \eqref{eq:EntireRepProof}, which gives due to
    $\supnorm{\basis{e}_j} = 1$ for $N > 0$
    \begin{equation*}
        \sum_{k=N}^\infty
        \frac{1}{k!}
        \sum_{\alpha \in \field{N}_n^k}
        \seminorm{q}_{R, c}
        \big(
            \Lie_{X_\alpha} \phi
        \big)
        \le
        \sum_{k=N}^\infty
        \frac{1}{k!}
        \sum_{\alpha \in \field{N}_n^k}
        k
        \cdot
        \seminorm{q}_{R, c+1}(\phi)
        =
        \seminorm{q}_{R, c+1}(\phi)
        \sum_{k=N}^\infty
        \frac{k \cdot n^k}{k!}
        \longrightarrow
        0.
    \end{equation*}
    Thus the Taylor series $\Taylor(\argument, \underline{z})$ is
    indeed absolutely convergent in $\Entire(G)$. In other words, each
    of the maps
    \begin{equation*}
        G
        \ni
        g
        \mapsto
        r_g^* \phi
        \in
        \Entire(G)
    \end{equation*}
    is analytic by \eqref{eq:LieTaylorSingleVariable}. The statement
    about left translations follows the usual way. As analytic maps
    are continuous, this also gives the strong continuity of
    \eqref{eq:TranslationRep}. We now turn to
    \ref{item:EntireVsCinfty} and notice that it suffices to handle
    the case $R=0$. Let $K \subseteq G$ be a compact set.  By a
    covering argument it is easy to see that there is a positive
    integer $m$ (depending only on $K$) with the property that for any
    $g \in K$ there are
    $\xi_1,\ldots, \xi_m \in \Ball_1(0)^\cl \subseteq \liealg{g}$ such
    that $g=\exp(\xi_1)\cdots \exp(\xi_m)$. This implies for any
    $\phi \in \Entire[0](G)$
    \begin{equation*}
        \abs[\big]{\phi(g)}
        =
        \abs[\Big]
        {
            \big(
            \phi \circ r_{g}
            \big)(\E)
        }
        =
        \abs[\big]
        {\Maj_{\phi \circ r_g}(0)}
        =
        \abs[\big]
        {\Maj_{\phi \circ r_{\exp(\xi_1)\cdots \exp(\xi_m)}}(0)}.
    \end{equation*}
    Applying Proposition \ref{prop:Translation} $m$-times we obtain
    \begin{equation*}
        \label{eq:convE0vsLUProof}
        \abs[\big]{\phi(g)}
        \le
        \abs[\big]
        {
            \Maj_{\phi}
            \big(
            \supnorm{\xi_1} + \cdots + \supnorm{\xi_m}
            \big)
        }
        \le
        \seminorm{q}_{0, m}(\phi)
        \tag{$**$}
    \end{equation*}
    for all $g \in K$, yielding the claim. Keeping this compact subset
    $K \subseteq G$, we turn towards
    \ref{item:EntireSeminormUniform}. Note that
    $\seminorm{r}_{0, c, \{\E\}} = \seminorm{q}_{0, c}$, wherefore the
    topology induced by the seminorms \eqref{eq:EntireSeminormUniform}
    is certainly finer than the $\Entire[0]$-topology. Using
    \eqref{eq:EntireRepProof} with $r_k = \tfrac{1}{k}$ again and what
    we have just shown gives on the other hand
    \begin{equation*}
        \seminorm{r}_{R, c}(\phi)
        \le
        \sum_{k=0}^\infty
        k!^{R-1} c^k
        \sum_{\alpha \in \field{N}_n^k}
        \seminorm{q}_{R, m}
        \big(
            \Lie_{X_{\alpha}}
            \phi
        \big)
        \le
        \sum_{k=0}^\infty
        k!^{R-1} (cn)^k
        \cdot
        \max\{k, 1\}
        \cdot
        \seminorm{q}_{R, m+1}(\phi),
    \end{equation*}
    wherefore both topologies do indeed coincide. Here we have used
    $\seminorm{q}_{R, c} \le \seminorm{q}_{R, c+1}$.
\end{proof}

The following theorem summarizes the main properties of the locally
convex algebras $(\Entire(G), d_R)$ and shows that they share many of
the pleasant properties of the Fr\'echet algebras
$\mathcal{H}_R(\mathbb{C})$ of all entire holomorphic functions (of
finite order and minimal type).
\begin{theorem}[Properties of \protect{$\Entire(G)$}]
    \label{theorem:Entire}%
    Let $G$ be a connected Lie group and let $R \ge 0$. Then the
    locally convex algebra $(\Entire(G),d_R)$ is
    \begin{theoremlist}
        \item \label{item:EntireFrechet}%
         a Fr\'echet algebra;
        \item \label{item:EntireNuclear}%
         nuclear;
        \item \label{item:EntireMontel}%
        a Montel
        space: Every bounded and closed set in $\Entire[0](G)$ is
        compact;
        \item \label{item:EntireSeparable}%
         separable and reflexive.
    \end{theoremlist}
\end{theorem}
\begin{proof}
    We tackle \ref{item:EntireFrechet} first. Let $(\phi_j)$ be a
    Cauchy sequence in $(\Entire(G), d_R)$. By
    Theorem~\ref{theorem:EntireRep}, \ref{item:EntireVsCinfty}, there
    is a $\phi \in \Cinfty(G)$ such that $\phi_j \rightarrow \phi$
    converges in $\Cinfty(G)$. We need to show that
    $\phi \in \Entire(G)$, i.e. $\phi$ is analytic as well as
    $\seminorm{q}_{R, c}(\phi) < \infty$ for all $c \ge 0$, and
    $\phi_j \rightarrow \phi$ in $(\Entire(G), d_R)$. We proceed in
    several steps.
    \begin{theoremlist}
    \item[(1)] \label{item:FrechetProof} For each $g \in G$ and each
        $k \in \mathbb{N}_0$ the $k$-th order homogeneous Taylor
        polynomial of $\phi_j$ at $g$ converges to the $k$-th order
        homogeneous Taylor polynomial of $\phi$ at $g$, that is,
        \begin{equation*}
            \Taylor_{k, \phi_j}(\underline{z}; g)
            =
            \frac{1}{k!}
            \sum \limits_{\alpha \in \mathbb{N}^k_n}
            \big(
            \Lie_{X_{\alpha}} \phi_j
            \big)(g)
            \cdot
            \underline{z}^{\alpha}
            \quad \overset{j \rightarrow \infty}{\longrightarrow} \quad
            \frac{1}{k!}
            \sum \limits_{\alpha \in \mathbb{N}^k_n}
            \big(
            \Lie_{X_{\alpha}} \phi
            \big)(g)
            \cdot
            \underline{z}^{\alpha}
            =
            \Taylor_{k, \phi}(\underline{z}; g)
        \end{equation*}
        locally uniformly w.r.t.~$\underline{z} \in \field{C}^n$. This
        follows from $\phi_j \to \phi$ in $\Cinfty(G)$ by continuity
        of the linear operators
        \begin{equation*}
            \Cinfty(G)
            \ni
            \psi
            \mapsto
            \Taylor_{k, \psi}
            (\argument; g)
            \in
            \Holomorphic(\field{C}^n).
        \end{equation*}

    \item[(2)] Fix $g \in G$. Then
        $\Taylor_{\phi_j}(\argument; g) \rightarrow
        \Taylor_{\phi}(\argument; g)$ locally uniformly on
        $\field{C}^n$. In order to see this, we note that as a Cauchy
        sequence, $(\phi_j)$ is in particular bounded in
        $(\Entire[0](G), d_0)$. With other words,
        $(\Maj_{\phi_j}) \subseteq \Holomorphic(\field{C})$ is locally
        bounded on $\field{C}$. By
        Corollary~\ref{corollary:LieTaylorTranslation} we can deduce
        that $(\Taylor_{\phi_j}(\argument; g))$ is locally bounded in
        $\Holomorphic(\field{C}^n)$. Since
        \begin{equation*}
            \Taylor_{\phi_j}
            (\underline{z}; g)
            =
            \sum \limits_{k=0}^{\infty}
            \Taylor_{k, \phi_j}
            (\underline{z}; g)
            \quad \textrm{and} \quad
            \Taylor_{\phi}
            (\underline{z}; g)
            =
            \sum \limits_{k=0}^{\infty}
            \Taylor_{k,\phi}
            (\underline{z};g),
        \end{equation*}
        we see from (1) and a standard Montel-type
        normal family argument that
        $\Taylor_{\phi_j}(\argument;g) \rightarrow
        \Taylor_{\phi}(\argument; g)$ locally uniformly on
        $\field{C}^n$.

    \item[(3)] In view of Corollary~\ref{corollary:LieTaylor}, we also have
        \begin{equation*}
            \Taylor_{\phi_j}
            (\underline{x}; g)
            =
            \phi_j
            \big(
            g \exp(x^\ell \basis{e}_\ell)
            \big)
            \rightarrow
            \phi
            \big(
            g \exp(x^\ell \basis{e}_\ell)
            \big)
        \end{equation*}
        locally uniformly for $(x^1,\ldots, x^n) \in \field{R}^n$.  It
        follows that
        \begin{equation*}
            \Taylor_{\phi}
            (\underline{x}; g)
            =
            \phi
            \big(
            g \exp(x^\ell \basis{e}_\ell)
            \big)
        \end{equation*}
        for each $g \in G$ and $\underline{x} \in \mathbb{R}^n$.  In
        particular, $\phi \in \Comega(G)$, since
        $\underline{x} \mapsto g \exp(x^\ell \basis{e}_\ell)$ is an
        analytic diffeomorphism around $0 \in \mathbb{C}^n$.

    \item[(4)] Next, we show $\seminorm{q}_{R, c}(\phi) < \infty$ for
        all $c \ge 0$.  By Theorem~\ref{theorem:EntireRep},
        \ref{item:EntireDifferentiation} we have
        $\Lie_{X_{\alpha}} \phi_j \to \Lie_{X_{\alpha}} \phi$ for each
        $k$-tuple $\alpha \in \mathbb{N}^k_n$. Hence we have
        convergence of the $k$-th Taylor coefficients
        $c_k(\phi_j) \rightarrow c_k(\phi)$. Using the local
        boundedness of $(\Maj_{R, \phi_j}(\argument))$ in
        $\Holomorphic(\field{C})$ it easily follows that
        $(\Maj_{R,\phi_j})$ converges in $\mathcal{H}(\field{C})$ to
        $\Maj_{R,\phi}$.  In particular,
        $\Maj_{R,\phi} \in \mathcal{H}(\field{C})$, so
        $\seminorm{q}_{R,c}(\phi) = \Maj_{R,\phi}(c) < \infty$ for
        each $c \ge 0$.

    \item[(5)] Finally, we prove $\phi_j \rightarrow \phi$ in
        $(\Entire(G), d_R)$.  An equivalent statement is that
        $\Maj_{R,\phi_j - \phi}(c) \rightarrow 0$ for each $c \ge
        0$. By
        \begin{equation*}
            \abs[\big]
            {\Maj_{R,\phi_j - \phi}(z)}
            \le
            \Maj_{R,\phi_j}
            \big(|z|\big)
            +
            \Maj_{R,\phi}
            \big(|z|\big)
        \end{equation*}
        the sequence
        $(\Maj_{R,\phi_j - \phi}) \subseteq \Holomorphic(\mathbb{C})$
        is locally bounded. Hence the convergence of each of the
        $k$-Taylor coefficients $c_k(\phi_j - \phi)$ to $0$ and a
        simple Montel-type argument imply
        $\Maj_{R,\phi_j - \phi} \rightarrow 0$ even locally uniformly
        on $\mathbb{C}$.
    \end{theoremlist}
    For \ref{item:EntireNuclear} it is convenient to identify
    $\Entire(G)$ with a Köthe space $\Lambda_R$ in the following
    manner: write
    $\field{N}_n^\infty = \bigcup_{k=0}^\infty \field{N}_n^k$ and
    define a Köthe matrix by
    \begin{equation*}
        a_{\alpha c}
        =
        (k!)^R \cdot \frac{c^{k}}{k!}
    \end{equation*}
    for $\alpha \in \field{N}_n^k$, $k \in \mathbb{N}_0$ and
    $c \in \field{N}$. Note that $a_{\alpha c} \le a_{\alpha c'}$,
    whenever $c \le c'$. We identify each $\phi \in \Entire(G)$ with
    the sequence $(\phi_\alpha)$ defined by
    \begin{equation*}
        \phi_\alpha
        =
        \big(
        \Lie_{X_\alpha}
        \phi
        \big)(\E).
    \end{equation*}
    Note that
    $c_k(\phi) = \frac{1}{k!} \sum_{\alpha \in \field{N}_n^k}
    \abs{\phi_\alpha}$ from
    \eqref{eq:LieTaylorMajorantCoefficients}. This yields an injective
    isometry
    \begin{equation*}
        \Entire(G)
        \longrightarrow
        \Lambda_R,
    \end{equation*}
    as the net $(\phi_\alpha)$ contains even \emph{more} information
    than just the Lie-Taylor coefficients at $\E$. We use the
    Grothendieck-Pietsch Theorem as it can be found in
    \cite[Thm.~6.1.2]{pietsch:1972a} to check nuclearity of
    $\Lambda_R$. As moreover any subspace of a nuclear space is
    nuclear, see \cite[Prop.~5.1.1]{pietsch:1972a} or
    \cite[(50.3)]{treves:1967a}, the claim will follow. Thus let
    $c \in \field{N}$. We have to find a $c' \in \field{N}$ such that
    the series
    \begin{equation*}
        \sum_{\alpha \in \field{N}_n^\infty}
        \frac{a_{\alpha c}}{a_{\alpha c'}}
        =
        \sum_{k=0}^{\infty} \sum \limits_{\alpha \in \mathbb{N}^k_n}
        \frac{c^{k} k!}{c'^k  k!}
        =
        \sum_{k=0}^\infty
        n^k
        \frac{c^{k}}{c'^{k}}
    \end{equation*}
    converges. Taking $c' = 2 c n$ does the job. By positivity of all
    the numbers, all of our considerations are independent of the
    enumeration we choose for the index set $\field{N}_n^\infty$,
    wherefore we do not make this choice at all. Thus the nuclearity
    of $\Entire(G)$ follows. By \cite[Prop.~50.2]{treves:1967a} every
    nuclear Fréchet space is Montel, which gives
    \ref{item:EntireMontel}. Notably, this can also be shown directly
    by using the classical Montel Theorem for the Taylor
    majorants. Nuclear Montel spaces are separable by
    \cite[Section~11.6, Thm.~2]{jarchow:1981a} and by
    \cite[Cor.~36.9]{treves:1967a} every Montel space is reflexive.
\end{proof}

Rounding out this section, we discuss the simple but surprisingly far
reaching example of the circle group $\mathbb{S}^1$, which is closely
related to $\liegroup{GL}_1(\field{C}) \cong \field{C}^\times$ through
the notion of universal complexification:
\begin{example}[Circle group]
    \label{example:CircleGroup}%
    Let $G = \mathbb{S}^1$, which is connected, but not simply
    connected. Its universal complexification
    $\big( \mathbb{S}^1_\field{C}, \eta \big)$ in the sense of
    \cite[Def.~15.1.2.]{hilgert.neeb:2012a} is given by
    \begin{equation}
        \label{eq:CircleGroupComplexification}
        \mathbb{S}^1_\field{C}
        =
        \field{C}^\times
        =
        \field{C}
        \setminus
        \{0\}
        \quad \textrm{and} \quad
        \eta
        \colon
        \mathbb{S}^1
        \longrightarrow
        \mathbb{S}^1_\field{C}, \quad
        \eta
        =
        \id_{\field{C}^\times}
        \at{\mathbb{S}^1},
    \end{equation}
    i.e. we embed $\mathbb{S}^1$ as the unit circle into
    $\field{C} \setminus \{0\}$. This induces a complex structure,
    which coincides with the one given by exponential charts due to
    the local existence of holomorphic logarithms.  Given a morphism
    of complex Lie groups
    $\Phi \colon \mathbb{S}^1 \longrightarrow H$, we set
    \begin{equation}
        \Phi_\field{C}
        \colon
        \field{C}^\times
        \longrightarrow
        H, \quad
        \Phi_\field{C}
        \big(
            r \cdot \E^{2 \pi \I t}
        \big)
        =
        r
        \cdot
        \Phi(\E^{2 \pi \I t}),
    \end{equation}
    which is a holomorphic group morphism and clearly fulfils
    $\Phi = \Phi_\field{C} \circ \eta$.  Notably, the universal
    complexification $\field{C}^\times$ is not compact, even though
    the circle group $\mathbb{S}^1$ was, but their fundamental groups
    are clearly isomorphic. Analogously, one obtains the universal
    complexification for higher tori as products of
    $\field{C}^\times$. Notice that the group morphism given by the
    reflection on the unit circle
    \begin{equation}
        \label{eq:CircleGroupConjugation}
        \sigma
        \colon
        \field{C}^\times
        \longrightarrow
        \field{C}^\times, \quad
        \sigma
        \big(
            r \cdot e^{2\pi \I t}
        \big)
        =
        \frac{1}{r}
        \cdot
        e^{2 \pi \I t}
    \end{equation}
    is antiholomorphic and fixes $\eta(\mathbb{S}^1) =
    \mathbb{S}^1$. Thus $\sigma$ is the unique antiholomorphic complex
    conjugation on $\mathbb{S}^1_\field{C}$, whose existence is
    guaranteed by \cite[Thm.~15.1.4,
    \textit{iv.)}]{hilgert.neeb:2012a}.

    Let $f \colon \field{C}^\times \longrightarrow \field{C}$ be
    holomorphic. As
    $\exp(\field{C}) = \field{C} \setminus \{0\} = \field{C}^\times$,
    we can form the composition
    $f \circ \exp \colon \field{C} \longrightarrow \field{C}$, which
    is holomorphic as the composition of holomorphic maps and thus
    entire in the classical sense. By commutativity of
    $\mathbb{S}^1_\field{C}$, the Lie-Taylor majorant $\Maj_f$ is
    entire if and only if the Taylor series of $f \circ \exp$
    converges absolutely. Thus holomorphic functions on
    $\field{C}^\times$ are automatically entire in our
    sense. Consequently, we have shown
    \begin{equation}
        \label{eq:EntireCylinder}
        \Entire[0]
        \big(
            \mathbb{S}^1_\field{C}
        \big)
        \cap
        \Holomorphic(\mathbb{S}^1_\field{C})
        =
        \Holomorphic
        \big(
            \mathbb{S}^1_\field{C}
        \big)
    \end{equation}
    as locally convex algebras in view of
    Theorem~\ref{theorem:EntireRep}, \ref{item:EntireVsCinfty}: having
    a convergent Taylor series gives uniform estimates for the
    function values at once. Moreover, this restricts nicely to the
    real situation, i.e. we have
    \begin{equation}
        \label{eq:EntireCircle}
        \Entire[0]
        \big(
            \mathbb{S}^1
        \big)
        =
        \eta^*
        \Entire[0]
        \big(
            \mathbb{S}^1_\field{C}
        \big)
        =
        \eta^*
        \Holomorphic
        \big(
            \mathbb{S}^1_\field{C}
        \big).
    \end{equation}
    This can be seen as follows: given
    $\phi \in \Entire[0](\mathbb{S}^1)$, the composition
    $\phi \circ \exp \colon \field{C} \longrightarrow \field{C}$ is
    entire and $2\pi$-periodic (as the Lie exponential is
    $\exp(z) = \E^{\I z}$ here). Applying
    \cite[Thm.~3.10.1]{simon:2015b} to $a = b = n \in \field{N}$ with
    the obvious rescaling yields a $\Holomorphic$-convergent Laurent
    expansion
    \begin{equation}
        \label{eq:LaurentExpansion}
        \phi
        \big(
            \E^{2\pi \I z}
        \big)
        =
        \phi \circ \exp
        \at[\Big]{2 \pi z}
        =
        \sum_{k=-\infty}^{\infty}
        a_k
        \E^{2\pi \I k z}
    \end{equation}
    for $z \in \field{C}$. This is the holomorphic extension of $\phi$
    to $\field{C}^\times$ we were looking for. Rephrasing the
    convergence of \eqref{eq:LaurentExpansion}, we find two entire
    functions $F, G \in \Holomorphic(\field{C})$ with
    \begin{equation}
        \label{eq:LaurentExpansionRestricted}
        \phi(z)
        =
        F(z)
        +
        G(\tfrac{1}{z})
        =
        F(z)
        +
        G(\cc{z})
    \end{equation}
    for all $z \in \mathbb{S}^1$. If now $\phi \in \Entire(G)$ for
    $R > 0$, we still get the Laurent expansion
    \eqref{eq:LaurentExpansion}. It is even convergent in the
    $\Holomorphic_{1/R}$-topology, see again
    Proposition~\ref{proposition:REntire}. Thus
    \eqref{eq:LaurentExpansionRestricted} yields
    \begin{equation}
        \label{eq:EntireRCircle}
        \Entire(\mathbb{S}^1)
        =
        \eta^*
        \Holomorphic_{1/R}
        \big(
            \field{C}^\times
        \big)
    \end{equation}
    as locally convex algebras for $R > 0$.
\end{example}
This example suggests to use the universal complexification of $G$
also in general to understand the algebra $\Entire(G)$.
\begin{remark}[Negative $R$]
    As already suggested in Proposition~\ref{proposition:REntire},
    \ref{item:EntireMultiplication}, one can in principle consider
    arbitrary $R \in \field{R}$ by using the series of seminorms in
    \eqref{eq:EntireSeminormR}. However, the approach we have taken to
    deal with $R \ge 0$ \emph{ceases} to work here: the Lie-Taylor
    majorants need no longer be entire. This already happens in the
    abelian case $G = (\field{R}, +) = \liealg{g}$. Consider the
    geometric series
    \begin{equation}
        \label{eq:Geometric}
        g
        \colon
        \field{R}
        \longrightarrow
        \field{C}, \quad
        g(x)
        =
        \frac{1}{1 + \I x},
    \end{equation}
    whose Taylor series converges if and only if $\abs{x} < 1$. The
    underlying problem here is of course the singularity at $x = \I$,
    which is hidden in the universal complexification $\field{C}$ of
    $\field{R}$. Nevertheless, we have
    \begin{equation}
        \label{eq:GeometricTaylorR}
        \seminorm{q}_{R, c}(g)
        =
        \sum_{k=0}^\infty
        k!^{R} \cdot c^k
        <
        \infty
    \end{equation}
    for all $c > 0$ and $R < 0$. Consequently, we know at lot less
    about the nature of such functions: in fact, all of the upcoming
    results simply work for \emph{arbitrary} $R \in \field{R}$ and are
    somewhat algebraic in nature. Nevertheless, we shall admit $R < 0$
    and consider the $R$-entire functions $\Entire(G)$ also in this
    extended sense in the sequel.

    By the already mentioned Proposition~\ref{proposition:REntire},
    \ref{item:EntireMultiplication}, the resulting additional vector
    spaces are locally multiplicatively convex. Also the continuous
    inclusions \eqref{eq:EntireInclusions}, the inversion invariance
    and the continuity of Lie derivatives from
    Theorem~\ref{theorem:EntireRep}, \ref{item:EntireInversion} and
    \ref{item:EntireDifferentiation} remain correct. Indeed, we have
    the obvious inequalities
    \begin{equation}
        \label{eq:EntireSeminormInequalities}
        \seminorm{q}_{R, c}
        \le
        \seminorm{q}_{S, c}
        \quad \textrm{and} \quad
        \seminorm{q}_{R, c}
        \le
        \seminorm{q}_{R, d},
    \end{equation}
    whenever $-\infty<R \le S<\infty$ and $0 \le c \le d$. First
    countability and the Hausdorff property are also clearly still
    intact. Beyond these trivial observations, our methods break down
    immediately. For instance, we can no longer infer anything from
    \eqref{eq:Translation} if the right hand side is infinite.
\end{remark}

%
%

\subsection{Representative Functions}

At this point of our discussion it is not at all clear, whether there
are examples of entire functions on a given Lie group beyond the
constant ones.  To remedy this, we first note the following
compatibility of the entire functions with pullbacks by group
morphisms:
\begin{proposition}
    \label{proposition:EntireFunctoriality}%
    Let $G$ and $H$ be Lie groups and $R \in \field{R}$.  Let
    $\Phi \colon G \longrightarrow H$ be a Lie group morphism.  The
    pullback with $\Phi$ is a morphism of locally convex algebras
    \begin{equation}
        \label{eq:EntirePullback}
        \Phi^*
        \colon
        \Entire(H)
        \longrightarrow
        \Entire(G).
    \end{equation}
    More precisely, we have the estimate
    \begin{equation}
        \label{eq:EntirePullbackEstimate}
        \seminorm{q}_{R, c}
        \big(
        \Phi^* \phi
        \big)
        \le
        \seminorm{q}_{R, cnD}(\phi)
    \end{equation}
    for $\phi \in \Entire(G)$ and $c \ge 0$, where $D$ is the matrix
    supnorm of the matrix representation of the tangent map
    $T_\E \Phi \colon \liealg{g} \longrightarrow \liealg{h}$ in the
    bases used for the construction of the seminorms of $\Entire(G)$
    and $\Entire(H)$, respectively.
\end{proposition}
\begin{proof}
    Note that the pullback $\Phi^*$ with $\Phi$ is an algebra morphism
    \begin{equation*}
        \Phi^*
        \colon
        \Comega(H)
        \longrightarrow
        \Comega(G),
    \end{equation*}
    i.e. its restriction to $\Entire(H)$ is an algebra morphism with
    values in $\Comega(G)$. Consequently, this is just a rephrasing of
    the chain rule Proposition~\ref{proposition:LeibnizAndChain},
    \ref{item:ChainRule}. Note that the additional weight from
    $R \neq 0$ does not interfere with the argument, as it shows up on
    both sides of the estimate.
\end{proof}

This way, we obtain a contravariant functor $\Entire(\argument)$ from
the category of connected real Lie groups to the category of
commutative Fréchet algebras.  Thus our construction of the entire
functions fits nicely into the otherwise functorial framework of the
$R'$-topologies and the deformation quantization itself.

Generating examples of entire functions on one group thus lets us
transport them to other groups in a continuous way. Once again, this
does not guarantee the existence of interesting entire functions,
yet. The idea is now that group representations on finite dimensional
vector spaces are particularly nice group morphisms, to which we can
associate special functions on $G$: the \emph{representative
  functions} or \emph{matrix coefficients}, see
e.g. \cite[Sect.~4.3]{duistermaat.kolk:2000a}.  Recall that a choice
of generators is of the form
\begin{equation}
    \label{eq:RepFunctions}
    \pi_{ij}
    \colon
    G
    \longrightarrow
    \field{C}, \quad
    \pi_{ij}(g)
    =
    \pi(g)_{ij},
\end{equation}
where $\pi$ ranges over all continuous (and thus automatically
analytic) finite dimensional representations of $G$. We are going to
show by direct estimation that every representative function is
$R$-entire for $R < 1$. To this end, we recall the following
well-known Lemma:
\begin{lemma}
    \label{lemma:Rep}%
    Let $G$ be a Lie group and $\phi \in \Comega(G)$ be a
    representative function and write
    $\Orbit{\phi} = \{\ell_g^* \phi \in \Comega(G) \;|\; g \in G \}$
    for the orbit of $\phi$ under the left action of $G$ on itself.
    \begin{lemmalist}
    \item \label{item:RepOrbit}%
        The orbit $\Orbit{\phi}$ is finite dimensional and coincides
        with the orbit of $\phi$ under right translation, i.e. we have
        \begin{equation}
            \label{eq:RepOrbit}
            \Orbit{\phi}
            =
            \big\{
            r_g^* \phi
            \in \Comega(G)
            \; \big| \;
            g \in G
            \big\}.
        \end{equation}
    \item \label{item:RepDerivatives}%
        Let $\xi \in \liealg{g}$. The Lie derivative
        $\Lie_{X_\xi}$ is contained in the orbit $\Orbit{\phi}$
        of $\phi$. The same is true for right invariant vector
        fields.
    \end{lemmalist}
\end{lemma}
\begin{theorem}[Representative Functions]
    \label{theorem:RepresentingFunctions}%
    Let $G$ be a Lie group, $R < 1$ and $\phi \in \Comega(G)$ be a
    representative function. Then $\phi \in \Entire(G)$.  More
    precisely, choosing an auxiliary norm $\norm{\argument}$ on
    $\Orbit{\phi}$, we have the estimate
    \begin{equation}
       \label{eq:RepEntireEstimate}
       \seminorm{q}_{R, c}(\phi)
       \le
       \Delta \cdot \norm{\phi}
       \sum_{k=0}^{\infty}
       k!^{R-1} \; (c \Psi n)^k
    \end{equation}
    for $c \ge 0$, where $\Psi$ is the maximum of the operator
    seminorms
    \begin{equation}
        \label{eq:RepEntireOpnormDerivatives}
        \Psi
        =
        \max_{i = 1, \ldots, n}
        \max_{\psi \in \Orbit{\phi}, \abs{\psi(\E)} \le 1}
        \abs[\Big]{\big(\Lie_{X_i} \psi\big)(\E)}
    \end{equation}
    of the Lie derivatives in direction of left invariant vector
    fields on the orbit $\Orbit{\phi}$ of $\phi$ and $\Delta$ is the
    operator norm of the Dirac functional at the group unit, i.e.
    \begin{equation}
        \label{eq:RepEntireOpnormDirac}
        \Delta
        =
        \max_{\psi \in \Ball_1(0)^\cl}
        \abs[\big]{\psi(\E)}.
    \end{equation}
\end{theorem}
\begin{proof}
    We equip the, by Lemma~\ref{lemma:Rep}, \ref{item:RepOrbit},
    finite-dimensional orbit $\Orbit{\phi}$ with some auxiliary norm
    $\norm{\argument}$.  Part~\ref{item:RepDerivatives} of the same
    lemma implies that the ``left invariant derivatives''
    \begin{equation*}
        \Lie_{X_i}
        \colon
        \Orbit{\phi}
        \longrightarrow
        \Orbit{\phi}
    \end{equation*}
    are well-defined and thus continuous as linear maps on a
    finite-dimensional topological vector space. Consequently, the
    maximum of the operator norms
    \begin{equation*}
        \Psi
        =
        \max_{i = 1, \ldots, n}
        \max_{\psi \in \Ball_1(0)^\cl}
        \norm[\big]{\Lie_{X_i} \psi}
    \end{equation*}
    is finite. For the very same reasons, the Dirac functional
    $\delta_\E$ at the group unit is continuous and thus its operator
    norm
    \begin{equation*}
        \Delta
        =
        \max_{\psi \in \Ball_1(0)^\cl}
        \abs[\big]{\psi(\E)}
    \end{equation*}
    is finite, as well. Putting both observations together, we obtain
    \begin{equation*}
        \seminorm{q}_\alpha(\phi)
        =
        \abs[\Big]
        {
            \delta_\E
            \circ
            \Lie_{X_{\alpha_k}} \cdots \Lie_{X_{\alpha_1}}
            \phi
       }
       \le
       \Delta
       \cdot
       \Psi^k
       \cdot
       \norm{\phi}
    \end{equation*}
    for any $k$-tuple $\alpha \in \{1, \ldots, n\}^k$. Plugging this
    into the full seminorms yields finally
    \begin{align*}
        \seminorm{q}_{R, c}(\phi)
        &=
        \sum_{k=0}^{\infty}
        k!^{R-1} \; c^k
        \sum_{\alpha \in \{1, \ldots, n\}^k}
        \seminorm{q}_{\alpha} (\phi) \\
        &\le
        \sum_{k=0}^{\infty}
        k!^{R-1} \; c^k
        \sum_{\alpha \in \{1, \ldots, n\}^k}
        \Psi^k \cdot \Delta \cdot \norm{\phi} \\
        &=
        \Delta \cdot \norm{\phi}
        \sum_{k=0}^{\infty}
        k!^{R-1} \; (c \Psi n)^k,
    \end{align*}
    which converges for all $c \ge 0$ iff $R < 1$.
\end{proof}

Notably, choosing a basis of $\Orbit{\phi}$ allows to proceed in the
spirit of Proposition~\ref{proposition:EntireFunctoriality} to derive
a similar estimate based on the matrix supnorm of the Lie derivatives
in direction of the left invariant vector fields instead. This is
fairly cumbersome in terms of bookkeeping due to numerous indices,
wherefore we chose the more abstract approach. For $R = 0$, we
recognize the series in the estimate \eqref{eq:RepEntireEstimate} as
$\exp(c \Psi n)$. Note that the condition $R < 1$ is sharp:
\begin{example}[Exponential representation]
    Consider the representation $\exp$ of the abelian Lie group
    $(\mathbb{R}, +)$ on $(\mathbb{R}^\times, \cdot)$.  Indeed, we
    have
    \begin{equation}
        \label{eq:ExpIsRepresenting}
        \ell_x^* \exp \at{t}
        =
        \exp(x + t)
        =
        \exp(x)
        \cdot
        \exp
        \at{t}
    \end{equation}
    for $x, t \in \field{R}$, confirming that $\exp$ is a representative
    function with one dimensional orbit. This matches with
    Lemma~\ref{lemma:Rep}, \ref{item:RepDerivatives}, as
    $\exp' = \exp$. This also implies
    \begin{equation}
        \label{eq:ExpSeminorm}
        \seminorm{q}_{R, c}(\exp)
        =
        \sum_{k=0}^\infty
        k!^{R-1} \; c^k
        \cdot
        1
    \end{equation}
    for $R \in \field{R}$ and $c \ge 0$. The series
    in \eqref{eq:ExpSeminorm} converges for all $c \ge 0$ iff $R < 1$.
\end{example}

For a compact Lie group, we now know that the span of the matrix
coefficients is already \emph{dense} in $\Continuous(G)$. This is the
classical Peter-Weyl Theorem, see
\cite[Thm.~(4.6.1)]{duistermaat.kolk:2000a}. Thus the same is true
for $\Entire(G)$, whenever $R < 1$. Here we also use that
pointwise complex conjugation is an isometry of $\Entire(G)$. In the
language of representation theory, see again the lengthy
Remark~\ref{remark:RepresentationTheory}, this means that the
subspace of entire vectors is dense in the space of continuous vectors
for either of the representations \eqref{eq:LieDerivativeRep} or
\eqref{eq:TranslationAction}.

%
%

\section{The $R, R'$-Topologies on the Observable Algebra}
\label{sec:TopologiesObservable}%

Having studied both the $R$-entire functions $\Entire(G)$ and the
symmetric algebra $\SymR(\liealg{g})$ with the $R'$-topology in
isolation, we now projectively tensorize them together to the
observable algebra of our strict deformation:
\begin{definition}[$(R,R')$-Topologies]
    \label{definition:Observable}%
    Let $G$ be a Lie group and $R, R' \in \field{R}$. We equip the
    tensor product
    \begin{equation}
        \label{eq:ObservableAlgebra}
        \PolR(T^*G)
        =
        \Entire(G) \tensor \SymR(\liealg{g})
    \end{equation}
    with the projective tensor product topology and call it the
    $(R,R')$-topology.
\end{definition}

Due to $\Entire(G) \subseteq \Cinfty(G)$ and the decomposition
\eqref{eq:PolynomialFactorization} we have the inclusion
$\PolR(T^*G) \subseteq \Pol(T^*G)$, explaining the notation.

As projective tensor products moreover inherit most of the desirable
properties of their factors, we immediately obtain the following
statements for $\PolR(T^*G)$ and the commutative pointwise
multiplication:
\begin{proposition}
    Let $G$ be a connected Lie group and $R, R' \in \field{R}$.
    \begin{propositionlist}
    \item \label{item:ObservableProduct}%
        The projective tensor product topology turns $\PolR(T^*G)$
        into a unital Hausdorff and first countable locally convex
        algebra.
    \item \label{item:CCisContinuous} Complex conjugation is a
        continuous involution on $\PolR(T^*G)$.
    \item \label{item:ObservableLMC}%
        Let $R, R' \le 0$. Then $\PolR(T^*G)$ is locally
        multiplicatively convex.
    \item \label{item:ObservableCompletion}%
        The completion $\PolRC(T^*G)$ of $\PolR(T^*G)$ contains the
        completion of each factor and they are dense, i.e.
        \begin{equation}
            \label{eq:ObservableCompletion}
            \EntireC(G)
            \tensor
            \SymRC(\liealg{g})
            \subseteq
            \PolRC(G).
        \end{equation}
    \item \label{item:ObservableFrechet}%
        The completion $\PolRC(T^*G)$ is a commutative Fréchet
        $^*$-algebra.
    \item \label{item:ObservableInclusions}%
        Let $R \le S$ and $R' \le S'$. We have the continuous
        inclusions of locally convex algebras
        \begin{equation}
            \label{eq:ObservableInclusions}
            \PolR[S, R'](T^*G)
            \subseteq
            \PolR(T^*G)
            \quad \textrm{and} \quad
            \PolR[R, S'](T^*G)
            \subseteq
            \PolR(T^*G).
        \end{equation}
    \item \label{item:ObservableNuclear}%
        Let $R, R' \ge 0$. The locally convex algebras $\PolR(T^*G)$
        and $\PolRC(T^*G)$ are nuclear.
    \item \label{item:ObservableMontel}%
        Let $R, R' \ge 0$. The locally convex algebra $\PolRC(T^*G)$
        is Montel, reflexive and separable.
    \end{propositionlist}
\end{proposition}
\begin{proof}
    All statements are standard results about projective tensor
    products and have nothing to do with our particular example. For
    detailed treatments, see e.g. the textbooks \cite[Chap.~43,
    50]{treves:1967a}, \cite[§41]{koethe:1979a} and
    \cite[Chap.~15]{jarchow:1981a}. The continuity of the complex
    conjugation is clear as all our seminorms are invariant under
    complex conjugation.
\end{proof}

As a first consequence of the construction, we note that restricting
to momentum zero, which geometrically is the map $\iota^*$, and
prolonging constantly in momentum direction, which is $\pi^*$, provide
continuous maps:
\begin{proposition}
    \label{proposition:IotaPiContinuous}%
    Let $R, R' \ge 0$.
    \begin{propositionlist}
    \item \label{item:IotaStetig} The restriction to the zero section
        yields a continuous map
        \begin{equation}
            \label{eq:IotaContinuous}
            \iota^*\colon \PolRC(T^*G) \longrightarrow \Entire(G).
        \end{equation}
    \item \label{item:PullBackContinuous} The pullback
        \begin{equation}
            \label{eq:PullbackStetig}
            \pi^*\colon
            \Entire(G)
            \longrightarrow
            \PolR(T^*G) \subseteq \PolRC(T^*G)
        \end{equation}
        is continuous.
    \end{propositionlist}
\end{proposition}
\begin{proof}
    From the above factorization we have
    $\iota^* = \id_{\Entire(G)} \tensor \delta_0$ on the dense
    subalgebra $\PolR(T^*G)$ with the $\delta$-functional
    \begin{equation*}
        \delta_0
        \colon
        \SymR(\liealg{g}) \longrightarrow \field{C}.
    \end{equation*}
    This is a continuous functional for $R' \ge 0$ according to
    Proposition~\ref{proposition:PropertiesSR},
    \ref{item:RTopologyEvaluationsContinuity}. The functoriality of
    the projective tensor product implies the continuity of $\iota^*$,
    which then extends to the completion. The pullback is even
    simpler, we have
    \begin{equation*}
        \pi^* \phi = \phi \tensor 1,
    \end{equation*}
    which is again continuous by general properties of the projective
    tensor product.
\end{proof}

As tensor products of continuous linear maps are continuous in the
projective tensor product topology, we moreover obtain the following
continuity of evaluations and symmetries:
\begin{proposition}
    \label{proposition:CharactersSymmetriesContinuous}%
    Let $R, R' \in \field{R}$.
    \begin{propositionlist}
    \item \label{item:CharactersContinuous}%
        Assume $R, R' \ge 0$ and let $g \in G$,
        $\alpha \in \field{N}_n^k$ and $\eta \in \liealg{g}^*$. The
        tensor product
        \begin{equation}
            \label{eq:CharactersContinuous}
            \delta_{g, \alpha}
            \tensor
            \delta_\eta
            \colon
            \PolR(T^*G)
            \longrightarrow
            \field{C}
        \end{equation}
        of the evaluation functionals \eqref{eq:EntireEvaluations} and
        \eqref{eq:RTopologyEvaluations} is continuous.
    \item \label{item:SymmetriesContinuous}%
        Let $\Phi \colon G \longrightarrow H$ be a covering map of Lie
        groups. The pullback with the point transformation
        \begin{equation}
            \label{eq:SymmetriesContinuous}
            (T_* \Phi)^*
            \colon
            \PolR(T^*H)
            \longrightarrow
            \PolR(T^*G)
        \end{equation}
        is well-defined and continuous.
    \end{propositionlist}
\end{proposition}
\begin{proof}
    The first part is clear. For the second, recall that the canonical
    isomorphism \eqref{eq:CanonicalIsoJ} fulfils
    \begin{equation*}
        (T_* \Phi)^*
        \mathcal{J}(X)
        =
        \mathcal{J}
        \big(\Phi^* X\big)
    \end{equation*}
    for $X \in \Secinfty(\Sym^\bullet(T_\field{C} H))$. In our
    polynomial factorization from \eqref{eq:PolynomialFactorization},
    the point transformation is thus given by
    \begin{equation*}
        \Phi^*
        \tensor
        \big(
            (T_\E \Phi)^{-1}
            \tensor
            \cdots
            \tensor
            (T_\E \Phi)^{-1}
        \big).
    \end{equation*}
    With this formula, our claims are clear in view of
    Proposition~\ref{proposition:EntireFunctoriality} and the basis
    independence of the $R'$-topology. Note that the invertibility of
    the tangent map $T_\E \Phi$ is equivalent to the group morphism
    $\Phi$ being a covering map.
\end{proof}

Recall that we may endow the cotangent bundle $T^*G$ with a natural
Lie group structure by choosing a trivialization. More precisely, this
allows for the semidirect product structure $T^*G = G \ltimes_{\Ad^*}
\liealg{g}^*$ coming from the coadjoint representation. The natural
question is thus whether this group structure preserves our observable
algebra $\PolRC(T^*G) \subseteq \Cinfty(T^*G)$. As before, we denote
the left and right multiplications with $(g, \eta) \in T^*G$ by
$\ell_{(g, \eta)}$ and $r_{(g, \eta)}$, respectively.
\begin{proposition}
    \label{proposition:TstarGActsAsWell}%
    Let $R, R' \ge 0$.
    \begin{propositionlist}
    \item \label{item:CotangentRepresentationLeft}%
        The pullbacks with left multiplications on $T^*G$ yield
        representations
        \begin{equation}
            \label{eq:CotangentRepresentationLeft}
            \ell^*
            \colon
            T^*G
            \longrightarrow
            L
            \big(
                \PolRC(T^*G)
            \big)
        \end{equation}
        by continuous linear maps and $\PolRC(T^*G)$ consists of
        corresponding entire vectors.
    \item \label{item:CotangentRepresentationRight}%
        Assume furthermore $R < 1$. The pullbacks with right
        multiplications on $T^*G$ yield representations
        \begin{equation}
            \label{eq:CotangentRepresentationRight}
            r^*
            \colon
            T^*G
            \longrightarrow
            L
            \big(
            \PolRC(T^*G)
            \big)
        \end{equation}
        by continuous linear maps and $\PolRC(T^*G)$ consists of
        corresponding entire vectors.
    \end{propositionlist}
\end{proposition}
\begin{proof}
    Let $(g, \eta), (h, \chi) \in G \times \liealg{g}^*$, $\phi \in \Entire(G)$ and $\xi_1, \ldots, \xi_k \in \liealg{g}$. We note the
    explicit formulae
    \begin{align*}
        \ell_{(g, \eta)}^*
        \big(
            \phi
            \tensor
            \xi_1 \vee \cdots \vee \xi_k
        \big)
        \at[\Big]{(h, \chi)}
        &=
        \ell_g^* \phi
        \tensor
        \big(
            \eta(\xi_1)
            \cdots
            \eta(\xi_k)
            \cdot
            1
            +
            \Ad_{g^{-1}} \xi_1
            \vee \cdots \vee
            \Ad_{g^{-1}} \xi_k
        \big)
        \at[\Big]{(h, \chi)}, \\
        r_{(h, \chi)}^*
        \big(
            \phi
            \tensor
            \xi_1 \vee \cdots \vee \xi_k
        \big)
        \at[\Big]{(g, \eta)}
        &=
        r_h^* \phi
        \tensor
        \xi_1 \vee \cdots \vee \xi_k
        \at[\Big]{(g, \eta)}
        +
        r_h^* \phi
        \cdot
        \chi
        \big(
            \Ad_{\inv(\argument)}
            \xi_1
        \big)
        \cdots
        \chi
        \big(
            \Ad_{\inv(\argument)}
            \xi_k
        \big)
        \tensor
        1
        \at[\Big]{(g, \eta)}
    \end{align*}
    for the pullbacks. From here, the continuity estimates can be
    handled by the same techniques we have employed throughout the
    paper, see in particular Theorem~\ref{theorem:EntireRep} and the
    upcoming Lemma~\ref{lemma:StdContinuityMixed}. Notice that the
    maps
    \begin{equation*}
        \Phi_\xi
        \colon
        G
        \ni
        g
        \mapsto
        \chi
        \big(
            \Ad_g
            \xi
        \big)
        \in
        \field{C}
    \end{equation*}
    are nothing be representative functions, which we have studied in
    Theorem~\ref{theorem:RepresentingFunctions}. This explains the
    additional requirement of $R < 1$ in
    \ref{item:CotangentRepresentationRight}
\end{proof}

After these abstract considerations, we derive a more explicit
description of $\PolRC(T^*G)$. A first observation is that for finite
dimensional vector spaces $V$, we have an analogue of
\eqref{eq:CanonicalIsoJ}, implementing the isomorphism
$\Sym^\bullet(V^*) \cong \Pol^\bullet(V)$ of graded vector spaces. We
shall identify both without further comment in the sequel. This
moreover gives $\seminorm{p}_{R', c} = \seminorm{q}_{R', c}$ for
$R' \in \field{R}$ and $c \ge 0$. Here the slightly different
prefactors match, as differentiation produces another factorial. In
particular, the subspace topology induced by
$\SymR(V) \subseteq \Entire(V^*)$ is the $\SymR$-topology again. The
idea is now that $\Entire(V^*)$ is the completion of $\SymR(V)$:
\begin{lemma}
    \label{lemma:EntireIsCompletionOfPolynomials}%
    Let $V$ be a finite dimensional vector space over $\field{C}$ and
    $R' \ge 0$. Then we have
    \begin{equation}
        \label{eq:EntireIsCompletionOfPolynomials}
        \SymRC(V)
        \cong
        \Entire[R'](V^*)
    \end{equation}
    with embedding given by the isomorphism $\mathcal{J}$. For $R' <
    0$ we have the inclusion $\Entire[R'](V^*) \subseteq \SymRC(V)$.
\end{lemma}
\begin{proof}
    Truncating the Taylor series of an entire function yields the
    desired polynomial approximation by elements of $\SymR(V)$ at
    once. Notably, this still works for negative $R'$, but the
    completeness of $\Entire[R'](V^*)$ relies on $R' \ge 0$, see again
    Theorem~\ref{theorem:Entire}, \ref{item:EntireFrechet}.
\end{proof}
\begin{corollary}
    \label{corollary:ObservableSmooth}%
    Let $R' \ge 0$ and $R \in \field{R}$.
    \begin{corollarylist}
        \item \label{item:ObservableSmooth}%
        Every function $\chi \in \PolRC(T^*G)$ is smooth.
        \item \label{item:ObservableSmoothTopology}%
        Let $R \ge 0$. The $(R,R')$-topology is finer than the
        $\Cinfty$-topology.
    \end{corollarylist}

\end{corollary}
\begin{proof}
    We invoke the triviality of the bundle $T^*G \cong G \times
    \liealg{g}$ once more: every derivative on $\Cinfty(T^*G)$
    factorizes into derivatives on $G$ and $\liealg{g}$, which commute
    with each other. Using this, \ref{item:ObservableSmoothTopology}
    is immediate from
    Lemma~\ref{lemma:EntireIsCompletionOfPolynomials} and
    Theorem~\ref{theorem:Entire}, \ref{item:EntireVsCinfty}. The case
    $R < 0$ in \ref{item:ObservableSmooth} is trivial, as $\Entire(G)
    \subseteq \Comega(G)$ by its very definition.
\end{proof}

Using the vector space structure of $\liealg{g}$, we arrive now at the
following explicit description of the completion $\PolRC(T^*G)$:
\begin{proposition}
    \label{proposition:ObservablesCompletion}%
    Let $R \in \field{R}$, $R' \ge 0$ and $\chi \in \PolRC(T^*G)$,
    viewed as an element of $\Cinfty(T^*G)$. Then there is a unique
    absolutely convergent decomposition
    \begin{equation}
        \label{eq:ObservableCompletionExplicit}
        \chi(g, \eta)
        =
        \sum_{k=0}^{\infty}
        \sum_{\alpha \in \field{N}_n^k}
        c_\alpha(g)
        \cdot
        \eta_1^{\alpha_1}
        \cdots
        \eta_n^{\alpha_n}
    \end{equation}
    for $g \in G$ and $\eta \in \liealg{g}^*$, where each
    $c_\alpha \colon G \longrightarrow \field{C}$ is $R$-entire and
    $c_\alpha$ is independent of the ordering of the entries of
    $\alpha \in \field{N}_n^k$. Moreover, for every $g \in G$ and as a
    function of $\eta$, \eqref{eq:ObservableCompletionExplicit} is an
    element of $\Entire[R'](\liealg{g})$.
\end{proposition}
\begin{proof}
    By Corollary~\ref{corollary:ObservableSmooth}, we know that
    $\chi \in \Cinfty(T^*G)$. Invoking \cite[Thm.~45.1]{treves:1967a},
    we find a summable sequence $(\phi_k) \subseteq \Entire(G)$ and
    another (not necessarily summable) sequence
    $(\psi_k) \subseteq \Entire[R'](\liealg{g})$ s.t.
    \begin{equation*}
        \label{eq:ObservableCompletionProof}
        \chi
        =
        \sum_{k=1}^{\infty}
        \phi_k
        \tensor
        \psi_k,
    \end{equation*}
    where the series converges absolutely in the projective tensor
    product topology. Given a $g \in G$ we use the product structure
    $T^*G \cong G \times \liealg{g}^*$ to define
    \begin{equation*}
        c_\alpha(g)
        =
        \Lie_{X^{\liealg{g}^*}_\alpha}
        \chi(g, \argument)
        =
        \sum_{k=1}^{\infty}
        \phi_k(g)
        \tensor
        \Lie_{X^{\liealg{g}^*}_\alpha}
        \psi_k
    \end{equation*}
    as the $\alpha$-th Lie-Taylor coefficient of
    $\chi(g, \argument) \colon \liealg{g}^* \longrightarrow
    \field{C}$. Note that this way, the $c_\alpha(g)$ do indeed have
    the claimed symmetry property, as the Lie derivatives on
    $\liealg{g}$ are just partial derivatives corresponding to the
    basis we chose. By summability of $(\phi_k)$, we moreover have
    $c_\alpha \in \Entire(G)$. Interchanging the series, it is
    straightforward to check that the right hand side of
    \eqref{eq:ObservableCompletionExplicit} indeed converges
    absolutely to the function $\chi$ we started with. This also gives
    the remaining statement, as each
    $\psi_k \in \Entire[R'](\liealg{g})$.
\end{proof}
\begin{corollary}
    \label{corollary:EntireAnalytic}%
    Let $R \in \field{R}$ and $R' \ge 0$. We have the inclusion
    $\PolRC(T^*G) \subseteq \Comega(T^*G)$ of algebras.
\end{corollary}

%
%

\section{Continuity Results}
\label{sec:Continuity}%

We begin our considerations on continuity by restating
\cite[Prop.~3.2, \textit{ii.)}, and
Prop.~3.6]{esposito.stapor.waldmann:2017a} on the Lie algebra star
product $\staralg$. By \eqref{eq:StdFactorization} this star product
is the restriction of $\starstd$ to the second tensor factors,
i.e. polynomials in the momenta only. For convenience, we already
specialize to the situation we are interested in, namely
finite-dimensional Lie algebras instead of general asymptotic estimate
algebras \cite{esposito.stapor.waldmann:2017a}:
\begin{proposition}
    \label{proposition:StdContinuityAlgebra}%
    Let $\liealg{g}$ be a finite-dimensional Lie algebra and
    $R' \ge 1$. The Lie algebra star product
    \begin{equation}
        \label{eq:StdAlgebra}
        \staralg
        \colon
        \SymR(\liealg{g}) \times \SymR(\liealg{g})
        \longrightarrow
        \SymR(\liealg{g})
    \end{equation}
    is well-defined and continuous. More precisely, we have the
    estimate
    \begin{equation}
        \label{eq:StdContinuityAlgebra}
        \seminorm{p}_{R', c'}
        \big(
            \xi \staralg \eta
        \big)
        \le
        \seminorm{p}_{R', \tilde{c}'}(\xi)
        \cdot
        \seminorm{p}_{R', \tilde{c}'}(\eta)
    \end{equation}
    for $\xi, \eta \in \SymR(\liealg{g})$, $c' \ge 1$ and
    $\tilde{c}' = 32(\hbar + 1)c'$. In particular, $\tilde{c}'$ is a
    continuous function of $\hbar$. Moreover, the map
    \begin{equation}
        \label{eq:StdHolomorphyAlgebra}
        \field{C}
        \ni
        \hbar
        \mapsto
        \xi \staralg \eta
        \in
        \SymR(\liealg{g})
    \end{equation}
    is entire for all $\xi, \eta \in \SymR(\liealg{g})$.
\end{proposition}

Another look at \eqref{eq:StdFactorization} reveals that there is only
one other interesting type of product to consider: a polynomial in
$\Sym(\liealg{g})$ on the left and a function on $G$ on the right.
The crucial idea is that the mixed product corresponds to a dual
pairing in the spirit of \cite{goodman:1970a} and
\cite{goodman:1971a}, where only the sum $R+R'$ of the parameters in
$\PolR(T^*G)$ matters for continuity, but not their individual
values. Note, however, that we do more than just pair: instead of
applying the differential operator to the function, we commute the
differential operators with the left multiplication, yielding numerous
additional contributions. Nevertheless, this yields the following
continuity result:
\begin{lemma}
    \label{lemma:StdContinuityMixed}%
    Let $G$ be a Lie group and $R, R' \in \field{R}$ with
    $R + R' \geq 1$. The restricted standard-ordered star product
    \begin{equation}
        \label{eq:StdStarMixed}
        \starstd
        \colon
        \big( \mathbb{1} \tensor \SymR(\liealg{g}) \big)
        \times
        \big(\Entire(G) \tensor 1 \big)
        \longrightarrow
        \PolR(T^*G)
    \end{equation}
    is well-defined and continuous with respect to the
    $R, R'$-topology. More precisely, in each symmetric degree and for
    $c, c' \ge 1$ we have the estimate
    \begin{equation}
        \label{eq:StdContinuityMixed}
        \big(
            \seminorm{q}_{R, c}
            \tensor
            \seminorm{p}_{R', c'}
        \big)
        \big(
            (\mathbb{1} \tensor \xi)
            \starstd
            (\phi \tensor 1)
        \big)
        \le
        2
        \cdot
        \seminorm{p}_{R', d'}(\xi)
        \cdot
        \seminorm{q}_{R, 2c}(\phi)
    \end{equation}
    for $\phi \in \Entire(G)$, $\xi \in \SymRC(\liealg{g})$ and
    $d' = \max\{2 \hbar, c'\}$, which depends continuously on
    $\hbar$. In particular, the map
    \begin{equation}
        \label{eq:StdHolomorphyMixed}
        \field{C}
        \ni
        \hbar
        \mapsto
        (\mathbb{1} \tensor \xi)
        \starstd
        (\phi \tensor 1)
        \in
        \PolRC(T^*G)
   \end{equation}
    is entire for all $\phi \in \Entire(G)$ and
    $\xi \in \SymRC(\liealg{g})$.
\end{lemma}
\begin{proof}
    Let $\phi \in \Entire(G)$, $k \in \field{N}_0$,
    $\big( \basis{e}_1, \ldots, \basis{e}_n \big)$ be a basis of
    $\liealg{g}$ corresponding to the $\ell^1$-norm $\seminorm{p}$ we
    chose and $1 \leq i_1, \ldots, i_k \leq n$. Recall that by
    Proposition~\ref{proposition:StdFactorizations},
    \ref{item:xiStarphi}, we have the explicit formula
    \begin{align*}
        \big(
            \mathbb{1}
            &\tensor
            \basis{e}_{i_1} \vee \cdots \vee \basis{e}_{i_k}
        \big)
        \starstd
        \big(
            \phi \tensor 1
        \big) \\
        &=
        \sum_{p=0}^k
        \bigg(\frac{\hbar}{\I}\bigg)^{p}
        \frac{1}{p! \; (k-p)!}
        \sum_{\sigma \in S_k}
        \Lie_{X_{i_{\sigma(p)}}} \cdots \Lie_{X_{i_{\sigma(1)}}} \phi
        \tensor
        \basis{e}_{\sigma(p+1)} \vee \cdots \vee \basis{e}_{\sigma(k)}.
    \end{align*}
    First note
    \begin{align*}
        \sum_{\sigma \in S_k}
        \seminorm{q}_{\alpha}
        \Big(
            \Lie_{X_ {i_{\sigma(p)}}}
            \cdots
            \Lie_{X_{i_{\sigma(1)}}} \phi
        \Big)
        &=
        \sum_{\sigma \in S_k}
        \abs[\Big]
        {
            \big(\Lie_{X_ {\alpha_\ell}}
            \cdots
            \Lie_{X_ {\alpha_1}}
            \Lie_{X_{i_{\sigma(p)}}}
            \cdots
            \Lie_{X_{i_{\sigma(1)}}}
            \phi \big)(\E)
        } \\
        &=
        \sum_{\sigma \in S_k}
        \seminorm{q}_{
          (i_{\sigma(1)},
          \ldots,
          i_{\sigma(p)},
          \alpha_1,
          \ldots,
          \alpha_\ell)
        }(\phi)
    \end{align*}
    for $\alpha \in \{ 1, \ldots, n \}^\ell$. In the sequel we write
    $(i_{\sigma(1)}, \ldots, i_{\sigma(p)}, \alpha)$ for
    $(i_{\sigma(1)}, \ldots, i_{\sigma(p)}, \alpha_1, \ldots,
    \alpha_\ell)$ by slight abuse of notation. Note that we sum over
    $S_k$, but only use the first $p$ values of the permutation. For
    the other factor we use
    \cite[Lem.~A.1]{cahen.gutt.waldmann:2020a}, which essentially says
    that projective tensor products of $\ell^1$-norms yield
    $\ell^1$-norms associated to the product bases. We write
    $\seminorm{p}^k$ for the $k$-th projective tensor power of the
    $\ell^1$-norm $\seminorm{p}$. This gives
    \begin{equation*}
        \seminorm{p}^{k-p}
        \big(
            \basis{e}_{\sigma(p+1)}
            \vee \cdots \vee
            \basis{e}_{\sigma(k)}
        \big)
        =
        1
        =
        \seminorm{p}^{k}
        \big(
            \basis{e}_{i_1}
            \vee \cdots \vee
            \basis{e}_{i_k}
        \big).
    \end{equation*}
    Here it is important that we use the $\ell^1$-norm $\seminorm{p}$
    with respect to the above basis, otherwise we would only get
    estimates instead of equalities. This implies
    \begin{equation*}
        \seminorm{p}_{R', c'}
        \big(
            \basis{e}_{\sigma(p+1)}
            \vee \cdots \vee
            \basis{e}_{\sigma(k)}
        \big)
        =
        (k-p)!^{R'} \; c'^{k-p}
        =
        \bigg(
            \frac{(k-p)!}{k!}
        \bigg)^{R'}
        c'^{-p} \;
        \seminorm{p}_{R', c'}
        \big(
            \basis{e}_{i_1}
            \vee \cdots \vee
            \basis{e}_{i_k}
        \big).
    \end{equation*}
    Let now $R, R' \le 1$ such that $R+R' \ge 1$ and $c, c' \ge
    1$. Due to
    \begin{equation*}
        \seminorm{q}_{R', c'}
        \le
        \seminorm{q}_{R', \max\{c', 2 \hbar\}}
    \end{equation*}
    we may assume $c' \ge 2 \hbar$ without loss of
    generality. Otherwise we just estimate $c'$ by a yet another
    polynomial weight $\tilde{c} \ge 2 \hbar$ in the very first
    step. With this in mind, we obtain
    \begin{align*}
        &\Big(
            \seminorm{q}_{R, c}
            \tensor
            \seminorm{p}_{R', c'}
        \Big)
        \big(
            (\mathbb{1} \tensor \basis{e}_{i_1} \vee \cdots \vee \basis{e}_{i_k})
            \starstd
            (\phi \tensor 1)
        \big) \\
        &\le
        \sum_{p=0}^k
        \frac{\hbar^p}{p! \; (k-p)!}
        \sum_{\sigma \in S_k}
        \seminorm{q}_{R, c}
        \Big(
            \Lie_{X_{i_{\sigma(p)}}} \cdots \Lie_{X_{i_{\sigma(1)}}} \phi
        \Big)
        \cdot
        \seminorm{p}_{R', c'}
        \big(
            \basis{e}_{\sigma(p+1)}
            \vee \cdots \vee
            \basis{e}_{\sigma(k)}
        \big) \\
        &=
        \sum_{p=0}^k
        \frac{\hbar^p}{p! \; (k-p)!}
        \sum_{\sigma \in S_k}
        \sum_{\ell = 0}^\infty
        \ell!^{R-1} \; c^\ell
        \sum_{\alpha \in \{1, \ldots, n\}^\ell}
        \seminorm{q}_{\alpha}
        \Big(\Lie_{X_{i_{\sigma(p)}}} \cdots \Lie_{X_{i_{\sigma(1)}}} \phi \Big)
        \cdot
        \seminorm{p}_{R', c'}
        \big(
            \basis{e}_{\sigma(p+1)}
            \vee \cdots \vee
            \basis{e}_{\sigma(k)}
        \big) \\
        &\le
        \sum_{p=0}^k
        \frac{\hbar^p}{p! \; (k-p)!}
        k!
        \sum_{\ell = 0}^\infty
        \ell!^{R-1} \; c^\ell
        \sum_{\beta \in \{1, \ldots, n\}^{\ell+p}}
        \seminorm{q}_{\beta}(\phi)
        \cdot
        \bigg(
            \frac{(k-p)!}{k!}
        \bigg)^{R'}
        c'^{-p} \;
        \seminorm{p}_{R', c'}
        \big(
            \basis{e}_{i_1}
            \vee \cdots \vee
            \basis{e}_{i_k}
        \big) \\
        &=
        \seminorm{p}_{R', c'}
        \big(
            \basis{e}_{i_1}
            \vee \cdots \vee
            \basis{e}_{i_k}
        \big)
        \sum_{p=0}^k
        \frac{\hbar^p \; k!^{1-R'}}{p!^R \; (k-p)!^{1-R'} \; c'^p}
        \sum_{m=p}^\infty
        p!^{R-1} \; (m-p)!^{R-1} \; c^{m-p}
        \sum_{\beta \in \{1, \ldots, n\}^{m}}
        \seminorm{q}_{\beta}(\phi) \\
        &\overset{(*)}{\le}
        \seminorm{p}_{R', c'}
        \big(
            \basis{e}_{i_1}
            \vee \cdots \vee
            \basis{e}_{i_k}
        \big)
        \sum_{p=0}^k
        \binom{k}{p}^{1-R'}
        \frac{1}{p!^{R'+R-1}}
        \frac{\hbar^p}{c'^p}
        \sum_{m=p}^\infty
        m!^{R-1} \;
        \big(
            2^{1-R} c
        \big)^{m}
        \sum_{\beta \in \{1, \ldots, n\}^{m}}
        \seminorm{q}_{\beta}(\phi) \\
        &\overset{(*')}{\le}
        \seminorm{p}_{R', c'}
        \big(
            \basis{e}_{i_1}
            \vee \cdots \vee
            \basis{e}_{i_k}
        \big)
        2^{k(1-R')}
        \sum_{p=0}^k
        \frac{\hbar^p}{c'^p}
        \cdot
        \seminorm{q}_{R, 2^{1-R}c}(\phi) \\
        &\le
        \seminorm{p}_{R', 2^{1-R'} c'}
        \big(
            \basis{e}_{i_1}
            \vee \cdots \vee
            \basis{e}_{i_k}
        \big)
        \cdot
        \seminorm{q}_{R, 2^{1-R}c}(\phi)
        \sum_{p=0}^\infty
        2^{-p} \\
        &\le
        \seminorm{p}_{R', 2c'}
        \big(
            \basis{e}_{i_1}
            \vee \cdots \vee
            \basis{e}_{i_k}
        \big)
        \cdot
        \seminorm{q}_{R, 2c}(\phi)
        \cdot
        2,
    \end{align*}
    where we have used $R \le 1$ as well as $c \ge 1$ in $(*)$, then
    $R' \le 1$ in $(*')$, and $c' \ge 2\hbar$ in the final
    estimate. Note again that we can make this assumption on $c'$
    without loss of generality by
    \eqref{eq:RSeminormInequalities}. This observation and the
    analogous statement \eqref{eq:EntireSeminormInequalities} for
    $\Entire$ also gives the worse estimate
    \eqref{eq:StdContinuityMixed} from what we have computed. If we
    have $R' \ge 1$, we estimate the binomial coefficient in the step
    $(*')$ by $1$ instead, which once again implies
    \eqref{eq:StdContinuityMixed}. In the case that $R \ge 1$ we note
    $(m-p)!^{R-1} \le m!^{R-1}$, yielding
    \begin{align*}
        \ldots
        &\overset{(*)}{\le}
        \seminorm{p}_{R', c'}
        \big(
            \basis{e}_{i_1}
            \vee \cdots \vee
            \basis{e}_{i_k}
        \big)
        \sum_{p=0}^k
        \binom{k}{p}^{1-R'}
        \frac{1}{p!^{R'}}
        \frac{\hbar^p}{c'^p}
        \sum_{m=0}^\infty
        m!^{R-1} \;
        c^m
        \sum_{\beta \in \{1, \ldots, n\}^{m}}
        \seminorm{q}_{\beta}(\phi) \\
        &=
        \label{eq:StdContinuityMixedProof1}
        \seminorm{p}_{R', c'}
        \big(
            \basis{e}_{i_1}
            \vee \cdots \vee
            \basis{e}_{i_k}
        \big)
        \cdot
        \seminorm{q}_{R, c}(\phi)
        \cdot
        \sum_{p=0}^k
        \binom{k}{p}^{1-R'}
        \frac{1}{p!^{R'}}
        \frac{\hbar^p}{c'^p}
        \tag{\dag} \\
        &=
        \label{eq:StdContinuityMixedProof2}
        \seminorm{p}_{R', c'}
        \big(
            \basis{e}_{i_1}
            \vee \cdots \vee
            \basis{e}_{i_k}
        \big)
        \cdot
        \seminorm{q}_{R, c}(\phi)
        \cdot
        \sum_{p=0}^k
        \bigg(
            \frac{k!}{(k-p)!}
        \bigg)^{1-R'}
        \frac{1}{p!}
        \frac{\hbar^p}{c'^p}
        \tag{\ddag}.
    \end{align*}
    From here, \eqref{eq:StdContinuityMixedProof1} gives the case
    $R' \ge 1$ and \eqref{eq:StdContinuityMixedProof2} the case
    $R' \le 1$ in the same fashion as before. Note that passing to the
    series in $p$ makes our estimate independent of the symmetric
    order $k$. Thus we have shown \eqref{eq:StdContinuityMixed} on
    generators. Consider now an arbitrary function
    $P \in \SymRC(\liealg{g})$ in the left factor. Expanding $P$ in
    the induced basis of $\Sym^\bullet(\liealg{g})$ corresponding to
    the basis of $\liealg{g}$ we chose earlier gives
    \begin{equation*}
        P
        =
        \sum_{k=0}^{\infty}
        \sum_{i_1 \le \cdots \le i_k = 1}^n
        a^{i_1 \cdots i_k}
        \basis{e}_{i_1}
        \vee \cdots \vee
        \basis{e}_{i_k}
        \in
        \SymR(\liealg{g}).
    \end{equation*}
    By distributivity of the standard-ordered star product this now
    implies
    \begin{align*}
        &\Big(
            \seminorm{q}_{R, c}
            \tensor
            \seminorm{p}_{R', c'}
        \Big)
        \big(
            (\mathbb{1} \tensor P)
            \starstd
            (\phi \tensor 1)
        \big) \\
        &\le
        2
        \cdot
        \seminorm{q}_{R, 2c}(\phi)
        \sum_{k=0}^{\infty}
        \sum_{i_1 \le \cdots \le i_k = 1}^n
        \abs[\big]{a^{i_1 \cdots i_k}}
        \cdot
        \seminorm{p}_{R', 2c'}
        \big(
            \basis{e}_{i_1}
            \vee \cdots \vee
            \basis{e}_{i_k}
        \big) \\
        &=
        2
        \cdot
        \seminorm{q}_{R, 2c}(\phi)
        \cdot
        \seminorm{p}_{R', 2c'}(P),
    \end{align*}
    where we once again utilized
    \cite[Lem.~A.1]{cahen.gutt.waldmann:2020a} to first infer the
    ``orthogonality''
    \begin{equation*}
       \seminorm{p}^k
       \bigg(
       \sum_{i_1 \le i_2 \le \ldots \le i_k}
       a^{i_1 \ldots i_k}
       \basis{e}_{i_1} \vee \cdots \vee \basis{e}_{i_k}
       \bigg)
       =
       \sum_{i_1 \le i_2 \le \ldots \le i_k}
       \abs[\big]{a^{i_1 \ldots i_k}}
       \seminorm{p}^k
       (\basis{e}_{i_1} \vee \cdots \vee \basis{e}_{i_k})
    \end{equation*}
    within a \emph{fixed} symmetric degree. Thus we have shown that
    \eqref{eq:StdContinuityMixed} also holds for arbitrary polynomial
    functions $P$. This finally implies the continuity of the standard
    ordered star product. For the holomorphy first note that our
    estimate works for $\hbar$ in a locally uniform and bounded way by
    continuity of the involved weights with respect to $\hbar$. Taking
    another look at the formula for the star product from
    \eqref{eq:StdExplicitMixed} we see that the star product is an
    absolutely convergent power series in $\hbar$, i.e. the limit of
    polynomials in $\hbar$, which are holomorphic. By our estimate the
    corresponding sequence is Cauchy with respect to the locally
    uniform topology, i.e. it converges to some element in the
    completion and that element is vector-valued holomorphic, as well.
\end{proof}

We have gathered all the necessary ingredients to prove the continuity
of the full star product. Notably, the sharp condition $R' \ge 1$ from
Proposition~\ref{proposition:StdContinuityAlgebra} breaks the symmetry
between $R$ and $R'$ from Lemma~\ref{lemma:StdContinuityMixed} and
reduces the condition $R+R' \ge 1$ to $R \ge 0$. Investing moreover
the continuity of the pointwise product on $\Entire(G)$ now yields our
main result:
\begin{theorem}[Continuity of $\starstd$]
    \label{theorem:StdContinuity}%
    Let $G$ be a Lie group, $R \ge 0$ and $R' \ge 1$. The full
    standard-ordered star product
    \begin{equation}
        \label{eq:StdFull}
        \starstd
        \colon
        \PolR(T^*G) \times \PolR(T^*G)
        \longrightarrow
        \PolR(T^*G)
    \end{equation}
    is well-defined and continuous, extending to a continuous product
    \begin{equation}
        \label{eq:StdFullOnCompletion}
        \starstd
        \colon
        \PolRC(T^*G) \times \PolRC(T^*G)
        \longrightarrow
        \PolRC(T^*G).
    \end{equation}
    More precisely, for $c, c' \ge 1$ there is a $d \ge 1$, which is
    continuous with respect to $\hbar$, such that
    \begin{align}
        \label{eq:StdContinuity}
        \big(
            \seminorm{q}_{R, c}
            \tensor
            \seminorm{p}_{R', c'}
        \big)
        (P \starstd Q)
        \le
        2
        \cdot
        \big(
            \seminorm{q}_{R, d}
            \tensor
            \seminorm{p}_{R', d}
        \big)
        (P)
        \cdot
        \big(
            \seminorm{q}_{R, d}
            \tensor
            \seminorm{p}_{R', d}
        \big)
        (Q)
    \end{align}
    holds for $P, Q \in \PolRC(T^*G)$.
\end{theorem}
\begin{proof}
    We first consider factorizing functions.  Let
    $\phi, \psi \in \Entire(G)$, $\eta \in \SymR(\liealg{g})$ as well
    as $\xi_1, \ldots, \xi_k \in \liealg{g}$. By
    Proposition~\ref{proposition:StdFactorizations},
    \ref{item:GeneralStarstdGeneral}, the full star product can be
    written as
    \begin{align*}
        &(\phi \tensor \xi_1 \vee \cdots \vee \xi_k)
        \starstd
        (\psi \tensor \eta) \\
        &=
        \sum_{p=0}^k
        \bigg(\frac{\hbar}{\I}\bigg)^{p}
        \frac{\phi}{p! \; (k-p)!}
        \sum_{\sigma \in S_k}
        \Lie_{X_{\xi_{\sigma(1)}}} \cdots \Lie_{X_{\xi_{\sigma(p)}}} \psi
        \tensor
        (\xi_{\sigma(p+1)} \vee \cdots \vee \xi_{\sigma(k)})
        \staralg
        \eta.
    \end{align*}
    Note that, compared to \eqref{eq:StdExplicitMixed}, we left
    multiply with the function $\phi$ in the first tensor factor and
    compose with the Lie algebra star product in the second one. Let
    $c, c' \ge 1$ and write $\tilde{c}' = 16(\hbar + 1)c'$. Using
    \eqref{eq:StdContinuityAlgebra}, \eqref{eq:StdContinuityMixed} as
    well as the continuity estimate for pointwise products from
    \eqref{eq:EntireMultiplication} gives
    \begin{align*}
        &\Big(
            \seminorm{q}_{R, c}
            \tensor
            \seminorm{p}_{R', c'}
        \Big)
        \big(
            (\phi \tensor \xi_{1} \vee \cdots \vee \xi_{k})
            \starstd
            (\psi \tensor \eta)
        \big) \\
        &\le
        \sum_{p=0}^k
        \frac{\hbar^p}{p! \; (k-p)!}
        \sum_{\sigma \in S_k}
        \seminorm{q}_{R, c}
        \Big(
            \phi
            \cdot
            \Lie_{X_{\xi_{\sigma(1)}}} \cdots \Lie_{X_{\xi_{\sigma(p)}}}
            \psi
        \Big)
        \seminorm{p}_{R', c'}
        \big(
            \xi_{\sigma(p+1)} \vee \cdots \vee \xi_{\sigma(k)}
            \staralg
            \eta
        \big) \\
        &\le
        \sum_{p=0}^k
        \frac
        {
            \hbar^p
            \cdot
            \seminorm{q}_{R, 2^{R}c} (\phi)
            \cdot
            \seminorm{p}_{R', \tilde{c}'}
            \big(
                \eta
            \big)
        }
        {p! \; (k-p)!}
        \sum_{\sigma \in S_k}
        \seminorm{q}_{R, 2^{R}c}
        \Big(
            \Lie_{X_{\xi_{\sigma(1)}}} \cdots \Lie_{X_{\xi_{\sigma(p)}}} \psi
        \Big)
        \seminorm{p}_{R', \tilde{c}'}
        \big(
            \xi_{\sigma(p+1)} \vee \cdots \vee \xi_{\sigma(k)}
        \big).
    \end{align*}
    Thus what remains to be estimated is
    \begin{equation*}
        \sum_{p=0}^k
        \frac{\hbar^p}{p! \; (k-p)!}
        \sum_{\sigma \in S_k}
        \seminorm{q}_{R, 2^{R}c}
        \Big(
            \Lie_{X_{\xi_{\sigma(1)}}} \cdots \Lie_{X_{\xi_{\sigma(p)}}} \psi
        \Big)
        \seminorm{p}_{R', \tilde{c}'}
        \big(
            \xi_{\sigma(p+1)} \vee \cdots \vee \xi_{\sigma(k)}
        \big),
    \end{equation*}
    which is exactly what we obtained by applying the triangle
    inequality to the mixed product
    \begin{equation*}
        \Big(
            \seminorm{q}_{R, 2^{R} c}
            \tensor
            \seminorm{p}_{R', \tilde{c}'}
        \Big)
        \big(
            (\mathbb{1} \tensor \xi_1 \vee \cdots \vee \xi_k)
            \starstd
            (\phi \tensor 1)
        \big).
    \end{equation*}
    Thus we can utilize our estimate from
    Lemma~\ref{lemma:StdContinuityMixed} to obtain
    \eqref{eq:StdContinuity}, taking $d$ as the largest coefficient we
    obtain upon putting everything together, which is obviously
    continuous in $\hbar$ as a pointwise maximum of continuous
    functions. Then the bilinear version of the argument as in the end
    of Lemma~\ref{lemma:StdContinuityMixed} extends
    \eqref{eq:StdContinuity} to arbitrary polynomial functions, as all
    other factors were already in full generality. Finally, the usual
    infimum argument (see e.g. \cite[Prop.~43.4]{treves:1967a}) for
    projective tensor products gives this estimate also for arbitrary
    mixed tensors, which implies continuity of the standard-ordered
    star product $\starstd$ on the entire observable algebra
    $\PolR(T^*G)$. From here it extends to the completion by
    continuity, preserving the estimates \eqref{eq:StdContinuity}.
\end{proof}

The second main statement is that the star product we obtained is a
holomorphic deformation in the following sense:
\begin{theorem}[Holomorphic dependence on $\hbar$]
    \label{theorem:HolomorphicDependenceHbar}%
    Let $G$ be a connected Lie group and let $R \ge 0$ and $R' \ge 1$.
    Then
    \begin{equation}
        \label{eq:StdHolomorphy}
        \field{C} \ni \hbar
        \; \mapsto \;
        P \starstd Q
        \in \PolRC(T^*G)
    \end{equation}
    is entire for all $P, Q \in \PolRC(T^*G)$. Its Taylor series in
    $\hbar$ coincides with the formal star product in the sense that
    the $\hbar^k$-term is given by \eqref{eq:StdExplicit}.
\end{theorem}
\begin{proof}
    As long as $P, Q \in \PolR(T^*G)$, their star product is a
    polynomial in $\hbar$, thus entire. For general elements $P$ and
    $Q$ in the completion, we can once again estimate locally
    uniformly in $\hbar$ and our explicit formula as well as
    \cite[Lem.~2.8 with
    $z = -\I \hbar$]{esposito.stapor.waldmann:2017a} then imply that
    we have polynomial partial sums, i.e. vector-valued holomorphic
    functions. Together, this implies vector-valued holomorphy of the
    full star product for fixed factors. The second statement is clear
    for elements $P, Q \in \PolR(T^*G)$ and extends to the completion
    by virtue of Proposition~\ref{proposition:ObservablesCompletion}.
\end{proof}

A first application of this continuity result is the continuity of the
standard-ordered quantization map:
\begin{corollary}
    \label{corollary:StdRepContinuity}%
    Let $G$ be a Lie group and $R \ge 0$ and $R' \ge 1$.  The
    standard-ordered quantization map $\stdrep$ yields a continuous
    bilinear map
    \begin{equation}
        \label{eq:StdRep}
        \stdrep\colon
        \PolRC(T^*G) \times \Entire(G)
        \ni (f, \phi)
        \; \mapsto \;
        \stdrep(f)\phi
        \in \Entire(G).
    \end{equation}
    In particular, every operator $\stdrep(f)$ with
    $f \in \PolRC(T^*G)$ is a continuous endomorphism of $\Entire(G)$.
\end{corollary}
\begin{proof}
    According to \eqref{eq:RepFromStar} we have for
    $f \in \PolR(T^*G)$ and $\phi \in \Entire(G)$
    \begin{equation*}
        \stdrep(f)\phi
        =
        \iota^*(f \starstd \pi(\phi)),
    \end{equation*}
    which is a composition of the continuous linear maps $\pi^*$ and
    $\iota^*$, see Proposition~\ref{proposition:IotaPiContinuous}, and
    the continuous bilinear star product. As usual, this extends the
    completion.
\end{proof}

By invoking the semiclassical limit we immediately obtain the
continuity of the Poisson bracket:
\begin{corollary}
    \label{corollary:StdContinuityPoisson}%
    Let $G$ be a Lie group, $R \ge 0$ and $R' \ge 1$. The Poisson
    bracket
    \begin{equation}
        \label{eq:StdPoisson}
        \big\{
            \argument, \argument
        \big\}
        \colon
        \PolR(T^*G) \times \PolR(T^*G)
        \longrightarrow
        \PolR(T^*G)
    \end{equation}
    is well-defined and continuous. Moreover, the explicit formula
    \eqref{eq:NeumaierLaplacianIsPoisson} extends to the completion
    $\PolRC(T^*G)$ order by order.
\end{corollary}

Of course, both Corollary~\ref{corollary:StdRepContinuity} and
Corollary~\ref{corollary:StdContinuityPoisson} can be shown by direct
estimation and the explicit formulas
\eqref{eq:NeumaierLaplacianIsPoisson} and \eqref{eq:StdRepExplicit},
as well. Notably, this extends the statements to arbitrary values of
$R$ and $R'$. The underlying reason for this is that each of the
mappings is an honest differential operator, i.e. only \emph{finitely}
many differentiations have to be estimated at once. Analogously, the
same is true for the bidifferential operators $D_k \in \Diffop(T^*G)$
given by
\begin{equation}
    \label{eq:StarStdBidiff}
    D_k(P, Q)
    =
    \frac{\D^k}{\D \hbar^k}
    \Big(
        P
        \starstd
        Q
    \Big)
    \at[\bigg]{\hbar = 0}
\end{equation}
for $P, Q \in \Cinfty(T^*G)$ and $k \in \field{N}_0$.

After having established the continuity of the structure maps for
classical mechanics and its standard-ordered quantization, we turn
towards other ordering prescriptions obtained by means of the Neumaier
operator. Instead of directly deriving continuity estimates for the
considerably more complicated formulas, we show the continuity of the
$\kappa$-Neumaier operators. From
Proposition~\ref{proposition:NeumaierIsMixedStd} we immediately get
the continuity of $N^2$ and, ultimately, the continuity of $N_\kappa$
for all $\kappa \in \field{R}$:
\begin{proposition}
    \label{proposition:NeumaierContinuity}%
    Let $G$ be a Lie group, $\kappa \in \field{R}$ and
    $R, R' \in \field{R}$ with $R+R' \ge 1$.
    \begin{propositionlist}
    \item \label{item:NeumaierStetig} The $\kappa$-Neumaier operator
        \begin{equation}
            \label{eq:NeumaierContinuity}
            N_\kappa
            \colon
            \PolR(T^*G)
            \longrightarrow
            \PolR(T^*G)
        \end{equation}
        is well-defined and continuous.
    \item \label{item:NeumaierExtends} The $\kappa$-Neumaier operator
        $N_\kappa$ extends by continuity to
        \begin{equation}
            \label{eq:NeumaierContinuityCompletion}
            N_\kappa
            \colon
            \PolRC(T^*G)
            \longrightarrow
            \PolRC(T^*G)
        \end{equation}
        and its explicit formula extends to the completion
        $\PolRC(T^*G)$ order by order.
    \item \label{item:NeumaierHolomorphic} For all $P \in
        \PolRC(T^*G)$ the map
        \begin{equation}
            \label{eq:NeumaierHolomorphy}
            \field{C}
            \ni
            \hbar
            \; \mapsto \;
            N_\kappa(P)
            \in
            \PolRC(T^*G)
        \end{equation}
        is entire.
    \item \label{item:KappaStarContinuous} Let now in addition
        $R \ge 0$ and $R' \ge 1$. Then the $\kappa$-ordered star
        product extends to a continuous multiplication
        \begin{equation}
            \label{eq:KappaStarContinuous}
            \stark\colon
            \PolRC(T^*G) \times \PolRC(T^*G)
            \longrightarrow
            \PolRC(T^*G).
        \end{equation}
    \item \label{item:KappaStarEntire} For $R \ge 0$ and $R' \ge 1$
        the $\kappa$-ordered star product yields an entire function
        \begin{equation}
            \label{eq:KappaOrderingHolomorphy}
            \field{C}
            \ni
            \hbar
            \; \mapsto \;
            P \stark Q
            \in
            \PolRC(T^*G)
        \end{equation}
        for all $P, Q \in \PolRC(T^*G)$.
    \end{propositionlist}
\end{proposition}
\begin{proof}
    This is a somewhat immediate consequence of
    \eqref{eq:NeumaierIsMixedStd}: first we get the continuity for
    $N^2 = N_2$ and all $\hbar$. Rescaling now $\hbar$ appropriately
    can be re-interpreted as a rescaling of $\kappa = 2$ in the
    continuity estimates for $N_2$. This gives the continuity for all
    $\kappa$. The second statement is then an immediate consequence of
    Proposition~\ref{proposition:ObservablesCompletion}. The entirety
    of $\hbar \mapsto N_\kappa(P)$ now follows from the entirety of
    \eqref{eq:StdHolomorphyMixed} and the formula
    \eqref{eq:NeumaierIsMixedStd}. Next,
    \begin{equation*}
        \label{eq:NeumaierContinuityProof}
        P \stark Q
        =
        N_{- \kappa}
        \big(
            (N_\kappa P)
            \starstd
            (N_\kappa Q)
        \big)
        \tag{$*$}
    \end{equation*}
    for $P, Q \in \PolR(T^*G)$ gives continuity of the
    $\kappa$-ordered star product $\stark$ as a composition of
    continuous maps. Being continuous, $\stark$ extends to the
    completion as usual.  Finally, \eqref{eq:NeumaierContinuityProof}
    implies entirety of the $\kappa$-ordered star products for fixed
    factors as a composition of entire functions.
\end{proof}
\begin{corollary}
    \label{corollary:WeylIsContinuous}%
    Let $G$ be a connected Lie group and let $R \ge 0$ and $R' \ge
    1$. Then the Weyl star product $\starweyl$ is a continuous
    multiplication
    \begin{equation}
        \label{eq:WeylContinuous}
        \starweyl\colon
        \PolRC(T^*G) \times \PolRC(T^*G)
        \longrightarrow
        \PolRC(T^*G)
    \end{equation}
    with entire dependence on $\hbar$.
\end{corollary}

\begin{proposition}
    \label{proposition:ContinuousMorphisms}%
    Let $\Phi \colon G \longrightarrow H$ be a covering map of Lie
    groups. Then pullback with the point transformation
    \begin{equation}
        \label{eq:SymmetriesContinuous}
        (T_* \Phi)^*
        \colon
        \PolRC(T^*H)
        \longrightarrow
        \PolRC(T^*G)
    \end{equation}
    is a continuous homomorphism with respect to the $\kappa$-ordered
    star products on $T^*G$ and $T^*H$, respectively.
\end{proposition}
\begin{proof}
    The fact that $(T_*\Phi)^*$ is a homomorphism holds in general
    since $\Phi$ preserves the half-commutator connection. The
    continuity was obtained in
    Proposition~\ref{proposition:CharactersSymmetriesContinuous},
    \ref{item:SymmetriesContinuous}.
\end{proof}

%
%

\appendix

%
%

\section{Star Products on Cotangent Bundles}
\label{sec:StarProductsCotangentBundles}

%
%

In this short appendix we recall the basic facts on star products on
general cotangent bundles from
\cite{bordemann.neumaier.waldmann:1998a,
  bordemann.neumaier.waldmann:1999a,
  bordemann.neumaier.pflaum.waldmann:2003a, pflaum:1998b,
  pflaum:1998c} to put the construction on the cotangent bundle of a
Lie group into the right perspective.

Let $Q$ be a smooth manifold, the configuration space, and denote its
cotangent bundle by the projection
$\pi \colon T^*Q \longrightarrow Q$.  For the zero section we will
write $\iota\colon Q \longrightarrow T^*Q$. On a cotangent bundle (as
on any vector bundle) we have smooth functions which are polynomial in
the fiber directions. They will be denoted by
$\Pol^\bullet(T^*Q) \subseteq \Cinfty(T^*Q)$, where we write
$\Pol^k(T^*Q)$ for those, which are homogeneous polynomials of degree
$k \in \mathbb{N}_0$. Recall that we always consider complex-valued
functions $\Cinfty(T^*Q)$.

As any vector bundle, $T^*Q$ has a particular vector field, the Euler
vector field $\xi \in \Secinfty(T(T^*Q))$, whose flow is given by
$(t, \alpha_q) \mapsto \E^t \alpha_q$, where $t \in \field{R}$ and
$\alpha_q \in T^*_qQ$ for $q \in Q$. It can be used to characterize
$\Pol^k(T^*Q)$ as the eigenfunctions of the Lie derivative $\Lie_\xi$
to the eigenvalue $k \in \mathbb{N}_0$ and no other eigenvalues
occur. Note that the canonical map
$\mathcal{J}\colon \Secinfty(T_{\field{C}}Q) \longrightarrow
\Pol^1(T^*Q)$, sending a complex vector field
$X \in \Secinfty(T_{\field{C}}Q)$ to the linear function defined by
$(\mathcal{J}(X))(\alpha_q) = \alpha_q(X(q))$, extends to a graded
unital algebra isomorphism
\begin{equation}
    \label{eq:CanonicalIsoJ}
    \mathcal{J}\colon
    \bigoplus_{k=0}^\infty \Secinfty(\Sym^k_{\field{C}} TQ)
    \longrightarrow
    \Pol^\bullet(T^*Q),
\end{equation}
if we set $\mathcal{J}(u) = \pi^*u$ for
$u \in \Cinfty(Q) = \Secinfty(\Sym^0_{\field{C}} TQ)$. Here
$\Sym^k_{\field{C}} TQ$ denotes the $k$-th complexified symmetric
power of the tangent bundle $TQ$.

To establish a global symbol calculus for the algebra of differential
operators $\Diffop(Q)$ acting on $\Cinfty(Q)$, we choose a
torsion-free covariant derivative $\nabla$ on $Q$. We use the same
symbol for all induced covariant derivatives on the various tensor
bundles. The covariant derivative $\nabla$ induces a symmetrized
covariant derivative
\begin{equation}
    \label{eq:SymCovDer}
    \SymD\colon
    \Secinfty(\Sym^k_{\field{C}} T^*Q)
    \longrightarrow
    \Secinfty(\Sym^{k+1}_{\field{C}} T^*Q)
\end{equation}
in such a way that for functions $u \in \Cinfty(Q)$ we have
$\SymD u = \D u$ and for one-forms
$\alpha \in \Secinfty(T^*_{\field{C}}Q)$ we have
\begin{equation}
    \label{eq:SymDDef}
    (\SymD \alpha)(X, Y)
    =
    \nabla_X (\alpha(Y))
    + \nabla_Y (\alpha(X))
    - \alpha (\nabla_X Y)
    - \alpha(\nabla_Y X).
\end{equation}
Then $\SymD$ is defined on higher symmetric forms by requiring a
Leibniz rule with respect to the symmetric tensor product $\vee$,
i.e. we have
$\SymD(\alpha \vee \beta) = \SymD \alpha \vee \beta + \alpha \vee
\SymD \beta$. In local coordinates $(U, x)$ of $Q$ this can then be
written as
\begin{equation}
    \label{eq:SymDLocally}
    \SymD = \D x^i \vee \nabla_{\frac{\partial}{\partial x^i}}.
\end{equation}
In fact, if $\frames{e}_1, \ldots, \frames{e}_n \in \Secinfty(TU)$ is
a local frame of $TQ$ on an open subset $U \subseteq Q$ with dual
local frame $\frames{e}^1, \ldots, \frames{e}^n \in \Secinfty(T^*U)$
then we have
\begin{equation}
    \label{eq:SymDLocalFrame}
    \SymD = \frames{e}^i \vee \nabla_{\frames{e}_i}
\end{equation}
for sections on $U$. This (local) formula will play a crucial role
whenever we have a \emph{global} frame, i.e. on a parallelizable
manifold.

There are now various ways to define the global symbol calculus with
respect to $\nabla$. Following
\cite{bordemann.neumaier.waldmann:1998a} one defines the
standard-ordered quantization map
\begin{equation}
    \label{eq:StandardOrderedQuantization}
    \stdrep\colon
    \Pol^\bullet(T^*Q) \longrightarrow \Diffop(Q)
\end{equation}
by specifying the differential operators $\stdrep(\mathcal{J}(X))$ for
all $X \in \Secinfty(\Sym^k TQ)$ on functions $\psi \in \Cinfty(Q)$ as
\begin{equation}
    \label{eq:stdrepExplicit}
    \stdrep\big(\mathcal{J}(X)\big)\psi
    =
    \iota^*\left(
        \inss(X) \E^{-\I\hbar\SymD}\psi
    \right),
\end{equation}
where
$\iota^*\colon \prod_{k=0}^\infty \Secinfty(\Sym^k_{\field{C}} T^*Q)
\longrightarrow \Cinfty(Q)$ is the projection onto the symmetric
degree $k = 0$ and $\inss(\argument)$ denotes the symmetric insertion
map, which is defined as follows: for a vector field
$X \in \Secinfty(T_{\field{C}}Q)$ it is the insertion into the first
argument as usual. For higher degrees we require
$\inss(X \vee Y) = \inss(X) \inss(Y)$ to get the correct
pre-factors. For a function $X = u \in \Cinfty(Q)$ we set
$\inss(u) = u$ as multiplication operator. Finally, the formal
exponential series of the iterated symmetrized covariant derivatives
of $\psi$ is interpreted as element in the Cartesian product over all
symmetric degrees. Since $\mathcal{J}$ is an isomorphism, this indeed
specifies $\stdrep$ on all polynomial functions $\Pol^\bullet(T^*Q)$
as wanted.

We note that \eqref{eq:StandardOrderedQuantization} is a
$\Cinfty(Q)$-linear isomorphism whenever $\hbar \ne 0$, where
$\Pol^\bullet(T^*Q)$ is equipped with the canonical
$\Cinfty(Q)$-module structure via $\pi^*$ and $\Diffop(Q)$ is
considered as left $\Cinfty(Q)$-module as usual. Moreover, $\stdrep$
is compatible with the filtrations of the differential operator by the
degree of differentiation and the filtration of $\Pol^\bullet(T^*Q)$
induced by the degree of the polynomials. Taking into account the
$\hbar$-dependence gives the \emph{homogeneity}
\begin{equation}
    \label{eq:HomogeneityStdRep}
    \left[\hbar \tfrac{\partial}{\partial \hbar}, \stdrep(f)\right]
    =
    \stdrep\left(
        \mathsf{H} f
    \right)
\end{equation}
for all $f \in \Pol^\bullet(T^*Q)$ possibly depending on $\hbar$ as
well, where
$\mathsf{H} = \hbar\frac{\partial}{\partial \hbar} + \Lie_\xi$. From a
physical point of view this means that $\stdrep$ is dimensionless.

The bijection \eqref{eq:StandardOrderedQuantization} allows us to pull
back the operator product to $\Pol^\bullet(T^*Q)$. This gives an
associative product, the \emph{standard-ordered star product}
$\starstd$, for $\Pol^\bullet(T^*Q)$ such that
\begin{equation}
    \label{eq:starstdDef}
    f \starstd g
    =
    \stdrep^{-1}\big(\stdrep(f)\stdrep(g)\big)
\end{equation}
for $f, g \in \Pol^\bullet(T^*Q)$. The homogeneity properties shows
that for $f, g \in \Pol^\bullet(T^*Q)$ the standard-ordered star
product $f \starstd g$ is a polynomial in $\hbar$ of degree at most
the sum of the degree of $f$ and $g$. More precisely,
\begin{equation}
    \label{eq:Homogeneity}
    \mathsf{H} (f \starstd g)
    =
    \mathsf{H}f \starstd g + f \starstd \mathsf{H}g
\end{equation}
for all $f, g \in \Pol^\bullet(T^*Q)$. Hence we have unique bilinear
operators
$C_r\colon \Pol^\bullet(T^*Q) \times \Pol^\bullet(T^*Q)
\longrightarrow \Pol^\bullet(T^*Q)$ with
\begin{equation}
    \label{eq:StarStdOrderByOrder}
    f \starstd g
    =
    \sum_{r=0}^\infty \hbar^r C_r(f, g),
\end{equation}
where each $C_r$ changes the polynomial degree by $-r$. In particular,
the sum is always finite as long as $f$ and $g$ are polynomial
functions.

It is a not completely obvious fact that the operators $C_r$ in
$\starstd$ are actually bidifferential operators and thus extend to a
formal star product for $\Cinfty(T^*Q)\formal{\hbar}$.  In fact, one
way to show this is to identify $\starstd$ with the Fedosov star
product based on standard-ordering, see
\cite{bordemann.neumaier.waldmann:1998a}. Note, however, that for
functions in $\Pol^\bullet(T^*Q)$ the usual convergence problem of
formal star products is absent since the series
\eqref{eq:StarStdOrderByOrder} terminates after finitely many
contributions.

The standard-ordered symbol calculus has one serious flaw: it lacks
compatibility with the $^*$-involutions. For the differential
operators $\Diffop(Q)$ one has no intrinsic involution. However,
fixing a smooth positive density $\mu \in \Secinfty(\Densities T^*Q)$
one induces an inner product for $\Cinfty_0(Q)$ by
\begin{equation}
    \label{eq:InnerProductCinftyNull}
    \SP{\phi, \psi}_\mu
    =
    \int_Q \cc{\phi} \psi \mu,
\end{equation}
where $\phi, \psi \in \Cinfty_0(Q)$. We fix $\mu$ once and for all to
define the adjoint of a differential operator $D \in \Diffop(Q)$ by
requiring
\begin{equation}
    \label{eq:AdjointDiffop}
    \SP{D^*\phi, \psi}_\mu = \SP{\phi, D\psi}_\mu
\end{equation}
for all $\phi, \psi \in \Cinfty_0(Q)$. A non-trivial global
integration by parts then computes the adjoint $D^*$, explicitly using
the standard-ordered symbol calculus, which we briefly recall:

Firstly, we define the one-form $\alpha \in \Secinfty(T^*Q)$ by
$\nabla_X \mu = \alpha(X) \mu$, thus measuring how $\mu$ is not
covariantly constant with respect to the chosen covariant derivative
$\nabla$. In many cases one can achieve $\alpha = 0$, say for a
Levi-Civita covariant derivative $\nabla$ of a Riemannian metric $g$
and the corresponding Riemannian volume density $\mu_g$.

Secondly, and more importantly, we note that $\nabla$ allows to
horizontally lift tangent vectors $v_q \in T_qQ$ to tangent vectors
$v_q^\hor\at{\alpha_q} \in T_{\alpha_q} T^*Q$. Canonically, we can
lift one-forms $\beta_q \in T^*_qQ$ vertically to tangent vectors
$\beta_q^\ver\at{\alpha_q} \in T_{\alpha_q} T^*Q$. This gives a
splitting $T_{\alpha_q} T^*Q = \Hor_{\alpha_q} \oplus \Ver_{\alpha_q}$
for all $\alpha_q \in T^*_qQ$ with the additional property, specific
for a cotangent bundle, that the horizontal and the vertical space are
equipped with a natural pairing originating from the pairing of $T_qQ$
and $T^*_qQ$. Thus we obtain a pseudo Riemannian metric $g_0$ on
$T^*Q$ of split signature $(n, n)$.  This metric has a Laplace
operator $\Laplace_0 \in \Diffop^2(T^*Q)$ for functions on $T^*Q$,
which locally is given by
\begin{equation}
    \label{eq:LaplaceTstarQ}
    \Laplace_0
    =
    \frac{\partial^2}{\partial q^i \partial p_i}
    +
    p_r \pi^*(\Gamma^r_{ij})
    \frac{\partial^2}{\partial p_i \partial p_j}
    +
    \pi^*(\Gamma^i_{ij}) \frac{\partial}{\partial p_j},
\end{equation}
where $(T^*U, (q, p))$ is a Darboux chart induced by a local chart
$(U, x)$ on $Q$ and where $\Gamma^r_{ij}$ are the Christoffel symbols
of $\nabla$ with respect to the chart $(U, x)$.

Putting things together we can then consider the Neumaier operator
\begin{equation}
    \label{eq:NeumaierOperatorTstarQ}
    \mathcal{N}
    =
    \exp\big(
    - \tfrac{\I\hbar}{2} (\Laplace_0 + \Lie_{\alpha^\ver})
    \big),
\end{equation}
which is a well-defined endomorphism of $\Pol^\bullet(T^*Q)$, since
both $\Laplace_0$ and the Lie derivative in direction of the vertical
lift of $\alpha$ decrease the polynomial degree by one, thus making
the exponential series terminate on polynomial functions. Using
$\mathcal{N}$ one can write the integration by parts to compute the
adjoint of a differential operator as
\begin{equation}
    \label{eq:AdjointStandardOrdered}
    \stdrep(f)^* = \stdrep(\mathcal{N}^2 \cc{f})
\end{equation}
for all $f \in \Pol^\bullet(T^*Q)$, see
\cite{bordemann.neumaier.waldmann:1998a,
  bordemann.neumaier.waldmann:1999a}. Since
\eqref{eq:StandardOrderedQuantization} is an isomorphism, this
computes the adjoint of all differential operators explicitly, once we
base their description on the standard-ordered symbol calculus
$\stdrep$.

One can then use $\mathcal{N}$ to pass from the standard-ordering to a
Weyl ordering and, more generally, to a $\kappa$-ordering
interpolating between the two. For $\kappa \in \field{R}$ one defines
a new ordering
\begin{equation}
    \label{eq:krepDef}
    \krep(f) = \stdrep(\mathcal{N}_\kappa f)
    \quad
    \textrm{where}
    \quad
    \mathcal{N}_\kappa
    =
    \exp(- \I \hbar \kappa (\Laplace_0 + \Lie_{\alpha^\ver}))
\end{equation}
together with a corresponding $\kappa$-ordered star product
\begin{equation}
    \label{eq:starkDef}
    f \stark g
    =
    \mathcal{N}_\kappa^{-1}
    \big(
        \mathcal{N}_\kappa(f) \starstd \mathcal{N}_\kappa (g)
    \big)
\end{equation}
for $f, g \in \Pol^\bullet(T^*Q)$, see \cite{pflaum:1998c,
  pflaum:1998b, bordemann.neumaier.pflaum.waldmann:2003a}. The case
$\kappa = \frac{1}{2}$ is then called the Weyl ordering $\weylrep$
with the corresponding Weyl star product $\starweyl$. For the Weyl
star product one has
\begin{equation}
    \label{eq:WeylHasInvolution}
    \weylrep(f)^* = \weylrep(\cc{f})
    \quad
    \textrm{and}
    \quad
    \cc{f \starweyl g} = \cc{g} \starweyl \cc{f}
\end{equation}
for all $f, g \in \Pol^\bullet(T^*Q)$.

Finally, we note the useful relation
\begin{equation}
    \label{eq:RepFromStar}
    \stdrep(f)\phi = \iota^*(f \starstd \pi^*\phi)
\end{equation}
for all $f \in \Pol(T^*Q)$ and $\phi \in \Cinfty(Q)$. This allows to
reconstruct the standard-ordered representation from the
standard-ordered star product.

%
%

\section{Noncommutative Higher Leibniz Rule}
\label{sec:Leibniz}%

The following well-known Leibniz rules are completely algebraic,
wherefore we treat them as such.
\begin{proposition}
    \label{proposition:HigherLeibniz}%
    Let $\ring{R}$ be a (not necessarily associative) ring with
    $\field{Q} \subseteq \ring{R}$,
    $D_1, \ldots, D_n \in \Der(\ring{R})$ derivations and
    $a, b \in \ring{R}$.
    \begin{propositionlist}
    \item \label{item:HigherLeibniz}%
        We have the higher Leibniz rule
        \begin{equation}
            \label{eq:HigherLeibniz}
            D_n \cdots D_1 (a b)
            =
            \sum_{p=0}^{n}
            \sum_{\sigma \in \Sh(p,n-p)}
            \big( D_{\sigma(n)} \cdots D_{\sigma(p+1)} a \big)
            \big( D_{\sigma(p)} \cdots D_{\sigma(1)} b \big).
        \end{equation}
        Here $\Sh(p, n-p)$ denotes the set of $(p, n-p)$-shuffles,
        i.e. permutations $\sigma \in S_n$ such that
        \begin{equation}
            \label{eq:Shuffle}
            \sigma(1)
            <
            \sigma(2)
            <
            \cdots
            <
            \sigma(p)
            \quad \textrm{and} \quad
            \sigma(p+1)
            <
            \sigma(p+2)
            <
            \cdots
            <
            \sigma(n).
        \end{equation}
    \item \label{item:HigherLeibnizSymmetrized}%
        Symmetrizing, it furthermore holds that
        \begin{equation}
            \label{eq:HigherLeibnizSymmetrized}
            \sum_{\sigma \in S_n}
            D_{\sigma(n)} \cdots D_{\sigma(1)}
            (a b)
            =
            \sum_{\sigma \in S_n}
            \sum_{p=0}^{n}
            \binom{n}{p}
            \big( D_{\sigma(n)} \cdots D_{\sigma(p+1)} a \big)
            \big( D_{\sigma(p)} \cdots D_{\sigma(1)} b \big).
        \end{equation}
    \end{propositionlist}
\end{proposition}
\begin{proof}
    Part~\ref{item:HigherLeibniz} is a straightforward induction. Use
    that $\sigma \in \Sh(p, n-p)$ satisfies either $\sigma(p) = n$ or
    $\sigma(n) = n$ by \eqref{eq:Shuffle}. The statement
    \ref{item:HigherLeibnizSymmetrized} is an easy consequence of
    $\abs[\big]{\Sh(p, n-p)} = \binom{n}{p}$.
\end{proof}

%
%

{
  \footnotesize

}

%
%


%
%

\end{document}